\documentclass[12pt, reqno]{amsart}
    \usepackage{amsmath, amsthm, amssymb, stmaryrd}
    \usepackage[normalem]{ulem}
    \usepackage{enumerate}
    \usepackage{url}
    \usepackage{comment}
    \usepackage[centering, margin={1in, 0.5in}, includeheadfoot]{geometry}
    \usepackage{tikz}
    \usetikzlibrary{arrows}
    \usepackage{pgfplots}
    \usepackage{color}
    \usepackage{xcolor}
    \definecolor{myred}{rgb}{0.2,0,0}
    \definecolor{myblue}{rgb}{0,0,0.6}
    \definecolor{mygreen}{rgb}{0,0.2,0}
    \usepackage[colorlinks=true,linkcolor=myred,citecolor=myblue,urlcolor=mygreen]{hyperref}
    \usepackage{enumitem}
    \usepackage{bm}
    \usepackage{bbm}

    \theoremstyle{plain}
    \newtheorem{thm}{Theorem}[section]
    \newtheorem{lemma}[thm]{Lemma}
    \newtheorem{prop}[thm]{Proposition}
    \newtheorem{cor}[thm]{Corollary}
    
    \newtheorem*{conjecture*}{Conjecture}
    
    \newtheorem{prob}{Problem}
    \newtheorem*{claim*}{Claim}
    \newtheorem{defn}[thm]{Definition}

    \theoremstyle{remark}
    
    \newtheorem*{remark*}{Remark}

    \numberwithin{equation}{section}
    \newcommand{\supp}{{\mathrm{supp}\,}}
    \newcommand{\hidden}[1]{}
    \newcommand{\floor}[1]{\left\lfloor #1 \right\rfloor}
    \newcommand{\ceil}[1]{\left\lceil #1 \right\rceil}

    \newcommand{\norm}[1]{\left\| #1 \right\|}

    \newcommand{\Id}{\mathbbm{1}}

    \def\alp{{\alpha}} 
    \def\bet{{\beta}}  
    \def\gam{{\gamma}}

    \def\lam{{\lambda}}

    \def\sig{{\sigma}}

    \def\eps{\varepsilon}
    \def\implies{\Rightarrow}
    \def\l{\ell}
    \def\le{\leqslant} \def\ge{\geqslant}
    \def\ls{\leqslant} \def\gs{\geqslant}
    \def\leq{\leqslant}
    \def\geq{\geqslant}

    \def \sig{{\sigma}}

    \def \bE {\mathbb E}
    
    \def \bG {\mathbb G}
    
    \def \bI {\mathbb I}

    \def \bN {\mathbb N}
    \def \bP {\mathbb P}
    \def \bQ {\mathbb Q}
    \def \bR {\mathbb R}
    \def \bZ {\mathbb Z}

    \def \N {\mathbb N}

    \def \cI {\mathcal I}
    
    \def \cK {\mathcal K}
    
    \def \cM {\mathcal M}

    \def \cP {\mathcal P}

    \def \cT {\mathcal T}

    \def \dim {\mathrm{dim}}

    \def \deg {\mathrm{deg}}

    \def \Bad {{\mathrm{\mathbf{Bad}}}}
    
    \def \dimH {\dim_{\mathrm{H}}}

    \def \NN {(N_n)_{n\in\bN}}

    \newcommand{\mfrac}[2]{\raisebox{0.15ex}{\scalebox{1.1}{$\frac{#1}{#2}$}}}

\begin{document}
    \title[Circle coverings driven by arithmetic sequences]{Circle coverings driven by arithmetic sequences: a percolation approach to Diophantine approximation and fractal intersections}
    \subjclass[2010]{}

\author[Hauke]{Manuel Hauke}
\address{Institute of Analysis and Number Theory, TU Graz, Austria}
\email{hauke@math.tugraz.at}

\author[Shubin]{Andrei Shubin}
\address{Institute of Analysis and Number Theory, TU Graz, Austria}
\email{shubin@math.tugraz.at}

\author[Stefanescu]{Eduard Stefanescu}
\address{Institute of Analysis and Number Theory, TU Graz, Austria}
\email{eduard.stefanescu@tugraz.at}

\author[Zafeiropoulos]{Agamemnon Zafeiropoulos}
\address{Institute of Analysis and Number Theory, TU Graz, Austria}
\email{zafeiropoulos@math.tugraz.at}
    
    % \author{Manuel Hauke \and Andrei Shubin \and Eduard Stefanescu \and  Agamemnon Zafeiropoulos}
    % \address{Institute of Analysis and Number Theory, TU Graz, Austria}
    % \email{hauke@math.tugraz.at}
    % \address{Institute of Analysis and Number Theory, TU Graz, Austria}
    % \email{shubin@math.tugraz.at}
    % \address{Institute of Analysis and Number Theory, TU Graz, Austria}
    % \email{eduard.stefanescu@tugraz.at}
    % \address{Institute of Analysis and Number Theory, TU Graz, Austria}
    % \email{zafeiropoulos@math.tugraz.at}

    \begin{abstract}  
        
    We study problems on covering $[0,1)$ by shrinking intervals centered at the points $\{q_n x\}$, where $(q_n)_{n\in \N}$ is a given real-valued sequence and $x \in [0,1)$ is random.
    
    First, for real-valued lacunary sequences $(q_n)_{n\in\bN}$, we show that the covering radius $\mfrac{1}{n}$ is sharp up to a constant: there exist $C>c>0$ such that, for Lebesgue-a.e.\ $x$, the intervals of length $\mfrac{C}{n}$ cover $[0,1)$ infinitely often, while this fails for intervals  of length $\mfrac{c}{n}$.   Moreover, the lower bound holds for certain sub-lacunary rates and the results partially extend to all probability measures with sufficiently fast Fourier decay. 
    
    As an application, we obtain a new bound for a variant of the inhomogeneous Littlewood--Cassels problem: for any $\alpha\in\Bad$ and $\gamma\in\bR$, there exists a set of $\beta\in\Bad$ of full Hausdorff dimension such that
    \[
        \norm{n\alpha-\gamma}\,\norm{n\beta-\delta}<\frac{C}{n\log n} \qquad \text{for infinitely many } n\gs 1,
    \] uniformly in $\delta\in\bR$. This improves upon previous works of Haynes--Jensen--Kristensen, Chow--Technau, and the third author, and is best possible when one restricts to best approximations of the first factor.

    Second, under certain arithmetic restrictions on $(q_n)_{n\in\N}$, we compute the almost-sure Hausdorff dimension of limsup sets generated by intervals of size $\mfrac{1}{n^{\nu}}$ for $\nu \ge 1$, centered at $\{q_n x\}$, and intersected with Ahlfors regular compact sets such as the middle-third Cantor set. In particular, our results apply to all real-valued lacunary sequences, to integer-valued polynomials, and to powers of primes. This substantially extends the work of Bugeaud and Durand, which applies only to certain super-lacunary integer-valued sequences.

    The proofs are based on a variant of a random coloring of a binary tree, combined with moments of exponential sums.

    \end{abstract}
    
    \maketitle

    %SECTION 1: Introduction
    
    %%%%%%%%%%%%%%%%%%%%%%%%%%%%%%%%%%%%%%%%%%%%%%%%%%%%%%%%%%%%
    
    \section{Introduction} \label{sec_intro}
    
    %\raggedbottom
    
    \subsection*{Random covering}
    
 In 1956, Dvoretzky~\cite{dvoretzky} raised a question about the covering of the circle $\bR/\bZ$ by shrinking arcs with random centers. Precisely, let $(\omega_n)_{n\in\bN}$ be i.i.d.~random points chosen uniformly from $\bR/\bZ$, and let $(\ell_n)_{n\in\bN}$ be a non-increasing sequence of lengths $1\ge \ell_1\ge \ell_2\ge \cdots$. Does there exist a necessary and sufficient condition on $(\ell_n)_{n\in\N}$ such that every point $\delta\in\bR/\bZ$ belongs to the shrinking arcs $I_n := (\omega_n-\mfrac{\ell_n}{2},\, \omega_n+\mfrac{\ell_n}{2})$ for infinitely\footnote{In the literature, Dvoretzky's problem is sometimes formulated as a \textit{one-time} covering problem: covering holds if, almost surely, every $\delta\in\bR/\bZ$ satisfies $\delta\in I_n$ for \textit{at least one $n$}, rather than for \textit{infinitely many~$n$}. However, when $\ell_1<1$, one implies the other: if $\bR/\bZ$ is covered almost surely, then it is covered almost surely \textit{infinitely often}~\cite[Remark~1]{shepp}. In this work, we are concerned only with infinite covering.} many $n$ almost surely?

Observe that the same question becomes easier when $\delta\in\bR/\bZ$ is \textit{fixed}: in this case, the condition
\begin{equation}\label{necessary_cond}
    \sum_{n=1}^{\infty}\ell_n=\infty
\end{equation} is both necessary and sufficient by the first and second Borel--Cantelli lemmas. Moreover, combining this with Fubini's theorem, one readily deduces that~\eqref{necessary_cond} is also necessary and sufficient for covering \textit{almost all} $\delta\in\bR/\bZ$ with probability one.
In particular, an infinite almost-covering holds for $\ell_n=\mfrac{1}{n\log n}$ immediately. But the problem of covering \textit{all}~$\delta$ (Dvoretzky's problem) is significantly harder.

Partial progress towards Dvoretzky's problem was made by L\'evy, Kahane, Erd\H{o}s, Billard, Orey, and Mandelbrot~\cite{billard, kahane68, orey, mandelbrot1, mandelbrot2, erdos_problems}. In particular, their results already showed that $\ell_n=\mfrac{1}{n}$ is a covering case, whereas $\ell_n=\mfrac{1-\eps}{n}$ is not, in stark contrast to classical Khintchine-type theorems where a constant invariance principle holds. The problem was fully solved in 1972 by Shepp~\cite{shepp}, who proved that the necessary and sufficient condition for covering is
\[
    \sum_{n=1}^{\infty}\frac{1}{n^2}\exp\big(\ell_1+\cdots+\ell_n\big)=\infty.
\] For example, Shepp's criterion implies that $\ell_n=\mfrac{1}{n}-\mfrac{1}{n\log n}$ is a covering case, while $\ell_n=\mfrac{1}{n}-\mfrac{1}{n(\log n)^{1-\eps}}$ is not. Shepp's work has since been generalized in several directions, including different limsup sets, metric spaces, and various probability distributions. For further references, see~\cite{kahane, FK, FJJS}.

Dirichlet's approximation theorem implies that for every $\delta\in[0,1)$ there exist infinitely many reduced rationals $\mfrac{a}{q}$ such that $\|\delta-\mfrac{a}{q}\|\ls \mfrac{1}{q^2}$ (here $\|x\|$ denotes the distance from $x\in\bR$ to the nearest integer). This may be viewed as a covering statement with centers given by the Farey fractions $\omega_{a,q}=\mfrac{a}{q}$ and interval lengths $\ell_{a,q}=\mfrac{2}{q^2}$. One can order the Farey fractions so that the associated sequence $(\ell_{a,q})$ is non-increasing. It is then straightforward to verify that if $(\ell_n)_{n\in\bN}$ is any such re-ordering, one has
\[
    \sum_{n=1}^{\infty} \frac{1}{n^2} \exp \Big( \sum_{k \le n} \ell_{n} \Big) < \infty.
\] Thus, this deterministic sequence (viewed as a particular realization of the random centers) does not satisfy Shepp's criterion. Nevertheless, covering holds, since the Farey fractions have a highly rigid structure compared with a generic sequence. Random covering may therefore be seen as a random analogue of Dirichlet's theorem. We also mention that a more precise random analogue was studied recently in~\cite{KLP}.

\subsection*{Metric covering} In this paper we study an analogue of Dvoretzky's covering problem for intervals centered at points $\{q_n x\}$, where $(q_n)_{n\in\N}$ is a given real-valued sequence and~$x$ is chosen randomly with respect to a given probability measure, e.g. the Lebesgue measure on $[0,1)$. Here $\{x\} := x-\floor{x}$ denotes the fractional part of $x$.

We first briefly analyze the situation when $x$ is \textit{fixed}. The \textit{Kronecker sequence} $\{nx\}$ is a well-known example of a highly rigid sequence in $[0,1)$. Interpreting the points $\{nx\}$ as centers of shrinking intervals, covering clearly fails for any $\ell_n\to 0$ if $x\in\bQ$. If $x\notin\bQ$, the theorem of Khintchine~\cite[Theorem~10.2]{Hua} gives for any $\eps > 0$
\[
    \norm{nx-\delta}\le \frac{1+\eps}{\sqrt{5} n} \qquad \text{for infinitely many } n\gs 1
\] uniformly in $\delta\in[0,1)$. Note that this covering length is also below the threshold in Shepp's criterion.
The highly rigid structure of the points $\{x\},\{2x\},\ldots,\{Nx\}$ can also be seen, for instance, via the three-gap theorem. 

It seems that the covering problem becomes harder when one considers sparser sequences of integers, such as squares or higher powers, or non-integer sequences, where a more random behavior is expected.

\begin{prob}
Under what conditions on the lengths $(\ell_n)_{n\in\N}$ and on $x\in[0,1)$ does covering hold with the interval centers $\omega_n=\{n^2 x\}$? In particular, for which $x$ does there exist $C>0$ such that covering holds with $\ell_n=\mfrac{C}{n}$? What about $\omega_n=\{n^\theta x\}$ with other values of $\theta>0$?
\end{prob}

In contrast, in the metric setting (e.g. for \textit{Lebesgue-generic $x$}) the covering problem with centers $\{q_n x\}$ is expected to become easier the faster the sequence $(q_n)_{n\in\bN}$ grows. While the centers are no longer independent as in Dvoretzky's problem --- for instance, the location of $\{q_n x\}$ influences the location of $\{q_{n+1} x\}$ --- this dependence weakens as $q_n$ grows faster, and the question typically gets easier.

Recall that a real-valued sequence $(q_n)_{n\in\N}$ is called \textit{lacunary} if it satisfies the Hadamard gap condition 
\[
    \frac{q_{n+1}}{q_n}\ge r \qquad \text{ for all }  n\gs 1,
\] for some fixed $r>1$. Lacunary dilates $\{q_n x\}$ with Lebesgue-uniform $x$ are known to behave similarly to i.i.d.~points in many regards.  For a comprehensive overview of their random statistics, see~\cite{ABT}. In fact, the resemblance of the points $\{q_nx\}$ with i.i.d. random variables on $[0,1)$ has been observed when $x$ is chosen randomly with respect to some measure $\mu$ such that its Fourier transform \[\widehat{\mu}(t):=\int_0^1 e(tu)\,d\mu(u), \qquad t\in \bR \] exhibits sufficiently fast decay, see e.g. \cite{M0,tz}. Our first result shows that the covering property holds for any lacunary sequence at the correct order of interval lengths for all measures with polynomial Fourier decay.

    \begin{thm} \label{thm1}
    
    Let $(q_n)_{n\in\N}$ be a real-valued lacunary sequence. Suppose that $\mu$ is a probability measure on $[0,1)$ satisfying
    \begin{equation} \label{rajchman_cond}
        \widehat{\mu}(t) \ll \frac{1}{(1+|t|)^{\eta}},
        \qquad |t|\to\infty
    \end{equation} for some $\eta>0$. Then there exists a constant $C>0$, depending on $(q_n)_{n\in\N}$, such that for $\mu$-almost all $x\in[0,1)$ one has
    \[
        \norm{q_n x-\delta}<\frac{C}{n} \qquad \text{for infinitely many } n\gs 1,
    \] uniformly in $\delta\in[0,1)$.
    \end{thm}
    
    \begin{remark*}
    The constant $C$ in Theorem \ref{thm1} can be made explicit. As can be seen from the proof, when $\mfrac{q_{n+1}}{q_n} \ge r$ one may take any 
    \[
        C \ge 4000\cdot\Big(\Bigl\lfloor\frac{\ln 10}{\ln r}\Bigr\rfloor+1\Big).
    \] In particular, this lower bound increases as $r\to 1$.
    \end{remark*}

    We should further note that the metric consideration in Theorem~\ref{thm1} is, in a sense, necessary, since the analogous statement does not hold \textit{for all} $x\in[0,1)$. Indeed, for any lacunary sequence $(q_n)_{n\in\bN}$, the set of $x$ for which $\{q_n x\}$ is not dense in $[0,1)$ is winning in the sense of Schmidt, and therefore has full Hausdorff dimension~\cite{schmidt_game, mathan, poll_lac}.

    To the best of our knowledge, most existing results in the Diophantine approximation and dynamics literature either allow a small exceptional set of points $\delta$ that are not covered (see e.g. \cite{chow25,hr_twist,kp25}), or, conversely, study the size of the covered set when it is thin (for instance, determine the Hausdorff dimension in both regimes). Apart from the Kronecker sequence $\{nx\}$, where the question is easy, another noteworthy example in which full covering was established is the work of Fan, Schmeling, and Troubetzkoy~\cite{FST}, who considered the sequence $q_n=2^n$ and various Gibbs measures. However, even in the Lebesgue case their result yields covering only for interval lengths of the form $\ell_n=\mfrac{1}{n^{1-\eps}}$ for arbitrary $\eps>0$, and their methods do not seem to extend to general lacunary sequences.

    Our next result establishes a matching lower bound for covering. We show that for a sufficiently small $0<c<1$ there exists an exceptional $\delta\in[0,1)$ that is covered at most finitely often, even for slower-growing sequences:
    
    \begin{thm}\label{thm2}
        Let $\eps>0$ be arbitrarily small, and let $(q_n)_{n\in\N}$ be a real-valued sequence satisfying, for all sufficiently large $n\in\N$, the gap condition
        \begin{equation} \label{gap_condition}
            \frac{q_{n+1}}{q_n} > 1+\frac{1}{n^{1-\eps}}.
        \end{equation} Then there exists $c_{\eps}>0$ such that for every $0<c\le c_{\eps}$, for Lebesgue-a.e.\ $x\in[0,1)$ there exists $\delta_0 \in [0, 1)$ such that
        \[
            \norm{q_n x-\delta_0}<\frac{c}{n}\qquad \text{for at most finitely many } n\gs 1.
        \]
    \end{thm}
    
    \begin{remark*}
     We note that the approach used to prove the lower bound is substantially different from that for the upper bound. Here the result applies only to the Lebesgue measure, but the growth-rate requirement can be relaxed. In fact, we note that the condition~\eqref{gap_condition} is less classical than the Hadamard gap condition, but still appeared in the literature many times (see, for example, \cite{Erdos_gap_2, BPT}). From the proof one can deduce the same statement with Lebesgue measure replaced by any probability measure $\mu$ satisfying~\eqref{rajchman_cond} additionally assuming the following technical condition: for every sufficiently large $L>0$, and all sufficiently large $n\in\N$, one has $q_{n+\Delta_n}^{\eta} > q_n, \
        \Delta_n= \lfloor n^{\,1-\frac{1}{L}} \rfloor$. This assumption is rather restrictive and, in general, does not hold even for lacunary sequences. 
    \end{remark*}

    \subsection*{Littlewood--Cassels problem}
    
       \emph{Littlewood's conjecture} is a famous open problem, dating back to the 1920s, which asserts that for any $\alp,\bet\in\mathbb{R}$ one has
    \begin{equation}\label{lc}
        \liminf_{n\to\infty} n\,\|n\alp\|\,\|n\bet\|=0.
    \end{equation} This conjecture has deep connections to homogeneous dynamics and, in particular, to measure rigidity~\cite{EKL}. From a metric point of view, \eqref{lc} even holds with a speed of convergence: Gallagher~\cite{Gallagher} showed that for Lebesgue-a.e.\ $(\alp,\bet)\in[0,1]^2$,
\[
    \liminf_{n\to\infty} n(\log n)^2\,\|n\alp\|\,\|n\bet\|=0.
\]
Note that any potential counterexample to~\eqref{lc} must have both $\alp$ and $\bet$ in
\[
    \Bad:=\big\{x\in[0,1): \inf_{n\ge 1} n\|nx\|>0\big\},
\]
the set of \emph{badly approximable} numbers, which is known to have Lebesgue measure zero. Thus, Gallagher's result provides no information when $\alp,\bet\in\Bad$, and it is therefore natural to study~\eqref{lc} on subsets of $\Bad$.

The first such result was established in the seminal work of Pollington and Velani~\cite{PV}, who proved that, given $\alpha\in\Bad$, there exists a set $\bG=\bG(\alpha)\subseteq\Bad$ with $\dimH \bG(\alpha)=1$ such that for every $\beta\in\bG(\alpha)$ one has
\begin{equation}\label{pv}
    n\,\|n\alp\|\,\|n\bet\| \ls \frac{1}{\log n}
    \qquad \text{for infinitely many } n\ge 1.
\end{equation}
The size of the exceptional set was improved in the groundbreaking work of Einsiedler, Katok, and Lindenstrauss~\cite{EKL}, who showed that~\eqref{lc} holds except for a set of pairs $(\alpha,\beta)$ of Hausdorff dimension zero.

Another active direction of research concerns an inhomogeneous generalization of Littlewood's conjecture, which asserts that for any $\alpha,\beta,\gamma,\delta\in\mathbb{R}$ one has
\begin{equation}\label{inhom_LC}
    \liminf_{n \to \infty} n \lVert n\alpha - \gamma\rVert \lVert n\beta - \delta\rVert = 0,
\end{equation}
unless there is a trivial obstruction (such as $\alpha,\beta\in\mathbb{Q}$ while $\gamma,\delta\notin\mathbb{Q}$, or more generally, $\alpha,\beta,1$ being linearly dependent, see the work of Moshchevitin \cite{mosh_little} on this and related questions). For example, when $\gamma$ and $\delta$ are \textit{fixed}, an analogue of Gallagher's result was obtained in~\cite{CT_memo}.

The situation is significantly different when one asks for~\eqref{inhom_LC} to hold \textit{uniformly in} $\gamma\in\bR$ or $\delta\in\bR$, or uniformly in both $(\gamma,\delta)\in\bR^2$. In the breakthrough work~\cite{shap_ann}, Shapira proved that for Lebesgue almost all $(\alpha,\beta)\in[0,1]^2$ the relation~\eqref{inhom_LC} holds uniformly in $(\gamma,\delta)\in\bR^2$, thereby confirming a conjecture of Cassels~\cite{cass_geom}. A quantitative improvement, with an additional factor of $(\log^{(5)}n)^{\varepsilon}$, was given by Gorodnik and Vishe~\cite{GV}.

    In this part of the paper, we are interested in a hybrid version of the problems of Littlewood and Cassels, where in~\eqref{inhom_LC} $\gamma$ is \textit{fixed} and $\delta$ is \textit{uniform} in $\bR$, which was first studied by Haynes, Jensen, and Kristensen~\cite{haynes} and can be stated as follows:

    \begin{prob}\label{problem1}
    For fixed $\alpha$ satisfying some Diophantine condition and fixed $\gamma\in\bR$, determine the fastest rate $\psi(n)\to 0$ for which there exists a set $\mathbb{G}(\alpha,\gamma) \subseteq \Bad$ of Hausdorff dimension~1 such that, for every $\beta\in\mathbb{G}(\alpha,\gamma)$ and every $\delta\in\bR$, one has
    \begin{equation}\label{our_LC}
        n\|n\alp - \gamma\|\,\|n\bet - \delta\| \leq \psi(n) \text{ for infinitely many } n \geq 1.
    \end{equation}
    \end{prob}

We note that the above question can be considered for any set $\bG(\alpha, \gamma)$ (not necessarily being a subset of $\Bad$) with $\mu(\bG)=1$ for any measure $\mu$ satisfying~\eqref{rajchman_cond} and all proofs in this article, as well as preceding works on Problem \ref{problem1}, generalize immediately. In particular, Pollington and Velani used Kaufman's measures~\cite{Kaufman} supported within $\Bad$ that allows to go immediately from $\mu(\bG)=1$ with~\eqref{rajchman_cond} to $\mathbb{G}(\alpha,\gamma) \subseteq \Bad$ with full Hausdorff dimension. While obviously the Lebesgue measure satisfies~\eqref{rajchman_cond}, the methods are explicitly aiming to say something about the generic behaviour within Lebesgue null sets.

     Motivated by the work of Pollington and Velani, Haynes, Jensen, and Kristensen~\cite{haynes} considered the above question in the case where $\alpha\in\Bad$. They proved that $\psi(n)=\mfrac{1}{(\log n)^{1/2-\eps}}$ is admissible when $\gamma=0$.
     Refinements of this result were obtained in~\cite{tz, cz}. For instance, it was extended to an arbitrary fixed $\gamma\in\bR$; the restriction $\alpha\in\Bad$ was relaxed to the Lebesgue full set
\begin{equation}\label{Kdef}
    \cK := \Big\{\alpha \in[0,1]: \sup_{k\ge 1}\frac{\log q_k(\alpha)}{k}<\infty\Big\},
\end{equation}
where $q_k(\alpha)$ denotes the denominator of the $k$-th convergent of $\alpha$ (note that $\Bad\subset \cK$); and the factor $(\log n)^{\eps}$ from~\cite{haynes} was replaced by $(\log\log\log n)^{1/2+\eps}$. All results in~\cite{cz,haynes,tz} rely on metric discrepancy estimates for lacunary sequences $\{q_k x\}$.

   A significant improvement was obtained recently by Chow and Technau~\cite{ct}, who showed that~\eqref{our_LC} holds with $\psi(n)=\mfrac{(\log\log n)^{3+\varepsilon}}{\log n}$, using dispersion instead of discrepancy. Subsequently, the exponent $3$ was reduced to $2$ by the third author~\cite{eduard}. \\

    As an application of our infinite covering result, we obtain a further improvement. The bound in Theorem~\ref{thm1} will allow us to remove the $\log\log n$-factors entirely:

    \begin{cor} \label{cor1}
        Let $\alp \in \cK$ with $\cK$ as in \eqref{Kdef} and $\gamma \in \bR$. Then there exist $C > 0$ and a set $\bG = \bG(\alp,\gam) \subseteq \Bad$ with $\dimH\bG=1$ such that for every $\beta \in \bG$ one has for all $\delta \in \bR$
    \begin{equation} \label{improved}
        n \|n\alp -\gamma\| \|n\beta-\delta\| \ls \frac{C}{\log n} \qquad \text{ for infinitely many } n\gs 1.
    \end{equation}
    \end{cor} 
     Here we shortly remark where the improvement upon the preceding works comes from. Note that the discrepancy in~\cite{cz,haynes, tz} and the dispersion in~\cite{ct, eduard} play essentially the same role. One first chooses a sequence $(q_k)_{k\in\bN}$ minimizing the first factor, so that $\norm{q_k\alpha-\gamma}\le C q_k^{-1}$ for all $k\in\bN$. Under~\eqref{Kdef}, this sequence can be chosen lacunary. One then seeks the best rate $\psi$ such that, for $\mu$-almost all $\beta\in[0,1)$, the second factor satisfies, uniformly in $\delta\in\bR$,
    \begin{equation}\label{one_factor}
        \norm{q_k\beta-\delta}\le \psi(q_k) \qquad \text{for infinitely many } k,
     \end{equation}
     where $\mu$ is a Kaufman measure~\cite{Kaufman} supported on $\Bad$. The results in~\cite{haynes,tz,cz} provide discrepancy bounds for lacunary sequences $\{q_k x\}$ of order $\lessapprox k^{-1/2}$, as in the i.i.d.~case. This leads to the rate $\psi(n)\approx (\log n)^{-1/2}$.

    The main idea of Chow and Technau~\cite{ct} was to estimate directly the maximal gap between consecutive points of the sequence $\{q_k x\}$ for $\mu$-almost every $x$. They obtained an upper bound of the form $\mfrac{(\log k)^{3+\eps}}{k}$, which was improved more recently to $\mfrac{(\log k)^{2+\eps}}{k}$ by the third author~\cite{eduard} (see also~\cite{eduard2}, where $\eps$ was removed and a higher-dimensional analogue was established). On the other hand, for the maximal gap one expects the almost sure bound of the order $\approx \mfrac{\log k}{k}$ from the result of Devroye~\cite{Devroye} for the i.i.d.~case. Therefore, the rate $\psi(n)=\mfrac{\log\log n}{\log n}$ is the best one can hope for using the dispersion approach.

We note that both discrepancy and dispersion approaches yield something \textit{stronger} than what is required in this problem. Bounds on the discrepancy and on the maximal gap imply that any interval in $[0,1)$ of length $\ge (\psi(q_k))^{-1}$ contains at least one of the points $\{q_1\beta\},\ldots,\{q_k\beta\}$, and hence every $\delta\in[0,1)$ lies at distance at most $(\psi(q_k))^{-1}$ from one of them. Consequently,~\eqref{one_factor} holds \textit{for all sufficiently large $k$}, whereas it is only required to hold for \textit{infinitely many} $k$. Therefore, one can seek for the smallest size of intervals around each $\delta$ containing a point $\{ q_k \beta\}$ only infinitely often (whereas intervals of the maximal gap size would contain $\{q_k \beta\}$ eventually always). This is exactly what is captured by the interval length in the infinite covering problem.

As in many questions in metric diophantine approximation, the difference between \textit{non-uniform} (i.e. infinitely often) and \textit{uniform} (i.e. for all but finitely many) approximation is a factor of $\log$ -- which in Problem \ref{problem1} arises via \eqref{one_factor} as $\log \log n$.\\

While our result is sharp in the sense of a one-factor optimization problem \eqref{one_factor}, one may hope to improve the rate in Problem~\ref{problem1} further by switching to a two-factor optimization problem. This leads to studying the distribution of $\{q_k x\}$ for several sequences $(q_k)$ simultaneously, which seems to be more challenging. In order to get some heuristics what the best possible rate could be, we consider here a randomized analogue of Problem~\ref{problem1}, in which the second factor $\norm{n\beta-\delta}$ is replaced by $\norm{X_n-\delta}$ for i.i.d.\ uniformly distributed points $X_n$.

    \begin{thm} \label{thm3}
        Let $(X_n)_{n \in \bN}$ be a sequence of i.i.d.~Lebesgue-uniform points from $[0, 1)$ and let $\bP$ be the corresponding probability measure.
        There exists an effective constant $0<C<\infty$ such that, for every $\alpha\in\Bad$, every $\gamma\in\bR$, and every monotonically decreasing function $\psi:\N\to[0,\infty)$ satisfying
            \begin{equation}\label{fake_khintchine}
                \limsup_{N \to \infty} \sum_{N < n \leq 2^N} \psi(n) > C,
            \end{equation}
        one has
        \begin{equation}\label{inhomo_meas1}
            \bP \big[\forall \delta \in \bR: \|n\alp -\gamma\| \|X_n-\delta\| \ls \psi(n) \text{ for inf. many } n\gs 1 \big]  = 1.
        \end{equation}

        \vspace{1ex}

        Conversely, there exists $\varepsilon>0$ such that, for every $\alpha\in\Bad$ and every monotonically decreasing $\psi$ satisfying
        \begin{equation}\label{fake_convergence}
            \psi(n) \leq \frac{1}{n \log n}, \qquad \limsup_{N \to \infty} \sum_{N < n \le 2^N} \psi(n) < \varepsilon,
        \end{equation}
        one has
        \begin{equation} \label{homo_meas0}
            \bP \big[ \forall \delta \in [0, 1]: \|n\alp\| \|X_n-\delta\| \ls \psi(n) \text{ for inf. many } n\gs 1 \big]  = 0.
        \end{equation}
\end{thm}

    \vspace{2ex}
    
While the proof of this theorem is deferred to the Appendix, we give several remarks below:

    \begin{itemize}
        \vspace{1ex}
        \item[(i)] Note that when $\psi(n)$ is of the form
        \[
            \psi(n) = \frac{1}{n \cdot \log n \cdot \log \log n \cdot \ldots \cdot \log^{(k)} n}, \qquad k \geq 1,
        \]
        then one has~\eqref{fake_khintchine}, which suggests that one can potentially improve upon Corollary \ref{cor1} by roughly a factor of $\log \log n$. For
        \[
            \psi(n) = \frac{1}{n \cdot \log n \cdot \log \log n \cdot \ldots \cdot (\log^{(k)} n)^{1+ \varepsilon}}, \qquad k \geq 1
        \]
        however, one has~\eqref{homo_meas0}. This may resemble, at first sight, a classical Khintchine-type dichotomy, but \eqref{fake_convergence} shows that this is not the case: there exists monotonically decreasing $\psi$ with
        \begin{equation}\label{no_khintchine}
            \sum_{n \in \bN} \psi(n) = \infty, \qquad \limsup_{N \to \infty} \sum_{N < n \le 2^N} \psi(n) < \varepsilon,
        \end{equation}
        such that for any $\alpha \in \Bad$, \eqref{homo_meas0} holds. \\

        \item[(ii)]  In contrast to the constant invariance principle that appears in various metric settings (e.g. Cassels's invariance principle~\cite[Lemma 9]{cassels}, \cite[Lemma 1]{Ber_Vel}), there is no such principle in this problem. This property can be seen to be inherited from the classical Dvoretzky covering problem, where covering with the approximation function $\psi(n)=\mfrac{c}{n}$ depends on the value of~ $c$. \\

        \item[(iii)]  Note that the second part of the theorem cannot hold with $\norm{n\alpha-\gamma}$ in place of $\norm{n\alpha}$ \textit{uniformly} in $\gamma$: suppose that $\alpha$ is (homogeneously) badly approximable, but we pick some $\gamma$ that is inhomogeneously well approximable, e.g.
        \[
            \lVert n\alpha - \gamma \rVert \leq \frac{1}{n^2} \quad \text{ for infinitely many } n \in \bN.
        \] Then for the function $\psi(n):=\mfrac{1}{n^2}$, even the trivial bound $\lVert X_n-\delta\rVert \leq \mfrac{1}{2}$ suffices to obtain infinitely many solutions to $\norm{n\alpha-\gamma}\,\norm{X_n-\delta}\le \psi(n)$, while $\sum_n \psi(n)<\infty$. In fact, Kurzweil's theorem~\cite{kurzweil} shows that if $\sum_{n\in\bN}\psi(n)=\infty$ and $\alpha\in\Bad$, then for almost every $\gamma$ one has $\lVert n\alpha-\gamma\rVert \le \psi(n)$ infinitely often. In particular, this shows that \eqref{no_khintchine} provides an example that does not generalize to inhomogeneous approximation. \\

        \item[(iv)] Finally, we note that the restriction $\psi(n)\le \mfrac{1}{n\log n}$ in~\eqref{fake_convergence} cannot be dropped completely: to see this, fix $\alpha\in\Bad$ and a sparse subsequence $(q_{n_k})_{k \in \mathbb{N}}$ of the convergent denominators of $\alpha$, and define
\[
    \psi(n):=\frac{\varepsilon/2}{q_{n_k}}, \quad q_{n_{k-1}} < n \le q_{n_k}.
\]
If $(q_{n_k})_k$ is chosen sufficiently sparse, then
\[
    \limsup_{N \to \infty} \sum_{N < n \le 2^N} \psi(n) < \varepsilon,
\]
while $\mfrac{\psi(q_{n_k})}{\lVert q_{n_k}\alpha \rVert} > \mfrac{\varepsilon}{2}$ shows that \eqref{inhomo_meas1} holds as soon as
\[
    \bP \Big[\forall \delta \in \bR: \|X_{q_{n_k}}-\delta\| \ls \frac{\eps}{2} \text{ for inf. many } k\gs 1 \Big] = 1,
\]
which follows immediately from Shepp's theorem.
\end{itemize}

    \subsection*{Intersection with fractal sets}
    
     The middle-third Cantor set $K$ can be defined as the set of all real numbers whose base-$3$ expansion contains no digit $1$. It is well known that $K$ has Lebesgue measure zero and Hausdorff dimension $\mfrac{\log 2}{\log 3}$. Mahler's 1984 questions~\cite{Mahler} asked about approximating the elements in~$K$ by
    \begin{itemize}
    \item[(i)] rationals in Cantor’s set (intrinsic approximation)
    \item[(ii)] by rationals not (necessarily) in Cantor’s set (extrinsic approximation).
    \end{itemize}

Both questions initiated a flurry of work on approximating points in the middle-third Cantor set~$K$ and generalizations to other fractal sets. In this article, we focus on the extrinsic question (ii). In this setup, recent breakthroughs on Khintchine-type statements for fractal sets~\cite{BHZ25,DJ24,KL23} have been obtained with various conditions on the fractal set, with the solution of the emblematic case of the middle-third Cantor set being established in~\cite{BHZ25}.

While this shows, in particular, that $\mu$-almost every (with respect to the natural measure on $K$) $x\in K$ has minimal irrationality exponent $2$, it does not provide any information on how many points $x\in K$ are well approximable with rationals in $\mathbb{R}$. Defining, as usually,
\[
    W_{\nu} := \Big\{\delta \in [0,1):  \Big| \delta - \frac{a}{n} \Big| \leq \frac{1}{n^{\nu}} \text{ for inf. many } \frac{a}{n}\in\mathbb{Q} \Big\},\quad  \nu \geq 2,
\] natural questions concern the properties of $W_{\nu}\cap K$ for $\nu>2$. The existence of irrational numbers in $K$ with any prescribed irrationality exponent (and hence the fact that $W_{\nu}\cap K\neq\emptyset$ for all $\nu\ge 2$) was shown earlier in~\cite{Bugeaud08, LSV, Weiss}. A recent work \cite{he2026} provides conjecturally sharp upper bounds, as well as non-trivial lower bounds for the Hausdorff dimension of these intersections.
Furthermore, it was shown in~\cite{chen2025} that $\bigcup_{\nu > 2}W_{\nu}\cap K$ has full Hausdorff dimension $\dimH K$. 
However, the problem of determining for fixed $\nu >2 $ the size of $W_{\nu}\cap K$ (or other fractal sets in place of $K$) remains widely open; see, for instance, the recent survey of Beresnevich--Velani~\cite{BV26} and the references therein.\\

In order to provide a conjectural framework, Bugeaud and Durand~\cite{BD} introduced several random analogues of this problem. In particular, in one of their models the rationals are replaced by random points. More precisely, let $(X_n)_{n\in\N}$ be a sequence of i.i.d.\ random points, uniformly distributed on $[0,1)$, and for $\nu \ge 1$ define the random set
    \begin{equation} \label{random_intersection}
        E\big((X_n),\nu\big) :=\Big\{\delta\in[0,1): \norm{\delta-X_n}\le \frac{1}{n^{\nu}}
        \ \text{for infinitely many } n\in\N\Big\}.
    \end{equation} 

    \begin{remark*}
        If we replaced $(X_n)_{n \in \mathbb{N}}$ by the sequence of Farey fractions, then the above coincides (up to a constant) with $W_{2\nu}$ since there are $\asymp N^2$ Farey fractions with denominator $\leq N$, and the Jarnik--Besicovitch Theorem implies that
        $\dimH W_{2\nu} = \mfrac{1}{\nu}$. The almost sure size of $E((X_n),\nu)\cap K$ stands therefore as a model for the size of $W_{2\nu}\cap K$, unless the intrinsic contribution, i.e., approximating with rationals from $K$,  gives a bigger contribution. For more context on the matter, we refer to \cite{BD} and \cite{he2026}.
    \end{remark*}

    Bugeaud and Durand~\cite{BD} computed the almost sure Hausdorff dimension of the intersection of~\eqref{random_intersection} with the Cantor set.
    More generally, their result applies to any \textit{Ahlfors regular} compact set $G$ of dimension $s\in(0,1]$. We postpone the exact definition to Section~\ref{sec_fractal}, and only mention here that all standard sets with missing digits are Ahlfors regular. Precisely, it was proven in~\cite{BD} that almost surely
    \begin{equation}\label{inter_dim}
        \dimH\big(E((X_n),\nu)\cap G\big)
        = \frac{1}{\nu}+\dimH(G)-1,
    \end{equation} when this value is non-negative; otherwise, this set is almost surely empty.

    The method of~\cite{BD} also allowed them to obtain a similar result when the random points are replaced by points of the form $\{q_n x\}$, where $(q_n)_{n\in\N}$ is a rapidly growing integer sequence and $x$ is Lebesgue-generic. More precisely, \eqref{inter_dim} continues to hold with $E((X_n),\alpha)$ replaced by the set
    \begin{equation}\label{def_E_set}
        E\big((q_n),x,\nu\big)
        :=\Big\{\delta\in[0,1]: \norm{\delta-q_n x}\le \frac{1}{n^{\nu}} \ \text{for infinitely many } n\in\N\Big\},
    \end{equation} subject to the assumption\footnote{We note that it is apparent from their proof and the discussion in \cite[p.1259]{BD}, that this assumption could be replaced by the (slightly) weaker condition of $\liminf_{n \to \infty}\frac{\log(q_{n+1}/q_n)}{\log n} > 1 + \frac{1}{1-s}$, with the right-hand side being interpreted as $\infty$ for $s = 1$.}
    \begin{equation}
    \label{superlac_BD}
    \lim_{n\to\infty}\frac{\log(q_{n+1}/q_n)}{\log n}=\infty.
    \end{equation}
    Note that the growth assumption \eqref{superlac_BD} is very restrictive: it forces $(q_n)_{n\in \bN}$ to grow at least on the scale $q_n=\lfloor n^{b_n n}\rfloor$ for some sequence $b_n\to\infty$. 
    This growth-rate condition ensures that one satisfies for all Borel sets $(B_i)_{i \in \mathbb{N}}$,
    \begin{equation}\label{bd_uncorr_cond}
        \frac{\mathbb{P}[q_1x \in B_1, \ldots ,q_N x \in B_N]}{\mathbb{P}[q_1 x \in B_1]\cdots \mathbb{P}[q_N x \in B_N]} < C < \infty.
    \end{equation}
    In our approach, we circumvent restrictive $N$-correlation assumptions such as \eqref{bd_uncorr_cond}, and replace this with bounds on certain conditional second moments.
    This allows us to reduce the growth-rate requirement to a certain sub-lacunary growth, extend it to real-valued sequences, and, additionally, extend the result to \textit{all} integer-valued sequences satisfying certain restrictions on their prime factorization. In order to keep the statement as general as possible, we formulate our next theorem for integer sequences in terms of a gcd-sum bound:

    \begin{thm}\label{thm4}
    Let $G$ be an Ahlfors regular set of dimension $s\in(0,1]$, and fix $\nu\ge 1$. When $1 \le \nu \le \frac{1}{1-s}$ (and for any $\nu \ge 1$ if $s=1$), one has: \\
    
    (1) Let $(q_n)_{n \in \bN}$ be an increasing sequence of integers satisfying the following: there exists a non-decreasing positive function $f(x) \ll x^{o(1)}$ satisfying $f(2x) \le c f(x) \ \forall x \in \bR$ for some fixed $c > 0$ and an index set $\bI:=\{n_k:k\ge 1\}\subseteq\N$ such that $\#(\bI\cap [N]) \asymp \frac{N}{f(N)}$, and an arbitrarily slowly growing function $\psi(n) \to \infty$ such that for all sufficiently large $N$ one has
    \[
        \sum_{N<k \le m\le 2N}
        \min\bigg[\frac{\gcd(q_{n_m},q_{n_k})}{q_{n_m}} \min\Big( \log \frac{q_{n_m}}{q_{n_k}}, \log \log N \Big) ,\, N^{1-\nu-\eps\nu}\bigg]
        \le \frac{N^{2-\nu}}{\psi(N) f^{\nu} (N)}.
    \]
   
    Then for Lebesgue-a.e.\ $x\in[0,1]$ one has
    \begin{equation}\label{dim_intersec}
    \dimH\bigl(E((q_n),x,\nu)\cap G\bigr)
    = \frac{1}{\nu}+\dimH(G)-1,
    \end{equation} when this value is non-negative; otherwise, this set is empty. Here $E((q_n),x,\nu)$ is defined in~\eqref{def_E_set}. \\

    (2) Let $(q_n)_{n\in\bN}$ be an increasing real-valued sequence satisfying the gap condition
    \[
    \frac{q_{n+1}}{q_n} > 1 + \frac{1}{\Phi(n)},
    \] where $\Phi(n)$ is any increasing function such that $\Phi(n) \ll n^{o(1)}$. Then for Lebesgue-a.e. $x \in [0, 1)$ the same implication holds. \\ 
    
    Otherwise, when $\nu > \frac{1}{1-s}$, for any real-valued sequence $(q_n)_{n \in \bN}$ the set in~\eqref{dim_intersec} is empty for Lebesgue-a.e.\ $x \in [0,1]$.
    \end{thm}
    
    \begin{remark*}
    \item
    \begin{itemize}
        \vspace{1ex}
        \item[(i)] We note that the result in the regime $\nu > \frac{1}{1-s}$ (where the set is empty) follows directly from the argument of Bugeaud and Durand~\cite{BD}. Moreover, when $\nu \le \frac{1}{1-s}$, their argument gives the correct upper bound on the dimension of the intersection for any real-valued sequence $(q_n)_{n \in \bN}$
    --- see Lemma \ref{packing_dichotomy} for details. Therefore, our main work is to establish the matching lower bound in the regime $1 \le \nu \le \frac{1}{1-s}$ (and for all $\nu \ge 1$ when $s=1$), which is where the technical conditions of Theorem~\ref{thm4} come into play. \\
    
    \item[(ii)] We note that most parts of Theorem \ref{thm4} could also be established for more general gauge functions than simply the dimension functions $x \mapsto x^h$, comparably to how it was done in the work of Bugeaud and Durand \cite{BD}. While the same argument applies, the necessary modification in the gcd-sum bound would make the general statement even more technical, so we omit it for the sake of readability. \\

    \item[(iii)] It follows from standard procedures (monotonicity of 
    $\dimH\big(E((X_n),\nu)\cap G\big)$, countability and density of $\mathbb{Q}$, and countable intersections of full measure sets being full that \eqref{dim_intersec} holds \textit{uniformly} in $\nu$ (as long as the requirements are met) on a set of full measure. 
    \end{itemize}
    \end{remark*}

    Part (2) of Theorem \ref{thm4} essentially requires the real-valued sequence to be almost lacunary; for example, it applies to any sequence of the form $q_n = \exp(n^{1-o(1)})$. Part (1) allows one to treat much slower growing integer-valued sequences, but its restrictions may in general be difficult to verify. Nevertheless, we demonstrate how the theorem applies in several important special cases:
    \begin{cor}\label{cor2}
        Let $G,s$ and $\nu$ be as in Theorem~\ref{thm3}, with $1 \le \nu \le \frac{1}{1-s}$ (and any $\nu \ge 1$ if $s=1$). Then \eqref{dim_intersec} holds for the following choices of $(q_n)_{n \in \mathbb{N}}$:
        \vspace{1ex}
        \begin{enumerate}
        \item[(i)] \emph{Prime powers:} $q_n=p_n^d$ for any $d \ge \nu$, $d \in \bN$;
        \item[(ii)] \emph{Monomials:} $q_n = n^d$ for any $d > \nu, d \in \bN$;
        \item[(iii)] \emph{Polynomials:} $q_n=P(n)$ for $P\in\bZ[x]$ with $\deg P > \nu+1$.
    \end{enumerate}
    \end{cor}
    
    To mention some important special cases, the above shows that when $\nu=1$, we obtain the precise Hausdorff dimension for primes and for any polynomial sequence of degree at least~$3$, for every Ahlfors regular set~$G$. When $G=K$ is the middle-third Cantor set, we obtain the precise dimension~\eqref{dim_intersec} for all $\nu$ and for any integer-valued polynomial of degree at least~$4$. Finally, when $G=[0,1)$, we obtain, for example,
    \[
        \dimH \Big\{ \delta \in [0,1): \norm{q_n x-\delta} \le \frac{1}{n^{\nu}} \text{ for inf. many } n \in \bN \Big\} = \frac{1}{\nu},
    \] for $q_n=n^2$, for any polynomial of degree~$> 2$, and for any real-valued lacunary sequence (including $q_n=2^n$). This gives a vast generalization of the result of Fan, Schmeling, and Troubetzkoy~\cite{FST} in the Lebesgue measure case.

    The proof of (i) is a straightforward application of the Prime Number Theorem: For $q_n=p_n^d$ we have for every $\varepsilon > 0$ and $N$ sufficiently large
    \[
        \sum_{N<m=k\le 2N} N^{1-\alpha-\varepsilon\alpha} \le N^{2-\alpha-\varepsilon\alpha},
    \] and
    \[
        \sum_{N< k \le m \le 2N} \frac{\gcd(p_{n_m}^d,p_{n_k}^d)}{p_{n_m}^d}
        \ll \sum_{N<m\neq k\le 2N} \frac{1}{(N\log N)^d}
        \ll \frac{N^{2-d}}{(\log N)^d}.
    \] 
    The proof of (ii) directly follows from (i): we take $\bI=\{p_n:n\ge 1\}$ as the index set; by the prime number theorem it satisfies $\#(\bI\cap [1,N]) \asymp \frac{N}{\log N}$, and we apply the previous estimate.
    The proof of~(iii) requires some input from sieve theory, so we postpone it until Section~\ref{sec_fractal}.\\
    
    We further note that any integer sequence without any special arithmetic structure, in the sense that it contains a sufficiently dense subsequence of primes (as would follow from an appropriate prime number theorem for the sequence), fits the hypotheses of Theorem~\ref{thm3}. For example, the direct application of the strategy for the proof of~(ii) would give the same result for~(iii) with $\deg~P > \nu$ for irreducible polynomials without fixed prime divisors under the assumption of the widely open Bateman-Horn Conjecture \cite{BH62}. In the actual proof for~(iii) we circumvent this assumption by showing the weaker statement that for an absolute $C > 0$ there are sufficiently many $n \in \bN$ with $\tau(P(n)) < C$, where $\tau$ is the divisor function.

    Another natural example would be to consider the Piatetski-Shapiro sequences $q_n=\lfloor n^c\rfloor$, $c>1,\ c\notin\bN$. Heuristically, one expects that the number of primes of this form up to~$N$ should satisfy
    \begin{equation} \label{PS_asymp}
        \sum_{\substack{p \le N \\ p = \floor{n^c}}} 1 \sim \frac{N^{\frac{1}{c}}}{\log N}.
    \end{equation} Thus, for example, when $G=K$ is the middle-third Cantor set, one expects that any sequence $q_n=\lfloor n^c\rfloor$ with $c \ge (1-\frac{\log 2}{\log 3})^{-1}=2.709511\ldots$ would satisfy the assumptions of Theorem~\ref{thm4}. Unfortunately, even the infinitude of primes in these sequences is currently known only for $c\le 1.185365\ldots$~\cite{Rivat_Wu}. Note that the hypothesis required in Theorem~\ref{thm4} is much weaker than~\eqref{PS_asymp}. Nevertheless, we leave this question open.

    \subsection*{Structure of the paper}

    All main statements in this work rely on a newly developed colored binary tree approach that might be of independent interest. The argument admits a natural percolation interpretation on the dyadic tree of sub-intervals of $[0,1)$. Namely, after assigning to each level $n$ a dyadic block of times $m \asymp 2^n$, one colors a vertex corresponding to the dyadic interval $I_{n,k} = [\mfrac{k}{2^n}, \frac{k+1}{2^n})$ whenever some center $\{q_mx\}$ from $m \asymp 2^n$ falls into $I_{n,k}$; in this language, covering (the upper bound) corresponds to extinction of infinite uncolored rays, whereas non-covering (the lower bound) is driven by the survival of a thick uncolored ray. We use ``path'' throughout, rather than ``ray'', for infinite downward paths in the tree. 
    
    In Section~\ref{sec_binary} we introduce this colored tree framework in details and sketch how this leads to a short proof of a weaker form of Shepp's result in the i.i.d.\ case, providing a blueprint for the proofs of Theorems~\ref{thm1} and~\ref{thm2}. In Section~\ref{sec_upper} we prove Theorem~\ref{thm1} by combining the colored-tree setup with a suitably formulated Borel--Cantelli lemma and Fourier analysis, and then deduce Corollary~\ref{cor1} using the (by now standard) mass distribution principle as in Pollington--Velani's work. In Section~\ref{sec_lower} we prove Theorem~\ref{thm2} using a variant of a branching random walk together with conditional second moment estimates. In Section~\ref{sec_fractal} we prove Theorem~\ref{thm4} by combining a Baire category argument with second moment estimates, and deduce Corollary~\ref{cor2}~(3) using an input from sieve theory. We conclude the paper with an appendix proving Theorem~\ref{thm3} for the randomized Littlewood--Cassels problem.

    \subsection*{Notation}
    
    Throughout, we write $e(x)=e^{2\pi i x}$. The relations $f(x)\ll g(x)$ and $f(x)=O(g(x))$ mean that there exists a constant $c>0$ such that $|f(x)|\le c\,g(x)$ for all sufficiently large $x$. The relation $f(x)=o(g(x))$, with $g(x)\neq 0$, means that $\frac{f(x)}{g(x)}\to 0$ as $x\to\infty$. The relation $f(x)\asymp g(x)$ means that $f(x)=O(g(x))$ and $g(x)=O(f(x))$. Next, $\lfloor x \rfloor$ denotes the greatest integer not exceeding $x$, and $\{x\}=x-\lfloor x \rfloor$ is the fractional part of $x$. Further, $\bP$ denotes the probability governing the i.i.d.\ points on $[0,1)$, and $\lambda$ denotes Lebesgue measure on $[0, 1)$. Given a positive integer $N\gs 1,$ we write $[N]$ for the set of integers $\{ 1, \ldots, N\}$. Finally, $\log x$ denotes $\log_2 x$ unless a different base is clear from the context, and $\log^{(k)} x$ denotes the $k$-iterated logarithm in base $2$.

    \subsection*{Acknowledgements} 
    
    This work was funded in whole, or in part, by the Austrian Science Fund (FWF). MH was supported by FWF project 10.55776/ESP5134624, AS was supported by FWF project 10.55776/ESP531, ES was supported by FWF projects 10.55776/P35322 and 10.55776/PAT5120424, AZ was supported by FWF-ANR project Arithrand (I 4945-N and ANR-20-CE91-0006) and by FWF Project 10.55776/PAT3862225. We are grateful to Christoph Aistleitner for his comments on the first version of this paper. \\

    %SECTION 2: Tree setup + random models
    
    %%%%%%%%%%%%%%%%%%%%%%%%%%%%%%%%%%%%%%%%%%%%%%%%%
    
    \section{Colored tree framework} \label{sec_binary}

      \subsection*{Setup}
    
The proofs of Theorems~\ref{thm1}, \ref{thm2}, and~\ref{thm4} fit into a general framework of random branching on a tree, which is natural for studying covering problems and fractal sets. In this paper it will be more convenient to work with the notion of \textit{random coloring} of the tree rather than \textit{random branching}, although the two are, in a sense, equivalent.

In this section, we give a detailed description of the tree-coloring process and explain how it can be used to detect whether a sequence of centers and lengths covers intervals. For random uniform i.i.d.\ centers, we sketch the proofs of the following statements:

(i) infinite covering holds almost surely for interval lengths $\ell_N=\frac{2026}{N}$;

(ii) infinite covering does not hold almost surely for interval lengths $\ell_N=\frac{1}{2026N}$.

These statements are, of course, weaker than Shepp's result, or even earlier results of Billard and Kahane, but they suffice as an illustration before the main proofs.

    Partition the interval $[0,1)$ into dyadic subintervals
    \[
        I_{n,k} := \Big[ \frac{k}{2^n}, \frac{k+1}{2^n} \Big), \qquad  0 \le k < 2^n,
    \] for $n=0,1,2,\dots$. They naturally give rise to an infinite binary tree, where each $I_{n,k}$ represents a vertex. The index $n$ corresponds to a \textit{level} of the tree, and $k$ is the index of a vertex within the level, as illustrated in Figure~\ref{figure1}.

    \begin{figure}[ht!]
    \centering
    \begin{tikzpicture}[scale=0.7,transform shape,
        level distance=2.0cm,
        level 1/.style={sibling distance=5.5cm},
        level 2/.style={sibling distance=3cm},
        every node/.style={circle, draw=black, fill=white, minimum size=5mm, inner sep=0pt},
        label distance=2mm
    ]
    
        \node[label=right:{$I_{0,0}$}] {}
            child {node[label=right:{$I_{1,0}$}] {}
                child {node[label=right:{$I_{2,0}$}] {}}
                child {node[label=right:{$I_{2,1}$}] {}}
            }
            child {node[label=right:{$I_{1,1}$}] {}
                child {node[label=right:{$I_{2,2}$}] {}}
                child {node[label=right:{$I_{2,3}$}] {}}
            };
    
    \end{tikzpicture}
    \caption{Binary tree representing the partition of $[0,1)$ into nested dyadic intervals}
    \label{figure1}
    \end{figure}
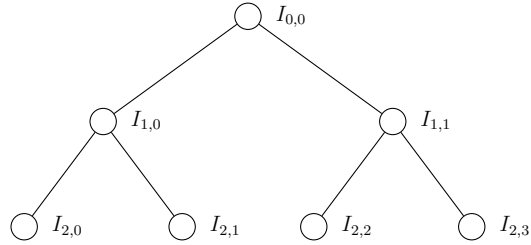

Each real number $\delta \in [0,1)$ corresponds to an infinite vertical path down the tree. This correspondence is unique, except when $\delta$ is rational: in that case, $\delta$ has two binary expansions,
$\delta = 0.a_1\ldots a_k 1000\ldots$ and $\delta = 0.a_1\ldots a_k 0111\ldots$, and hence can be represented by two distinct paths. For definiteness, we always choose the first one.

We now describe how a sequence of points $(X_N)_{N \in \bN}$ on $[0,1)$ gives rise to a tree coloring. Given a sequence of points $(X_N)_{N \in \bN}$ and a sequence of non-negative integers $(N_n)_{n\in\bN}$, we define a \textit{colored binary tree}
\[
    \cT_2 = \cT_2\big( (X_N)_{N \in \bN}, (N_n)_{n\in\bN}\big)
\] as follows: at each level $n = 0,1,2,\ldots$ we place  the points $X_N$ with indices $N$ in the range
    \begin{equation} \label{index_range}
      N_{n-1} < N \ls N_n
    \end{equation} among their corresponding intervals $I_{n,k}$. An interval (vertex) is colored if it contains at least one point, and remains uncolored otherwise (see Figure~\ref{figure2}). In this way, the colored tree is uniquely determined by the sequences $(X_N)_{N \in \bN}$ and $(N_n)_{n\in\bN}$. We note that the choice $N_n=N_{n-1}$ for some $n$ is admissible in this construction, meaning that no points will be placed on the $n$-th level of the tree.

    \vspace{2ex}
    
    \newcommand{\pararrow}[6]{
    \pgfmathanglebetweenpoints{\pgfpointanchor{#1}{center}}{\pgfpointanchor{#2}{center}}
    \let\EdgeAngle\pgfmathresult
    \draw[->, line width=0.45pt]
    ($ (#1)!#3!(#2) + ({\EdgeAngle + #6}:#5) $) --
    ($ (#1)!#4!(#2) + ({\EdgeAngle + #6}:#5) $);
    }
    
    \usetikzlibrary{calc}
    
    \setlength{\abovecaptionskip}{-50pt}
    
    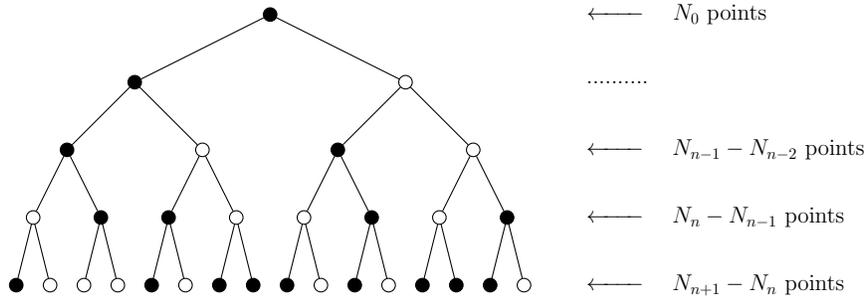
\begin{figure}[ht!]
    \centering\vspace{-8ex}
    \begin{tikzpicture}[scale=0.6,transform shape,
    level distance=1.5cm,
    level 1/.style={sibling distance=6cm},
    level 2/.style={sibling distance=3cm},
    level 3/.style={sibling distance=1.5cm},
    level 4/.style={sibling distance=0.75cm},
    every node/.style={circle, draw=black, minimum size=3mm, inner sep=0pt},
    blacknode/.style={fill=black},
    whitenode/.style={fill=white},
    levelLabel/.style={draw=none, anchor=west, font=\large}
    ]
    
    \node[blacknode] (n1) {}
    child {node[blacknode] (n2) {}
        child {node[blacknode] (n3) {}
            child {node[whitenode] (n4) {}
                child {node[blacknode] (n5) {}}
                child {node[whitenode] (n6) {}}
            }
            child {node[blacknode] (n7) {}
                child {node[whitenode] (n8) {}}
                child {node[whitenode] (n9) {}}
            }
        }
        child {node[whitenode] (n10) {}
            child {node[blacknode] (n11) {}
                child {node[blacknode] (n12) {}}
                child {node[whitenode] (n13) {}}
            }
            child {node[whitenode] (n14) {}
                child {node[blacknode] (n15) {}}
                child {node[blacknode] (n16) {}}
            }
        }
    }
    child {node[whitenode] (n17) {}
        child {node[blacknode] (n18) {}
            child {node[whitenode] (n19) {}
                child {node[blacknode] (n20) {}}
                child {node[whitenode] (n21) {}}
            }
            child {node[blacknode] (n22) {}
                child {node[blacknode] (n23) {}}
                child {node[whitenode] (n24) {}}
            }
        }
        child {node[whitenode] (n25) {}
            child {node[whitenode] (n26) {}
                child {node[blacknode] (n27) {}}
                child {node[blacknode] (n28) {}}
            }
            child {node[blacknode] (n29) {}
                child {node[blacknode] (n30) {}}
                child {node[whitenode] (n31) {}}
            }
        }
    };

    \coordinate (labelx) at (7,0);

    \node[levelLabel] at (labelx |- n1) {$\xleftarrow{\hspace{1cm}} \quad N_0$ points};
    \node[levelLabel] at (labelx |- n2) {..........};
    \node[levelLabel] at (labelx |- n3) {$\xleftarrow{\hspace{1cm}} \quad N_{n-1}-N_{n-2}$ points};
    \node[levelLabel] at (labelx |- n4) {$\xleftarrow{\hspace{1cm}} \quad N_n - N_{n-1}$ points};
    \node[levelLabel] at (labelx |- n5) {$\xleftarrow{\hspace{1cm}} \quad N_{n+1} - N_n$ points};
    
    \end{tikzpicture}
    \vspace{4ex}
    \caption{Coloring the binary tree}
    \label{figure2}
    \end{figure}
    
    \setlength{\abovecaptionskip}{10pt plus 2pt minus 2pt}
    
    \medskip
  
   The infinite covering property now follows from a  property of the colored tree $\cT_2= \cT_2((X_N)_{N=1}^{\infty}, \NN)$ for the specific choice \begin{equation} \label{NN_def} N_n = \lfloor L2^0\rfloor + \cdots + \lfloor L2^{n}\rfloor \qquad (n\gs 1) \end{equation}  where $L>0$ is some parameter. One verifies directly that if there is no \textit{infinite uncolored path} down the tree $\cT_2\big((X_N)_{N\in\N},\NN\big)$, then infinite covering holds with the intervals $(X_N-\mfrac{\ell_N}{2},\, X_N+\mfrac{\ell_N}{2})$ whenever $\ell_N \ge \frac{2L}{N}$. Indeed, assume there are infinitely many colored vertices along any vertical path. Then since any $\delta$ corresponds to such a  path, $\delta$ is infinitely often at distance at most $\frac{1}{2^n}$ from some point $X_N$ with $N$ in the range~\eqref{index_range}. But by \eqref{index_range} we have $N < L2^{n+1}$. Thus, for all sufficiently large $n$ and the corresponding $N$, one has
    \[
        \norm{\delta - X_N} \le \frac{1}{2^n} < \frac{2L}{N}.
    \]  
   The converse statement is not true in general: the existence of an infinite uncolored path does not imply non-covering. Indeed, it is possible to have such an uncolored path for which, infinitely often, the left or the right neighbouring vertex is colored. In that case, the number $\delta$ corresponding to the path can still be infinitely often very close to $X_N$, for instance satisfying $\norm{\delta - X_N} \le \mfrac{1}{N^{100}}$. This is why for non-covering, we require a stronger property: namely, the existence of a \textit{thick} infinite uncolored path in $\cT_2\big((X_N)_{N\in\N}, \NN\big)$ --- a path for which the left and right neighbours of each vertex are also uncolored (see Figure~\ref{figure3}). 
    
    One can then check directly that the existence of such a thick path implies the existence of a point $\delta \in [0,1)$ that is at least $\frac{L-\eps}{N}$ away from every $X_N$, for arbitrarily small $\eps>0$ and all sufficiently large $N = N(\eps)$. Indeed, in this case, from~\eqref{index_range} we have
    \[
        L 2^n - (n+1) < N \ \ \Longrightarrow \ \ \frac{L-\eps}{N} < \frac{1}{2^n} \ \ \Longrightarrow \ \ \norm{X_n - \delta} \ge \frac{1}{2^n} > \frac{L-\eps}{N}.
    \]

\begin{figure}[ht!]
\centering
\begin{tikzpicture}[
  scale=0.75,
  x=6mm, y=10mm,
  blk/.style={circle,draw,fill=black,minimum size=2.0mm,inner sep=0pt},
  wht/.style={circle,draw,fill=white,minimum size=2.0mm,inner sep=0pt},
  edge/.style={line width=0.6pt}
]

\newcommand{\brsc}{0.6}

\newcommand{\Lsqbracket}[1]{%
  \draw[line width=0.6pt]
    ($(#1)+(0,\brsc*0.35)$) -- ($(#1)+(0,-\brsc*0.35)$)
    ($(#1)+(0,\brsc*0.35)$) -- ($(#1)+(\brsc*0.18,\brsc*0.35)$)
    ($(#1)+(0,-\brsc*0.35)$) -- ($(#1)+(\brsc*0.18,-\brsc*0.35)$);
}
\newcommand{\Rsqbracket}[1]{%
  \draw[line width=0.6pt]
    ($(#1)+(0,\brsc*0.35)$) -- ($(#1)+(0,-\brsc*0.35)$)
    ($(#1)+(0,\brsc*0.35)$) -- ($(#1)+(-\brsc*0.18,\brsc*0.35)$)
    ($(#1)+(0,-\brsc*0.35)$) -- ($(#1)+(-\brsc*0.18,-\brsc*0.35)$);
}

\foreach \i in {0,...,3} {
  \pgfmathsetmacro{\x}{4*\i + 1.5}
  \ifnum\i=3
    \node[blk] (a\i) at (\x,0) {};
  \else
    \node[wht] (a\i) at (\x,0) {};
  \fi
}

\coordinate (brL) at ($(a0)+(-\brsc*0.9,0)$);
\coordinate (brR) at ($(a2)+(\brsc*0.9,0)$);
\Lsqbracket{brL}
\Rsqbracket{brR}

\foreach \j in {0,...,7} {
  \pgfmathsetmacro{\x}{2*\j + 0.5}
  \ifnum\j=0
    \node[blk] (b\j) at (\x,-1) {};
  \else\ifnum\j=4
    \node[blk] (b\j) at (\x,-1) {};
  \else\ifnum\j=5
    \node[blk] (b\j) at (\x,-1) {};
  \else\ifnum\j=6
    \node[blk] (b\j) at (\x,-1) {};
  \else
    \node[wht] (b\j) at (\x,-1) {};
  \fi\fi\fi\fi
}

\coordinate (brL2) at ($(b1)+(-\brsc*0.9,0)$);
\coordinate (brR2) at ($(b3)+(\brsc*0.9,0)$);
\Lsqbracket{brL2}
\Rsqbracket{brR2}

\foreach \k in {0,...,15} {
  \ifnum\k<4
    \node[blk] (c\k) at (\k,-2) {};
  \else\ifnum\k>7
    \node[blk] (c\k) at (\k,-2) {};
  \else
    \node[wht] (c\k) at (\k,-2) {};
  \fi\fi
}

\coordinate (brL3) at ($(c4)+(-\brsc*0.8,0)$);
\coordinate (brR3) at ($(c6)+(\brsc*0.8,0)$);
\Lsqbracket{brL3}
\Rsqbracket{brR3}

\foreach \i in {0,...,3} {
  \pgfmathtruncatemacro{\jone}{2*\i}
  \pgfmathtruncatemacro{\jtwo}{2*\i+1}
  \draw[edge] (a\i) -- (b\jone);
  \draw[edge] (a\i) -- (b\jtwo);
}

\foreach \j in {0,...,7} {
  \pgfmathtruncatemacro{\kone}{2*\j}
  \pgfmathtruncatemacro{\ktwo}{2*\j+1}
  \draw[edge] (b\j) -- (c\kone);
  \draw[edge] (b\j) -- (c\ktwo);
}

\end{tikzpicture}
\caption{Uncolored thick path in the binary tree}
\label{figure3}
\end{figure}

    \subsection*{The covering case}
    
    We now present an easy proof for the upper bound in the above framework for i.i.d.\ points, which serves as a blueprint for proving the same for lacunary dilates later. In view of the previous discussions, we need to prove that almost surely, the colored tree does not contain any infinite uncolored path. In the i.i.d. regime, we require only an application of the Borel--Cantelli lemma with suitably chosen events. We proceed as follows.

    Let $(n_i)_{i\in\N}$ be a strictly increasing sequence of positive integers, $n_1<n_2<n_3<\ldots$, representing the starting levels of the finite uncolored paths, and let $(R_i)_{i\in\N}$ be a strictly increasing sequence of positive integers,
    $R_1<R_2<R_3<\ldots$, representing their respective lengths (heights). To avoid vertical overlap, we also require that $n_{i+1}>n_i+R_i$. Note that the sequence $(n_i)_{i\in\N}$ can be chosen arbitrarily sparse.

    Now define a sequence of ``bad events'' by
    \begin{equation}\label{event}
    \begin{split} A_i &:= A_i\bigl(n_i, R_i, \mathcal{T}_2\bigl((X_N)_{N\in\N},\NN\bigr)\bigr) \\
    &= \big\{ \text{there exists an uncolored path of length $R_i+1$ starting at level $n_i$} \big\}.
\end{split}  \end{equation} See also Figure~\ref{figure4}. In the definition of $\NN$ in \eqref{NN_def} we take $L = 1013$. Clearly, the existence of an infinite uncolored path implies that $A_i$ occurs for infinitely many $i\in\N$ (in fact, for all $i\ge i_0$, although this is not needed).

    \begin{figure}[ht!]
    \centering
    \begin{tikzpicture}[scale=0.6]
    
    \coordinate (A) at (-3,0);   
    \coordinate (B) at ( 3,0);   
    \coordinate (C) at ( 2,2);  
    \coordinate (D) at (-2,2);   
    \draw[thick] (A)--(B)--(C)--(D)--cycle;
    
    \coordinate (E) at (-4,-2);
    \coordinate (F) at (4,-2);
    \draw[thick] (A)--(B)--(F)--(E)--cycle;
    
    \coordinate (G) at (-5,-4);
    \coordinate (H) at (5,-4);
    \draw[thick] (E)--(F)--(H)--(G)--cycle;

    \coordinate (S) at ($(D)!0.3!(C)$);
    \coordinate (N1) at ($(S)+(0.5,-0.5)$);
    \coordinate (N2) at ($(N1)+(-0.5,-0.5)$);
    \coordinate (N3) at ($(N2)+(0.5,-0.5)$);
    \path let \p1 = (N3) in coordinate (I) at (\x1-\y1,0);
    \draw[thick] (S)--(N1)--(N2)--(N3)--(I);
    \foreach \p in {S,N1,N2,N3,I}
    \filldraw[thick,fill=white] (\p) circle (5pt);

    \coordinate (T) at ($(E)!0.6!(F)$);
    \coordinate (M1) at ($(T)+(0.5,-0.5)$);
    \coordinate (M2) at ($(M1)+(-0.5,-0.5)$);
    \coordinate (M3) at ($(M2)+(0.5,-0.5)$);
    \coordinate (M4) at ($(M3)+(0.5,-0.5)$);
    \draw[thick] (T)--(M1)--(M2)--(M3)--(M4);
    \foreach \p in {T,M1,M2,M3,M4}
    \filldraw[thick,fill=white] (\p) circle (5pt);

    \draw[<->,thick] (6.5,0.1) -- (6.5,1.9)
    node[midway,right] {$R_i+1$};
    \draw[<->,thick] (6.5,-3.9) -- (6.5,-2.1)
    node[midway,right] {$R_{i+1}+1$};
    
    \node[right] at (2.2,2) {$n_i$};     
    \node[right] at (4.2,-2) {$n_{i+1}$}; 
    
    \end{tikzpicture}
    \caption{Tree with uncolored finite paths}
    \label{figure4}
    \end{figure}
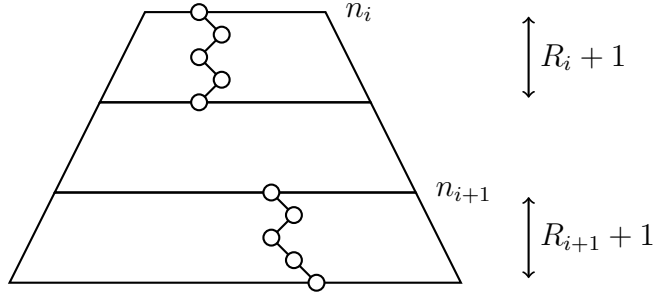

    We now estimate the probability of $A_i$. Note that, at level $n_i$, there are $2^{n_i}$ choices for the starting vertex of a path. From a given vertex there are $2^{R_i+1}$ choices for a path of length $R_i+1$. By the union bound and using the independence assumption, we have
    \begin{equation} \label{prob_product}
        \bP(A_i) \le 2^{n_i+R_i+1} \prod_{j=0}^{R_i} \Big( 1 - \frac{1}{2^{n_i+j}} \Big)^{\lfloor L \cdot 2^{n_i+j}\rfloor} \le 2^{n_i+R_i+1} \frac{1}{e^{1012\cdot (R_i + 1)}}.
\end{equation} We may choose $R_i=n_i$ for all $i\in\N$. By the first Borel--Cantelli lemma it follows that $\bP(\limsup_{i\to \infty}A_i)=0,$ and the existence of an infinite uncolored path also happens with probability zero. 
\\

    We now briefly describe how the previous arguments need to be modified in the case $X_N=\{q_N x\}$, as in the proof of Theorem~\ref{thm1}. Note that the events $\{q_M x\} \notin I_{n_1,k_1}$ and $\{q_N x\} \notin I_{n_2,k_2}$ are not independent, and one cannot estimate the probability by a product as in~\eqref{prob_product}.

To make these events almost independent for $M$ and $N$ coming from different levels, we split the sequence $\{q_N x\}$ into consecutive blocks $\Delta_n, \Delta_n', \Delta_{n+1}, \Delta_{n+1}', \ldots$, where the $\Delta_n$ are the main blocks containing $\floor{L2^n}$ points each, and the $\Delta_n'$ are buffer blocks containing $o(L2^n)$ points. Within a fixed level $n$, the distribution of the points $\{q_N x\}$ can be analyzed by standard Fourier-analytic methods together with the second moment method. The buffer blocks $\Delta_n'$ can be deleted from the sequence because they do not essentially affect the number of uncolored intervals, and the covering property of the original sequence clearly follows from the covering of a sparser subsequence.

In the resulting sequence $(a_N)_{N\in \bN}$, the points $\{a_M x\}$ and $\{a_N x\}$ from two different blocks $\Delta_n$ and $\Delta_{n+j}$ are sufficiently close to being independent because the ratio $\mfrac{a_N}{a_M}$ is large. We then bound the tail probability by taking a product of $R+1$ second moments and applying the Chebyshev inequality.

We should mention that a similar trick of deleting points was crucial in~\cite{eduard} for the bound on the maximal gap. In that case, however, almost all points were deleted in order to create sufficient independence, resulting in the loss of a power of $\log N$ in the final bound --- this would correspond in the above to taking each mass block $\Delta_n$ to contain only a single element. In contrast, in our approach we may delete only a small proportion of points, since the variables in the underlying counting problem have different scales and the counting problem can be solved inductively.

    \subsection*{The non-covering case} 

    For the non-covering case, the argument proceeds by induction using a variant of a branching random walk. We fix an initial uncolored interval $I_{n_0,k_0}$ such that $I_{n_0,k_0-1}$ and $I_{n_0,k_0+1}$ are also uncolored. We then count the number of ``potential'' infinite thick uncolored paths emanating from $I_{n_0,k_0}$. This yields a random walk in the following sense: whenever a descendant of $I_{n_0,k_0}$ is not colored (together with its left and right neighbours), the number of potential paths increases by $1$, since the path may continue through either of its children. Conversely, whenever a descendant of $I_{n_0,k_0}$ (or one of its left or right neighbours) gets colored, the number of potential paths decreases by $1$.

We aim to show that almost surely the colored tree has an infinite thick uncolored path. We fix some $n_0 \gs 1$ and consider the paths starting at level $n_0.$ For $n \gs n_0,$ let $\cI_n$ be the set of vertices at level $n$ such that all of their ancestors at levels $n-1, \ldots, n_0$   are uncolored and have uncolored left and right neighbors.  

Let 
    \[
        I(\cI_n) := \bigcup_{k: I_{n,k} \in \cI_n} I_{n,k}
    \] be the union of these vertices. Note that we clearly have $I(\cI_n) \subseteq I(\cI_{n-1}) \subseteq \ldots \subseteq I(\cI_{n_0}) = [0, 1)$. Let  $N_n := \floor{L 2^0} + \ldots + \floor{L 2^n}$ be the number of points of $(X_N)_{N\in\bN}$ that have been placed up to the $n$-th level of the tree, and define the event
    \[
        B_n := \big\{ \# \cI_n \ge (1.2)^n \big\}.
    \] Clearly, an infinite thick uncolored path starting at level $n_0$ exists if the event
    \[
        \bigcap_{n \ge n_0} B_n \quad \text{takes place.}
    \] It will be sufficient to show that
    \[
        \bP\Big( \bigcap_{n \ge n_0} B_n \Big) \ge \prod_{n \ge n_0} \Big( 1 - \frac{1}{L \cdot (1.1)^n} \Big);
    \] since the RHS of the last inequality can be made arbitrarily close to $1$ by choosing sufficiently large $n_0$, this will yield the existence of an infinite thick uncolored path almost surely. 
    
    First, note that
    \[
        \bP\big( B_n \cap B_{n-1} \cap \ldots \cap B_{n_0} \big) = \bP \big( B_n \big| B_{n-1} \cap \ldots \cap B_{n_0} \big) \cdot \bP(B_{n-1} | B_{n-2} \cap \ldots \cap B_{n_0}) \cdot \ldots \cdot \bP(B_{n_0}).
    \] We have $\bP(B_{n_0})=1$ trivially since $2^{n_0} > (1.2)^{n_0}$. Next, assume that for all $n_0 \le r \le n-1$ the inequality
    \[
        \bP \Big( B_r \Big| \bigcap_{n_0 \le i \le r-1} B_i \Big) \ge 1 - \frac{1}{L \cdot (1.1)^r}
    \] has been proven. It thus remains to show that
    \[
        \bP \Big( B_n \Big| \bigcap_{n_0 \le r \le n-1} B_r \Big) \ge 1 - \frac{1}{L \cdot (1.1)^n}.
    \]

    To show this, we introduce the counting function for the number of points falling into $I(\cI_n)$:
    \[
        C_n := C_n(\cI_n) = \#\big\{ N_{n-1} < N \le N_n : X_N \in I(\cI_n) \big\}.
    \]

    Since every new point $X_N$ with $N_{n-1} < N \le N_n$ can remove at most most three consecutive intervals from consideration, we have for the number of survived intervals on the next level the inequality
    \begin{equation} \label{rw_relation}
        \# \cI_{n+1} \ge 2 \big( \# \cI_n - 3 C_n \big).
    \end{equation} Further, let $\cP_1, \ldots, \cP_T$ be all possible collections of thick (finite) paths from $n_0$ to $n$ such that for each collection  
    \[
        \cP_i = \big( \cI_{n_0}^{(i)}, \ldots, \cI_n^{(i)} \big)
    \] one has $\# \cI_r^{(i)} \ge (1.2)^r$ for all $n_0 \le r \le n$. Then clearly we can partition $B_{n_0}\cap \ldots \cap B_{n-1}$ into a disjoint union 
    \[
        \bigcap_{n_0 \le r \le n-1} B_r = V_1  \cup \ldots  \cup V_T,
    \] where $V_i$ is the event that all thick paths from $\cP_i$ (and only them) remained uncolored. Thus,
    \[
         \bP \big( B_n  |  B_{n_0} \cap \ldots \cap B_{n-1} \big) =\bP \big( B_n  | V_1\cup \ldots \cup V_T \big) \gs  \min_{1\le i \le T} \bP \big( B_n \big| V_i \big),
    \] and therefore, it is enough to show that
    \[
        \bP \big( B_n \big| V \big) \ge 1 - \frac{1}{L \cdot (1.1)^n}
    \] for a fixed event $V$ corresponding to some collection of paths $\cP$ satisfying the conditions above.
    
    Note that conditioned on $V$, the random variable 
    $\cI_n$ is deterministic. We compute the (conditional on $V$) expectation of $C_n$. By the independence between the levels we have $\bE[C_n | V] = \bE[C_n]$ and so
    \begin{equation} \label{cond_exp_rand}
        \bE[C_n | V] = (N_n - N_{n-1}) \frac{\# \cI_n}{2^n} \ge L \cdot (1.1)^n.
    \end{equation}
    Similarly, for the variance we get
    \begin{multline*}
        \text{Var}[C_n | V] = \bE[C_n^2 |V] - \bE[C_n | V]^2 =
        \sum_{N_{n-1} < M,N \le N_n} \bP \Big( X_M \in I(\cI_n) \land X_N \in I(\cI_n) \big| V \Big) - \\
        \bigg( \sum_{N_{n-1} < N \le N_n} \bP \big( X_N \in I(\cI_n)  | V \big) \bigg)^2 \le \bE[C_n | V].
    \end{multline*} Thus, by Chebyshev's inequality,
    \begin{equation} \label{Chebyshev_bound}
        \bP \Big( \big| C_n - \bE[C_n | V] \big| > \bE[C_n | V]  \Big| V  \Big) \le \frac{1}{\bE[C_n | V]}. 
    \end{equation} 
    Assume that $L \le \mfrac{1}{100}$ and that $C_n \le 2\bE[C_n | V]$. Then, from~\eqref{rw_relation} and~\eqref{cond_exp_rand}, we get
    \[
        \# \cI_{n+1} \ge 2 \big(  \# \cI_n - 6 \bE[C_n | V] \big) \ge 2 \Big( \# \cI_n - \frac{6\floor{L 2^n}}{2^n} \# \cI_n \Big) \ge (1.2)^{n+1}.
    \] Thus,
    \begin{align*}
        \bP \big( B_n \big| V \big) & \ge \bP \big( C_n > 2\bE[C_n | V] \ \big| \ V \big)   \\
        & \gs 1 - \bP \Big( \big| C_n - \bE[C_n | V] \big| > \bE[C_n | V] \ \Big| \ V \Big) \ge 1 - \frac{1}{\bE[C_n | V]} \ge 1 - \frac{1}{L (1.1)^n}
    \end{align*} as desired. Choosing $L = \mfrac{1}{2026} + \eps$ concludes the proof. \\

    \medskip

    We finish this section by  discussing the main difficulty arising when one tries to adapt the above argument to the sequence $X_N=\{q_N x\}$ for the proof of Theorem \ref{thm2}. In that case,  the  points $X_N$ both on different levels as well as within one level are not independent. To reduce the dependence between points placed in different levels, we create a sufficiently large separation between the points $q_M$ and $q_N$ coming from different levels. The sequence is again split into main and buffer blocks $\Delta_n, \Delta_n', \Delta_{n+1}, \Delta_{n+1}', \ldots$, and, in order for the gap condition in Theorem~\ref{thm2} to be satisfied, we choose $\Delta_n'$ as large as possible.

    However, in contrast to the covering case, one cannot simply delete the buffer blocks $\Delta_n'$. Instead, we allocate the points from $\Delta_n'$ in the worst possible way within the surviving set $\cI_n$. This forces us to keep $\cI_n$ relatively large for the above argument to work, i.e. $\# \cI_n \ge 2^{(1-\eps)n}$. This constraint can be enforced by taking the parameter $L = L(\eps)$ sufficiently small.

    %SECTION 3: Upper bound
    
    %%%%%%%%%%%%%%%%%%%%%%%%%%%%%%%%%%%%%%%%%%%%%%%%%%%%%%%

    \section{The covering case} \label{sec_upper}
    
    In this section, we prove Theorem~\ref{thm1} using the framework developed in Section~\ref{sec_binary}, combined with estimates for higher moments of suitable exponential sums arising from $C_c^{\infty} (\bR)$ functions supported on the corresponding intervals (vertices of the tree). We conclude the section with a short proof of Corollary~\ref{cor1}.

    Throughout the proofs, we will make use of a standard smoothing argument. The indicator functions of relevant intervals will be replaced by suitable functions in $C_c^\infty(\bR)$ with sufficiently rapid Fourier decay. We will then compute moments using the corresponding Fourier expansions. The following lemma introduces the cutoffs that we will use throughout the next three sections.
    
    \begin{lemma}\label{smoothing}
    There exist real-valued, even, non-negative functions $g^{+}, g^{-} \in C_c^\infty(\bR)$ such that
    \begin{enumerate}
    \item $g^{\pm}(x) \le 1 \ \forall x \in \bR$;
    \item $\supp(g^-)\subset[-1,1]$ and $\widehat{g}^{-}(0)>0$;
    \item $\supp(g^+)\subset\big[-\frac{3}{2},\frac{3}{2}\big]$ and $g^+(x) = 1$ for $|x|\le \frac{10}{9}$; \ $2 < |\widehat{g}^{+}(0)| < \mfrac{5}{2}$;
    \item $\widehat g^\pm$ are real-valued and even on $\bR$;
    \item there exist constants $C,c>0$ such that for all $|\xi|\ge 1$,
    \begin{equation}\label{Fourier_bound}
    |\widehat g^\pm(\xi)| \le
    \min\!\Big(\frac{1}{|\xi|},\, C\exp\!\big(-c |\xi|^{1/2} \big)\Big).
    \end{equation}
    \end{enumerate}
    \end{lemma}
    
    The construction of such functions is standard; see, for instance, \cite{TLAS22} or~\cite[Chapter~1]{Karatsuba}. As an immediate consequence of~\eqref{Fourier_bound}, we obtain the tail estimate
    \begin{equation} \label{tail_decay}
    \sum_{k \ge N(\log N)^{100}} \widehat{g}^{\pm}\!\left(\frac{k}{N}\right) \ll N^{-100}
    \qquad (N\to\infty).
    \end{equation}

    \subsection*{Borel-Cantelli setup}
    
    We use the same approximation of a ``bad event'' as in the covering case argument in Section~\ref{sec_binary}. Suppose that, for some sufficiently large $L>1$, there exists a point $\delta_0\in[0,1]$ such that
    \[
    \|q_N x-\delta_0\|>\frac{L}{N}\qquad \text{for all } N>N_0.
    \]
    Then, setting $N_n = \lfloor \frac{L}{2}2^0 \rfloor + \ldots + \lfloor \frac{L}{2}2^n\rfloor,$ the tree $\mathcal{T}_2\big((\{q_N x\})_{N\in\bN}, \NN\big)$ contains an infinite uncolored path (the one corresponding to~$\delta_0$). We next estimate the $\mu$-measure of the set of $x\in[0,1]$ for which this can occur, where $\mu$ is any fixed measure as in the statement of Theorem~\ref{thm1}.

    As in the previous section, let $(n_i)_{i\in\N}$ be a sequence of positive integers satisfying $100\le n_1<n_2<n_3<\cdots$, representing the starting levels of the finite uncolored path segments, and let $(R_i)_{i\in\N}$ be a sequence of positive integers with $100\le R_1<R_2<R_3<\cdots$, representing their respective lengths. We also require that $n_{i+1}>n_i+R_i$ and that $100n_i\le R_i\le 200n_i$. Define
    \begin{multline}\label{event}
    A_i := A\big(n_i,R_i,\mathcal{T}_2\big((\{q_N x\})_{N\in\bN},\NN\big)\big) \\
    = \big\{ x\in[0,1] : \text{there exists an uncolored path of length $R_i+1$ starting at level $n_i$} \big\}.
    \end{multline}

    Then our goal is to show that
    \[
    \mu\Bigl(\limsup_{i\to\infty} A_i\Bigr)=0.
    \]
    We first show that Theorem~\ref{thm1} can be reduced to the following proposition:
    
    \begin{prop}\label{main_prop_2}
    Let $n,R\ge 100$ be integers such that $100n\le R\le 200n$, let $(q_N)_{N\in\N}$ be a real-valued lacunary sequence satisfying $\frac{q_{N+1}}{q_N}> 10$ for all $N\in\N$, and let $\mu$ be any measure on $[0,1]$ satisfying~\eqref{rajchman_cond}. \par Further, set $N_n= 1000\cdot 2^0 + \dots + 1000\cdot 2^n$ and assume that the tree $\mathcal{T}_2\bigl((\{q_N x\})_{N\in\bN}, \NN\bigr)$ satisfies the following property: for all positive integers $M<N$ corresponding to different levels of $\mathcal{T}_2$ (i.e., for some $n$ one has $M \le N_n < N)$,
    we have
    \[
    \frac{q_N}{q_M} \ge N^{100}.
    \]
    
    Let $A(n,R,\mathcal{T}_2( \{q_N x\},\NN))$ be as in~\eqref{event}. Then, for all sufficiently large $n$ and $R$ satisfying the above restrictions, we have
    \[
    \mu\big(A(n,R)\big) \le \frac{1}{1.01^R}.
    \]
    \end{prop}
    
    First, we need to modify a general lacunary sequence $(q_n)_{n\in\N}$, as in Theorem~\ref{thm1}, so that it satisfies the hypotheses of Proposition~\ref{main_prop_2}. We do this by discarding a fixed proportion of terms from the sequence $(q_n)_{n\in\N}$. The construction proceeds in two steps.
    
    In the first step, to ensure the condition $\frac{q_{N+1}}{q_N}> 10$, we replace the original lacunary sequence $(q_n)_{n\in\N}$ by a new sequence $(a_N)_{N\in\N}$ defined by
    \[
    a_N := q_{N \big(\floor{\frac{\log 10}{\log r}} + 1\big)}.
    \] This partitions the sequence $(q_n)_{n\in\N}$ into consecutive blocks of length $\lfloor \frac{\log 10}{\log r}\rfloor + 1$ and removes all but one element from each block. One can then verify directly that the resulting sequence $(a_N)_{N\in\N}$ satisfies $\frac{a_{N+1}}{a_N}> 10$ for all $N\in\N$. Note that if $r>10$, no elements need to be removed.

    Second, to ensure the condition $\frac{q_N}{q_M}\ge N^{100}$ whenever
    $M \le N_n < N$ for some $n$, we construct a third sequence $(b_N)_{N\in\N}$ by removing blocks of consecutive elements from $(a_N)_{N\in\N}$ as follows.
    Suppose that the first $N_n$ elements of the sequence $(b_N)_{N\in\N}$ have already been constructed from the first $N_n+Y_n$ elements of $(a_N)_{N\in\N}$.  We take the next unused $N_{n+1}-N_n$ elements of $(a_N)_{N\in\N}$ and append them to $(b_N)_{N\in\N}$. We then skip the next $\lfloor 100\,\frac{\log N_{n+1}}{\log 10}\rfloor$ elements of $(a_N)_{N\in\N}$. Repeating this procedure inductively, one checks directly that $(b_N)_{N\in\N}$ satisfies both conditions in Proposition~\ref{main_prop_2}. Note that, in this second step, once $n$ is sufficiently large we remove only a very small proportion of terms.

    Now consider the tree $\cT_2 ( (\{b_N x\})_{N\in\bN},\NN)$ and two sequences of positive integers
    $100 \le n_1 < n_2 < n_3 < \ldots$ and $100 \le R_1 < R_2 < R_3 < \ldots$
    satisfying $n_{i+1} > n_i + R_i+1$ and $100 n_i \le R_i \le 200 n_i$. By Proposition~\ref{main_prop_2},
    \[
    \mu(A_i) \le \frac{1}{1.01^{R_i}},
    \]
    where $A_i$ is defined by~\eqref{event} with $\NN$ as in the hypothesis. Then, by the Borel--Cantelli lemma,
    \[
    \mu \big( \limsup_{i \to \infty} A_i \big) = 0.
    \]
    Thus, for $\mu$-almost all $x \in [0,1]$, there is no infinite uncolored path in  $\cT_2 ( (\{b_N x\})_{N\in\bN},\NN)$.

    Consequently, for $\mu$-almost all $x \in [0,1]$, one has
    \[
    \norm{b_N x - \delta} < \frac{2000}{N} \qquad \text{for infinitely many } N\gs 1,
    \]
    uniformly in $\delta \in [0,1]$. Reverting to the original indices, one can verify directly that
    \[
    \norm{q_N x - \delta} < \frac{4000 (\floor{\frac{\ln 10}{\ln r}} + 1)}{N}
    \qquad \text{for infinitely many } N \gs 1.
    \]
    
    \vspace{1ex}

    \subsection*{Proof of Proposition~\ref{main_prop_2}: Fourier analysis}
    
    Let $N_n$ be as in the hypothesis. Note that the event $A=A(n, R, \cT_2(  \{ q_N x\}, \NN))$ can be written as
    \begin{multline*}
        \Big\{ x \in [0, 1): \text{there exist nested intervals $I_n \subset \ldots \subset I_{n+R}$ of the form} \\
        I_j := I_{j, k_j} = \Big[ \frac{k_j}{2^j}, \frac{k_j+1}{2^j} \Big)
        \text{ such that} \sum_{N_{j-1} < N \le N_j}\hspace{-4mm} \Id \big( \{ q_N x \} \in I_j \big) = 0 \text{ for all } n \le j \le n+R \Big\}.
    \end{multline*}
    
    By the union bound,  
    \begin{equation} \label{product}
        \mu(A) \le \max_{I_n \subset \ldots \subset I_{n+R}} 2^{n+R+1} \mu\Big( \bigcap_{j=n}^{n+R} A_j (I_j) \Big),
    \end{equation} where
    \[
        A_j (I_j) := \Big\{ x \in [0, 1): \sum_{N_{j-1} < N \le N_j} \Id \big( \{ q_N x \} \in I_j \big) = 0 \Big\}.
    \]

    Next, we replace the indicator function of each $I_j$ by $g^{-}(x)$ from Lemma~\ref{smoothing}. We obtain
    \[
        \mu(A_j (I_j)) \le \mu\bigg( \Big\{ x \in [0, 1): \mathop{  \mathop{\sum\sum}_{ N_{j-1} < N \le N_j  }   }_{h_j \in \bZ} g^{-} \Big( \frac{q_N x - t_j + h_j}{2^{-j-1}}  \Big) = 0 \Big\}  \bigg),
    \] where $t_j$ is the center of $I_j$. By Poisson summation,
    \[
        \mathop{  \mathop{\sum\sum}_{ N_{j-1} < N \le N_j  }   }_{h_j \in \bZ} g^{-} \Big( \frac{q_N x - t_j + h_j}{2^{-j-1}}  \Big) = \frac{1}{2^{j+1}} \mathop{  \mathop{\sum\sum}_{ N_{j-1} < N \le N_j  }   }_{h_j \in \bZ} \widehat{g}^{-} \Big( \frac{h_j}{2^{j+1}} \Big) e \big( h_j q_N x - h_j t_j \big). 
    \] Next, using~\eqref{tail_decay}, we can write
    \begin{multline*}
        \frac{1}{2^{j+1}} \mathop{  \mathop{\sum\sum}_{ N_{j-1} < N \le N_j } }_{h_j \in \bZ} \widehat{g}^{-} \Big( \frac{h_j}{2^{j+1}} \Big) e \big( h_j q_N x - h_j t_j \big) =  
        \frac{N_j - N_{j-1}}{2^{j+1}} + \\ 
        + \frac{1}{2^{j+1}} \mathop{ \mathop{ \sum\sum}_{N_{j-1} < N \le N_j}}_{1 \le |h_j| \le 2^{2j} } \widehat{g}^{-} \Big( \frac{h_j}{2^{j+1}} \Big) e \big( h_j q_N x - h_j t_j \big) + O\big( N_j^{-100} \big). 
    \end{multline*} Since $\widehat{g}^{-}$ is even, the contributions of positive and negative $h_j$ in the last double sum have the same absolute values. Note that $N_j - N_{j-1} \ge \frac{N_j}{2}$ by definition of $N_j$. Then for all sufficiently large $j$ we can write
    \[
        \mu\big(A_j (I_j)\big) \le \mu \bigg( \Big\{ x \in [0, 1): \bigg|    \mathop{\mathop{\sum\sum}_{N_{j-1} < N \le N_j}}_{ 1 \le h_j \le 2^{2j} } \widehat{g}^{-} \Big( \frac{h_j}{2^{j+1}} \Big) e \big( h_j q_N x - h_j t_j \big) \bigg| > \frac{N_j}{4} \Big\} \bigg).
    \] Next, using the last inequality, we can write
    \[
        \mu \bigg( \bigcap_{j=n}^{n+R} A_j (I_j) \bigg) \le 
        \mu \bigg( \Big\{ x\in[0, 1): \prod_{j=n}^{n+R} \bigg| \mathop{\mathop{\sum\sum}_{N_{j-1} < N \le N_j}}_{ 1 \le h_j \le 2^{2j} } \hspace{-3mm}\widehat{g}^{-} \Big( \frac{h_j}{2^{j+1}} \Big) e \big( h_j q_N x - h_j t_j \big) \bigg| > \prod_{j=n}^{n+R} \frac{N_j}{4} \Big\} \bigg).
    \] Then, applying Chebyshev inequality, we find
    \begin{equation*}
        \mu \Big( \bigcap_{j=n}^{n+R} A_j (I_j) \Big) \le  \Big( \prod_{j=n}^{n+R} \frac{4}{N_j} \Big)^2 \int_0^1 \prod_{j=n}^{n+R} \bigg| \mathop{\mathop{\sum\sum}_{N_{j-1} < N \le N_j}}_{ 1 \le h_j \le 2^{2j} }\hspace{-2mm} \widehat{g}^{-} \Big( \frac{h_j}{2^{j+1}} \Big) e \big( h_j q_N x - h_j t_j \big) \bigg|^2 d\mu(x).
    \end{equation*} Opening the square and using the property~\eqref{rajchman_cond} of the measure $\mu$, we get
    \begin{equation} \label{orthogonality}
        \mu \Big( \bigcap_{j=n}^{n+R} A_j (I_j) \Big) \le 4^{2(R+1)} \Big( \prod_{j=n}^{n+R} \frac{1}{N_j^2} \Big) S(n, R),
    \end{equation} where

    \[
    S(n, R) :=  \sum_{\bm{\ell}, \bm{m}, \bm{h},\bm{k}} c\left(\bm{h}\right)c\left(\bm{k}\right)
    \left(1 +  J(\bm{\ell}, \bm{m}, \bm{h}, \bm{k})\right)^{-\eta},
    \] with: 
    \begin{itemize}
    \item $\bm{\ell} = (\ell_{n},\ldots,\ell_{n+R}),\bm{m} = (m_{n},\ldots,m_{n+R})$ with all $\ell_j, m_j$ running over the integers in $\Delta_j := (N_{j-1},N_j]$;
    
    \item $\bm{h} = (h_{n},\ldots,h_{n+R}),\bm{k} = (k_{n},\ldots,k_{n+R})$ with
    $1 < h_j, k_{j} \leq 2^{2j}$ for all $n \leq j \leq n+R$;
    
    \item $c\left(\bm{h}\right) := \prod\limits_{j = n}^{n+R} \min\left(\frac{2^{j+1}}{h_j},1\right)$; 
    
    \item $J(\bm{\ell}, \bm{m}, \bm{h}, \bm{k}) := \sum\limits_{j=n}^{n+R} \big( h_j q_{\ell_j} - k_j q_{m_j} \big)$.
    \end{itemize}
    
    \vspace{1ex}
    
    Below we will show that $S(n, R)$ satisfies the following bound:
    
    \begin{equation} \label{bound_S_}
            S(n, R) \le
            10 N_n^2 2^{2n} (\log 2^{2n})^2 30^R \prod_{j=n}^{n+R} N_j 2^j + 1.
    \end{equation} Combining 
    \eqref{product}, \eqref{orthogonality}, and \eqref{bound_S_}, we finally find
    \[
        \mu\big(A(n, R)\big) \le 2^{25+4n+R+1} \cdot 4^{2(R+1)} \cdot 30^{R+1} (\log 4^n)^2 \Big( \prod_{j=n}^{n+R} \frac{2^j}{N_j} \Big) \le (\log 4^n)^2 \frac{2^{25+4n} \cdot 960^{R+1}}{1000^{R+1}} \le \frac{1}{1.01^R}, 
    \] where for the last inequality we have also used the estimates $(\log 4^n)^2 \le 2^{0.1 n}$ and $n \le \frac{R}{100}$. This concludes the proof, assuming that~\eqref{bound_S_} holds.\\

    \subsection*{Proof of Proposition~\ref{main_prop_2}: Counting problem}

    In order to show \eqref{bound_S_}, we partition the sum $S(n, R)$ in $S_1,S_2$ where $S_1$ corresponds to the sum over ($\bm{h}$, $\bm{k}$, $\bm{\ell}$ $\bm{m}$) where
    \[
        \big|J(\bm{\ell}, \bm{m}, \bm{h}, \bm{k})\big| < q_{N_n},
    \] and $S_2$ corresponds to the remaining vectors. For $S_2$, we use the trivial bound of $c(\bm{h}) \leq 1$ and $\big|J(\bm{\ell}, \bm{m}, \bm{h}, \bm{k})\big| \geq q_{N_n}$
    to obtain (recall that $R \asymp n$ and $q_n \ge 10^n$)
    \[
        S_2 \leq \frac{1}{\left(q_{N_n}\right)^{\eta}}\sum_{\bm{\ell}, \bm{m}, \bm{h},\bm{k}} 1 \ll
        \frac{1}{10^{\eta N_n}} \exp \big(O(R \log N_{n+R})\big) \leq \frac{\exp(O(n^2))}{\exp(\eta 2^n)},
    \]
    showing $S_2 \leq 1$ for $n$ sufficiently large. Next, $S_1$ can be trivially bounded as
        \begin{equation} \label{S1bound}
            S_1 \le \mathop{\sum_{\bm{\ell}, \bm{m}, \bm{h},\bm{k}}}_{|J(\bm{\ell}, \bm{m}, \bm{h}, \bm{k}) | < q_{N_n}} c\left(\bm{h}\right)c\left(\bm{k}\right).
        \end{equation} 
    
    The following Lemma will be a crucial step in the evaluation of the last sum.
    
    \begin{lemma}\label{local_count}
    Let $j \geq 1$ be fixed and set
    $c(h) := \min(1, \frac{2^{j+1}}{h})$ for $h \in \bN$. Further, let $\Delta_j$ and $(q_n)_{n \in \bN}$ be as above. Then uniformly in $B$, we have
    \[
        \mathop{\sum_{\ell,m \in \Delta_j}
        \sum_{1 \leq h,k \leq 2^{2j}} }_{| k q_{m} - hq_{\ell}  - B| < \frac14 q_{N_{j-1}}}
        c(h)c(k) \leq 20\cdot N_j 2^j.
    \]
    \end{lemma}
   
    \begin{proof}  
 It suffices to examine the contribution of the terms with $m \ge \ell$; we will then only need to multiply by a factor of $2$ to derive the upper bound for the sum in question. \par   Let $r := m - \ell$. Then
    \[
    \big| kq_{r+\ell} - hq_{\ell} - B \big| < \frac{q_{N_{j-1}}}{4} \quad \Longrightarrow \quad \Big| k \frac{q_{r+\ell}}{q_{\ell}} - h - B_{\ell} \Big| < \frac{1}{4},
    \] where $B_{\ell} := \mfrac{B}{q_{\ell}}$. The aim is to prove that for any fixed $\ell \in \Delta_j$ and for any real $B_{\ell}$, we have
    \begin{equation} \label{the_aim}
    \sum_{r \leq N_j} \mathop{\sum_{1 \leq k,h \leq 2^{2j}}}_{| k \frac{q_{\ell+r}}{q_{\ell}} - h - B_{\ell}| < \frac{1}{4}}c(h)c(k) \leq 10 \cdot 2^j.
    \end{equation}
    The result will then follow after summing over all $\ell \in \Delta_j$. \par First, for $\ell \in \Delta_j$ write
    \[
    A_{\ell} := \Big\{0 \leq r \leq N_j: \exists h,k \in[1,2^{2j}] \text{ s.t. } \Big|k \frac{q_{\ell+r}}{q_{\ell}} - h  - B_{\ell} \Big| < \frac{1}{4} \Big\}.
    \]
   Note that for $r \in A_{\ell}$ one has
    \[
    \frac{B_{\ell}}{2^{2j}} \leq \frac{q_{\ell+r}}{q_{\ell}} \leq B_{\ell} + 2^{2j} + \frac14.
    \]
    When $B_{\ell} > 2^{2j} + \frac14$, one necessarily has
    \[
    \frac{B_{\ell}}{2^{2j}} \leq \frac{q_{\ell+r}}{q_{\ell}} \leq 2B_{\ell},
    \] and when $B_{\ell} < 2^{2j} + \frac14$, one has
    \[
    1 \leq \frac{q_{\ell+r}}{q_{\ell}} \leq 2^{2j+2}.
    \] Then, since $\frac{q_{\ell+r}}{q_{\ell}} \geq 10^r$, in both cases there are at most 
    \[  1 + \log_{10} 2^{2j+2} = 1 + \frac{2j + 2}{\log_{2}10}  \ls  \frac{2j}{3} + 2  \] 
    admissible values of $r$ in $A_\ell$. Let us fix one such $r$ and compute the sum
    \[
    \mathop{\sum_{1 \leq k,h \leq 2^{2j}}}_{| k \frac{q_{\ell+r}}{q_{\ell}} - h  - B_{\ell}| < \frac{1}{4}} c(h)c(k).
    \] Let $(k_0,h_0), (h_1,k_1), \ldots (h_T,k_T)$ denote all solutions to $|k \frac{q_{\ell+r}}{q_\ell}-h-B_\ell| < \frac14 $ with $k_0 < k_1 <\ldots < k_T$ (and consequently $h_0 < h_1 < \ldots < h_T$). Then for $i \geq 0$ we have
    \[
    \frac{q_{\ell+r}}{q_{\ell}}(k_i-k_0) = (h_i - h_0) + O(1),
    \] and since $k_i - k_0 \geq i$, we obtain $h_i \geq 10^ri - 1$. Consequently, for any $r \in A_{\ell}$ we get
    \begin{align*}
    \mathop{\sum_{1 \leq k,h \leq 2^{2j}}}_{| k \frac{q_{\ell+r}}{q_{\ell}} - h  - B_{\ell}| < \frac{1}{4}} \hspace{-5mm} c(h)c(k) &\leq 1 + \sum_{i \geq 1} c(i)c(9^ri) = 1 + \sum_{ i \le 2^{j+1} / 9^r} 1 + \sum_{2^{j+1} / 9^r < i \le 2^{j+1}} \frac{2^{j+1}}{9^r i} + \sum_{i > 2^{j+1}} \frac{2^{2(j+1)}}{9^r i^2} \\[-2ex]  &\ls 1 + 2^{j+1}\frac{\ln 9^r + 3}{9^r} .
    \end{align*}
    Summing over $r$ for $\ell$ fixed, since $A_\ell$ contains at most $\frac{2j}{3}+2$ values of $r$ we find
    \begin{multline*}
    \sum_{r \leq N_j} \mathop{\sum_{1 \leq k,h \leq 2^{2j}}}_{| k \frac{q_{\ell+r}}{q_{\ell}} - h - B_{\ell}| < \frac{1}{4}} \hspace{-4mm}c(h)c(k)  =    \sum_{r \in A_{\ell}}\hspace{-2mm} \mathop{\sum_{1 \leq k,h \leq 2^{2j}}}_{| k \frac{q_{\ell+r}}{q_{\ell}} - h - B_{\ell}| < \frac{1}{4}} \hspace{-3mm} c(h)c(k)  \ls  \sum_{r \in A_{\ell}} \Big(   1 + 2^{j+1}\frac{\ln 9^r + 3}{9^r}  \Big) \\  \ls
    \frac{2j}{3} + 2 +  \sum_{r \geq 0}  2^{j+1}\frac{\ln 9^r + 3}{9^r}  \ls 10 \cdot2^{j}. \qquad
    \end{multline*} 
    
    This proves the estimate in \eqref{the_aim} and thus the statement of the Lemma.
    \end{proof}

    \vspace{1ex}
    
    We can now find an upper bound for the sum appearing in the right hand side of \eqref{S1bound}.   
   The condition $|J(\bm{\ell}, \bm{m}, \bm{h}, \bm{k})| < q_{N_n}$ implies 
    \[ 
    q_{N_n} \geq \big|  k_{n+R} q_{m_{n+R}} - h_{n+R} q_{\ell_{n+R}} \big| - \Big| \sum_{j=n}^{n+R-1} \big( h_j q_{\ell_j} - k_j q_{m_j} \big)\Big|.
    \] Note that
    \[
    \Big| \sum_{j=n}^{n+R-1} \big( h_j q_{\ell_j} - k_j q_{m_j} \big)\Big| < \sum_{j=n}^{n+R-1} 2^{2j} q_{N_{j}} <  \sum_{j=n}^{n+R-1} \frac{2^{2j}}{N_{j+1}^{100}\cdots N_{n+R-1}^{100}}  q_{N_{n+R-1}} <  \frac{q_{N_{n+R-1}}}{2^{98(n+R)}}  \cdot 
    \] 
    Thus, we have necessarily
        \begin{equation} \label{restriction}
            \big|  k_{n+R} q_{m_{n+R}} - h_{n+R} q_{\ell_{n+R}}\big| < \frac{q_{N_{n+R-1}}}{2^{98(n+R)}} 
            + q_{N_n} < \frac{q_{N_{n+R-1}}}{4} \cdot 
        \end{equation} 
       We write $\bm{\ell}_j, n \leq j \leq n+R$ for the truncation $\bm{\ell}_j = (\ell_j,\ldots,\ell_{n+R})$ and define analogously $\bm{m}_j, \bm{h}_j,\bm{k}_j$. Letting $B_j = B_{j}(\bm{\ell}_j, \bm{m}_j, \bm{h}_j, \bm{k}_j) := \sum_{i = j}^{n+R} \big( h_i q_{\ell_i} - k_i q_{m_i} \big)$, we obtain by the same argumentation as above that any solution to $|J(\bm{\ell}, \bm{m}, \bm{h}, \bm{k})| < q_{N_n}$ satisfies for all $n < j \leq n+R$
    
    \begin{equation} \label{restriction_j}
            \big|k_{j} q_{j}- h_{j} q_{\ell_{j}} - B_{j+1}\big| < \frac{q_{N_{j-1}}}{4}.
    \end{equation} 
       Combining \eqref{S1bound} with \eqref{restriction_j}, we find 
    \begin{multline*}
         S_1  \ls  \sum_{\substack{\bm{\ell}, \bm{m}, \bm{h},\bm{k}, \\ |J(\bm{\ell}, \bm{m}, \bm{h}, \bm{k})| < q_{N_n}}} c(\bm{h}) c(\bm{k}) \leq
           \mathop{\mathop{\mathop{\sum\sum}_{\ell_{n+R}, m_{n+R} \in \Delta_{n+R}}}_{1 \leq h_{n+R}, k_{n+R} \leq 2^{2(n+R)}}}_{| k_{n+R} q_{m_{n+R}} - h_{n+R} q_{\ell_{n+R}} | <   \frac14 q_{N_{n+R-1}} }  \hspace{-4mm} c(h_{n+R}) c(k_{n+R}) \cdot \\
           \cdot  \mathop{\mathop{\mathop{\sum\sum}_{\ell_{n+R-1}, m_{\ell+R-1} \in \Delta_{n+R-1}}}_{1 \leq h_{n+R-1},k_{n+R-1} \leq 2^{2(n+R-1)}}}_{ | k_{n+R-1} q_{m_{n+R-1} - h_{n+R-1} q_{\ell_{n+R-1}} - B_{n+R} } | <  \frac14 q_{N_{n+r-1}}} \hspace{-20mm} c(h_{n+R-1})c(k_{n+R-1}) \cdot \ldots
            \\ 
            \ldots \cdot \mathop{\mathop{\mathop{\sum\sum}_{\ell_{n+1}, m_{\ell+1} \in \Delta_{n+1}}}_{1 \leq h_{n+1},k_{n+1} \leq 2^{2(n+1)}}}_{  | k_{n+1} q_{m_{n+1}} - h_{n+1} q_{\ell_{n+1}} - B_{n+2} | < \frac14 q_{N_n}}  \hspace{-5mm}c(h_{n+1})c(k_{n+1})
            \cdot 
            \mathop{\mathop{\sum\sum}_{\ell_{n}, m_{\ell} \in \Delta_{n}}}_{1 \leq h_{n},k_{n} \leq 2^{2n}}c(h_{n})c(k_{n}).
    \end{multline*}   
    The  innermost sums get estimated trivially, i.e.,
        \[  \mathop{\mathop{\sum\sum}_{\ell_{n}, m_{\ell} \in \Delta_{n}}}_{1 \leq h_{n},k_{n} \leq 2^{2n}}c(h_{n})c(k_{n})
        \leq 2  N_n^2 2^{2n}.
        \]
        Next, we use Lemma \ref{local_count} to deduce
        \[
           \mathop{\mathop{\mathop{\sum\sum}_{\ell_{n+1}, m_{\ell+1} \in \Delta_{n+1}}}_{1 \leq h_{n+1},k_{n+1} \leq 2^{2(n+1)}}}_{ | k_{n+1} q_{m_{n+1}} - h_{n+1} q_{\ell_{n+1}} - B_{n+2}| < \frac14 q_{N_n} }  \hspace{-5mm}c(h_{n+1})c(k_{n+1})  \leq 20 N_{n+1} 2^{n+1}.
        \]
    We continue to apply Lemma \ref{local_count} iteratively to the remaining next inner sums, which in the end shows
    \[
    S_1 \leq N_n^2 2^{2n} \prod_{j = n+1}^{n+R} 20 N_j 2^j
    \leq N_n^2 30^R 2^{2n} \prod_{j = n+1}^{n+R} N_j 2^j,
    \] and this completes the proof of~\eqref{bound_S_}.

    \subsection*{Proof of Corollary \ref{cor1}}
    
    Here we deduce Corollary~\ref{cor1} from Proposition~\ref{main_prop} below using Kaufman's measures. This is done exactly as in  previous works, and we include the proof for completeness.
    
    We briefly mention that Kaufman constructed probability measures supported on $\Bad$ that satisfy the following two key properties: \begin{itemize}
    \item[(i)] $\mu(J) \ls c_1|J|^\rho$ for all intervals $J \subseteq [0, 1],$ 
    \item [(ii)] $\widehat{\mu}(t) \ls c_2(1 + |t|)^{-\eta} $ for all $t\in \bR.$
    \end{itemize} Here $c_1, c_2>0$ are absolute constants, while $0 < \rho <1$ can be chosen arbitrarily close to $1$ and $\eta>0$ depends on the choice of $\rho$. For the precise values of the parameters appearing here, we refer to the refinement of Kaufman's construction due to Queffelec \& Ramar\'e \cite{QR}.   
    
    Property (i) above is the one which allows for the calculation of the Hausdorff dimension of $\bG$ in the result of Pollington and Velani via the \emph{Mass Distribution Principle} \cite{falconer}.

    Probability measures for which $\widehat{\mu}(t) \to 0$ as $|t|\to \infty$ are known as \emph{Rajchman measures} and have been studied extensively in the literature, see \cite{lyons}. It is worth mentioning that when a measure $\mu$ supported on a set $F\subseteq \bR$ satisfies property (ii) for some $ 0< \eta < 1/2,$ then the set $F$ has Hausdorff dimension $\dimH(F) \gs 2\eta,$ see \cite{mattila}.

    \begin{prop} \label{main_prop}
    Let $\alp \in \cK$ and $\gamma \in [0,1)$. Furthermore, let $\mu$ be a probability measure on $[0, 1)$ such that $\widehat{\mu}(t) \ll (1 + |t|)^{-\eta}$ for some $\eta > 0$. Then for $\mu$-almost all $\beta \in [0,1]$, one has
    \begin{equation} \label{ineq_in_main_prop}
        n \|n\alp -\gamma\| \|n\beta-\delta\| \ls \frac{c}{\log n} \qquad \text{ for infinitely many } n\gs 1 
    \end{equation} uniformly in $\delta \in [0,1]$.
    \end{prop} 
    
    \vspace{1ex}
    
    \textit{Proof of Corollary~\ref{cor1} from Proposition~\ref{main_prop}.} 
    
    Let $0 < \rho < 1$ be arbitrary, and let $\mu$ be one of Kaufman's measures such that (i) holds for $\rho$. Let $\mathcal{U}$ be the set of $\beta \in [0, 1)$ such that there exists $\delta \in [0, 1)$ such that the inequality~\eqref{ineq_in_main_prop} holds only for finitely many $n \in \bN$. Then by Proposition~\ref{main_prop} one has $\mu(\mathcal{U})=0$. Now take $\mathfrak{S} = \Bad \cap ([0, 1] \backslash \mathcal{U})$. Then $\mu(\mathfrak{S})=1$. Next, by the mass distribution principle~\cite{falconer} one has $\dimH(\mathfrak{S}) \gs \rho$. Since $\rho$ can be chosen arbitrarily close to~$1$, we conclude that $\dimH(\mathfrak{S}) = 1$. \\
    
    \textit{Proof of Proposition~\ref{main_prop} from Theorem~\ref{thm1}.} 
    
    Given $\alp \in \cK$ and $\gam \in [0, 1)$, it is shown in \cite[Lemma 2.1]{cz} that there exists a lacunary sequence $(q_n)_{n \in \N}$ of integers such that 
    \begin{equation} \label{first}
    \|q_n\alp - \gamma\| \ls \frac{8}{q_n} 
    \end{equation} and
    \[
    8^n < q_n \le 4^{6 \Lambda(\alpha) n}
    \] for all $n \in \N$. Here
    \[
    \Lambda(\alpha) = \sup \Big\{ \frac{q_k (\alpha)}{k}: k \ge 1 \Big\}.
    \] We note that $\Lambda(\alpha)$ equals the L\'evy constant $\frac{\pi^2}{12 \log 2}$ for Lebesgue almost all $\alpha \in [0, 1)$. In particular, $\Lambda(\alp)<\infty$ for all $\alpha \in \cK$ and $\cK$ is a set of full Lebesgue measure. Then, by Theorem~\ref{thm1}, for $\mu$-almost all $\beta \in [0, 1)$ one has
    \begin{equation} \label{second}
    \norm{q_n \beta - \delta} \le \frac{c}{n} \qquad \text{for infinitely many } n \gs 1
    \end{equation} for all $\delta \in [0, 1)$. Thus, multiplying~\eqref{first} and~\eqref{second}, we find
    \[
    \norm{q_n \alpha - \gamma} \norm{ q_n \beta - \delta} \le \frac{8}{q_n} \cdot \frac{c}{n} \le \frac{c_2}{q_n \log q_n}
    \] as desired. We note that in~\cite{ct} and~\cite{eduard} the  bound $\norm{q_n \beta - \delta} \ll \frac{(\log n)^{\kappa + \eps}}{n}$ was used instead of~\eqref{second}.

    %SECTION 4: Lower bound
    
    %%%%%%%%%%%%%%%%%%%%%%%%%%%%%%%%%%%%
    
    \section{The non-covering case} \label{sec_lower}
    
  We  now proceed with the proof of Theorem \ref{thm2}. We will employ the colored binary tree  associated to the sequence $\{q_nx\}$, with  the number of points that are placed at the first $n$ levels of the tree equal to
    \[   
        N_n := \floor{L \cdot 2^0} + \ldots + \floor{L \cdot 2^n}. 
    \] Here $L \in (0,1)$ will be a sufficiently small parameter,  as in Section~\ref{sec_binary}. Our aim is to prove that for a suitable choice of $L,$ the tree contains a thick infinite uncolored path for Lebesgue almost all $x\in [0,1)$. As alluded to earlier, we follow the strategy outlined in Section~\ref{sec_binary} for random i.i.d. points, with the appropriate modifications that are necessary to make up for the lack of independence.\par 
Fix a starting level $n_0$. For $n \gs n_0,$ we write $\cI_n$ for the set of vertices at level $n$ such that all of their ancestors at levels $n-1, \ldots, n_0$  are uncolored and have uncolored left and right neighbors, and    
    \[
        I(\cI_n) := \bigcup_{k: I_{n,k} \in \cI_n} I_{n,k}
    \] for the union of these vertices.  
Note that $\# \cI_{n_0} = 2^{n_0}$. We denote the intervals in $\cI_n$ by $J_{n,j}$ for $1 \le j \le \#\cI_n$. Analogously to the case of independent points, we aim to find a lower bound for the probability that all events \begin{equation} \label{survived_size}
    \# \cI_n \ge 2^{(1 - \frac{\eps}{2})n}, \qquad n\gs n_0
\end{equation}hold simultaneously, where $\eps>0$ is as in the statement of Theorem~\ref{thm2}. \par We may now introduce the main technique that we use to make up for the lack of independence of the points $\{q_n x\}.$  For each $n\gs 1,$ write $(N_{n-1},N_n] = \Delta_n' \cup \Delta_n,$ where $\Delta_n'$ and $\Delta_n$ are consecutive disjoint intervals of integers with \[ \# \Delta_n ' = \floor{L \cdot 2^{(1-\frac{\eps}{2})n}} \qquad (n\gs 1). \] With this partition of the interval $(N_{n-1}, N_n],$  the points $\{q_N x\}$ that are placed on the $n$-th level of the tree are split into a ``buffer block" of points $\{q_Nx\}, N \in \Delta_n'$ we shall ignore, and a ``main block" of points $\{q_Nx\}, N \in \Delta_n$ we will work with. \par By ignoring the points that come from the buffer blocks, we ensure that the following property is satisfied: whenever $M,N$ are indices of points placed on different levels of the tree and $\{q_Nx\}$ is in one of the main blocks, i.e. whenever $M < N_n \ls N$ and $N \in \Delta_n$ for some $n,$ then  by \eqref{gap_condition} we ensure that
\[
    \frac{q_N}{q_M} \ge \Big( 1 + \frac{1}{N^{1-\eps}} \Big)^{\frac{L}{10} \cdot N^{1-\eps/2}} \ge \exp \Big( \frac{L}{100} N^{1 - \frac{\eps}{2}} \cdot \frac{1}{N^{1-\eps}} \Big) \gg N^{100}.
\]
At the same time, instead of calculating a lower bound for the probability that all events $\#\cI_n \gs 2^{(1-\frac{\varepsilon}{2})n}, n\gs n_0$ hold simultaneously, it will be enough to bound from below the probability that at most $2L\# \cI_n$ intervals of $\cI_n$ get colored by the points $\{q_Nx\}, N \in \Delta_n.$ If this event happens, then the points $\{q_Nx\}, N \in \Delta_n$ from the main blocks exclude at most $6L\#\cI_n$ intervals of $\cI_n$ from belonging to some thick uncolored path. Taking into account the $\#\Delta'_n \ls L \#\cI_n$ points from the buffer block, there will be at most $9L\#\cI_n$ intervals of $\cI_n$ that cannot belong to a thick uncolored path, and then inductively
\begin{equation*} \label{growth_req}
    \#\cI_{n+1} \ge 2  (1 - 9L)\cdot \#\cI_n \ge 2(1-9L) \cdot 2^{(1-\frac{\eps}{2})n} \ge 2^{(1-\frac{\eps}{2})(n+1)},
\end{equation*} when $L>0$ is chosen sufficiently small with respect to $\varepsilon.$ Since this reduction can be made, in the rest of this section, whenever we speak about colored vertices of the binary tree, we will always implicitly mean colored by points in some of the main blocks; the points coming from the buffer blocks will be completely ignored since they are treated in a deterministic (worst case) way.\\

An arbitrary dyadic interval $I_{r,j} = [ \frac{j}{2^r}, \frac{j+1}{2^r})$ at level  $n_0 \le r \le n$ is colored by a point $\{q_Nx\}, N\in \Delta_n$ if and ony if
\begin{equation} \label{original_cond}
        \{ q_N x \} \in I_{r,j} = \Big[ \frac{j}{2^r}, \frac{j+1}{2^r} \Big) \ \ \Longleftrightarrow \ \ x \in \bigcup_{k=0}^{q_N} \Big[ \frac{k}{q_N} + \frac{j}{2^r q_N}, \frac{k}{q_N} + \frac{j+1}{2^r q_N} \Big).
\end{equation}
Because we seek a lower bound for the probability that a thick uncolored path exists, we can replace the sets appearing in \eqref{original_cond} by slightly enlarged sets; this will only decrease the probability. More precisely, we replace the condition in \eqref{original_cond} by 
        \begin{equation} \label{cond_on_x_}
            x \in J_{N,r,j} : = \bigcup_{k=0}^{q_N} \Big[ \frac{c_{k,1}}{2^{\ceil{\log q_N} + 2r}}, \frac{c_{k,2}}{2^{\ceil{\log q_N} + 2r}} \Big),
        \end{equation} where $c_{k,1}$ and $c_{k,2}$ are the largest and respectively  smallest integer  such that
        \[
            \frac{c_{k,1}}{2^{\ceil{\log q_N} + 2r}} \le \frac{k}{q_N} + \frac{j}{2^r q_N} \quad \text{resp.} \quad \frac{c_{k,2}}{2^{\ceil{\log q_N} + 2r}} \ge \frac{k}{q_N} + \frac{j+1}{2^r q_N} \cdot
        \] Therefore, if for $N_{r-1} < N \le N_r$ we define the function $X_{N, r, j} (x) := \Id_{J_{N,r,j}} (x)$ we clearly have 
        \[
            \Id \big( \{ q_N x\} \in I_{r,j} \big) \le X_{N,r,j} (x).
        \] In view of this, in order to obtain a thickened uncolored (finite) path from $n_0$ to $n$ consisting of intervals $(I_{r,j_r})_{r = n_0}^{n}$ it suffices that\footnote{for completeness, we define $j_r+t \pmod {2^r}$} 
        \begin{equation} \label{thick}
            \mathop{\mathop{\sum\sum}_{n_0 \le r \le n}}_{N\in \Delta_r} X_{N,r,j_r+t}(x) = 0 \qquad \text{ for } t = -1,0,1.
        \end{equation} The reason for considering these enlarged intervals with dyadic endpoints is that they allow for an easier calculation of the conditional probabilities that show up later in the proof.

        We now define the desired event at the level $n$ as
        \begin{multline*}
            B_n := \Big\{x \in [0, 1): \text{ the number of paths $(j_r)_{r = n_0}^{n}$ such that \eqref{thick} holds is} \geq 2^{(1- \tfrac{\varepsilon}{2})n} \Big\}.
        \end{multline*}

        Since all events $\#\cI_r \gs 2^{(1-\frac{\varepsilon}{2})r}, n_0 \ls r \ls n$ hold as long as the event $B_{n_0} \cap \ldots \cap B_n$ takes place, in the remaining part of the proof we will show   that
        \[
            \lambda(B_{n_0} \cap \ldots \cap B_n) \ge \prod_{r=n_0}^n \big( 1 - 2^{-\frac{\eps r}{4}} \big).
        \] The desired result will then follow by taking $n \to \infty$ and then $n_0 \to \infty$. Since for all $n\gs n_0$ we have
        \[
            \lambda(B_{n_0} \cap \ldots \cap B_n) = \lambda(B_n | B_{n-1} \cap \ldots \cap B_{n_0}) \cdot \lambda(B_{n-1} | B_{n-2} \cap \ldots \cap B_{n_0}) \cdot \ldots \cdot \lambda(B_{n_0}),
        \]   we assume by induction hypothesis that
        \[
            \lambda(B_r | B_{r-1} \cap \ldots \cap B_{n_0}) \ge 1 - 2^{\frac{\eps r}{4}} \qquad \text{for } n_0 \le r \le n-1
        \] (recall that for $r=n_0$ it is trivially true by the definition of $\cI_{n_0}$). Thus, we only need to show
        \[
            \lambda(B_n | B_{n-1} \cap \ldots \cap B_{n_0}) \ge 1 - 2^{-\frac{\eps n}{4}}.
        \]  As in the previous section, let $\cP_1,\ldots, \cP_T$ be all possible collections of paths $(j_r)_{r=n_0}^{n}$ with cardinality $\#\cP_i \gs 2^{(1-\eps/2)n}$ and define 
        \[
            V_i = \Big\{x \in [0, 1): \eqref{thick}  
            \text{ holds iff  } (j_r)_{r=n_0}^n \in \mathcal{P}_i\Big\} \qquad (1\ls i \ls T).
        \] Then we can decompose the event $B_{n-1} \cap \ldots \cap B_{n_0}$ into the disjoint union
        \[
            B_{n-1} \cap \ldots \cap B_{n_0} = V_1 \cup V_2 \cup \ldots \cup V_T.
        \] Since 
        \begin{equation} \label{fixed_collection}
            \lambda(B_n | V_1 \cup \ldots \cup V_T ) \ge \min_{1 \le i \le T} \lambda(B_n | V_i),
        \end{equation}
         we can work with a fixed collection $\cP : = \cP_i$ of paths of cardinality $\ge 2^{(1-\eps/2)n}$, the corresponding set of vertices $\cI_n$ of the same cardinality, and the corresponding event $V := V_i$. It will be enough to find a lower bound for $\lam(B_n|V).$ \\

         \medskip

        Let $g^{+} $ be the function from Lemma~\ref{smoothing}. Then the function $t \mapsto g^{+} ( 2^{n+1} ( t - t_{n,j} ) ) $ is a weighted indicator for $I_{n,j}$ (recall that $t_{n,j}$ is the center of $I_{n,j}$). By the properties of $g^{+}$, it is $\ge 1$ on an interval of size at least $\mfrac{10}{9}\,|I_{n,j}|$ and is supported on an interval of size at most $\mfrac{3}{2}\,|I_{n,j}|$. Let $\cI_n$ be the set of intervals on the level $n$ corresponding to the paths from $\cP$.  Define
        \[
              C_n := \mathop{\mathop{\sum\sum}_{  \,I_{n,j} \in \cI_n}}_{N \in \Delta_n} g^{+} \Big( 2^{n+1} \big( \{ q_N x\} - t_{n,j} \big) \Big).
        \] The function $C_n$ is a smoothened version of the counting function of points $\{q_Nx\}, N\in \Delta_n$ lying in $\cI_n.$ \par  We begin by computing  the conditional expectation of $  C_n $. The sigma-algebra generated by the random variables $\big\{ X_{N, r, j} : N\in \Delta_r, \ 0 \le j < 2^r, \ n_0 \leq r \leq n -1\big\}$ is clearly contained in the sigma-algebra generated by the sets 
        \begin{equation} \label{generic_set}
            D_{\ell} := \Big[ \frac{\ell}{2^{\lceil \log q_{N_{n-1}}\rceil+2r}},  \frac{\ell+1}{2^{\lceil \log q_{N_{n-1}}\rceil + 2r}} \Big), \quad 0 \leq \ell < 2^{\lceil \log q_{N_{n-1}}\rceil+2r}.
        \end{equation} 
     Let $D$ be an (arbitrary) union of some intervals of the form~\eqref{generic_set} (corresponding to $N_{n-1}$) such that $D \subseteq V$. Then 
        \[
            \bE [  C_n | D] = \frac{1}{\lambda(D)} \int_{D}  \mathop{\mathop{\sum\sum}_{  \,I_{n,j} \in \cI_n}}_{N \in \Delta_n}  g^{+} \Big( 2^{n+1} \big( \{ q_N x\} - t_{n,j} \big) \Big) dx. 
        \] Applying Poisson summation, we find
        \begin{align*}
            \bE [  C_n | D] &= \frac{1}{\lambda(D)} \int_{D}  \mathop{\mathop{\sum\sum}_{  \,I_{n,j} \in \cI_n}}_{N \in \Delta_n} \sum_{h \in \bZ} g^{+} \Big( 2^{n+1} \big( q_N x - t_{n,j} + h \big) \Big) dx  \\
            &= \frac{1}{\lambda(D)} \int_{D}  \mathop{\mathop{\sum\sum}_{  \,I_{n,j} \in \cI_n}}_{N \in \Delta_n} \frac{1}{2^{n+1}} \sum_{h \in \bZ} \widehat{g}^{+} \Big( \frac{h}{2^{n+1}} \Big) e\Big( h\big(q_N x - t_{n,j}\big) \Big) dx.
        \end{align*} Assume that $D = D_1 \cup \ldots \cup D_H$, where $D_1, \ldots, D_H$ are disjoint intervals of the form~\eqref{generic_set}, so that
        \[
            \lambda(D_i) = \frac{1}{2^{\lceil \log q_{N_{n-1}} \rceil + 2(n-1)}} \qquad \text{for all } 1 \le i \le H.
        \] Then
        \begin{equation} \label{integral_for_exp}
           \begin{split} \frac{1}{\lambda(D)} \int_{D} e\big( hq_N x \big) dx &= \frac{2^{\lceil \log q_{N_{n-1}} \rceil + 2(n-1)}}{H} \sum_{i=1}^H \int_{D_i} e\big( hq_N x \big) dx \\  &\ll 2^{\lceil \log q_{N_{n-1}} \rceil + 2(n-1)} \cdot \frac{1}{hq_N}. \end{split}
        \end{equation}
       Thus, using the decay rate of $|\widehat{g}^{+}(x)|$, we find
        \begin{multline} \label{expectation_asymptotics}
            \bE [  C_n | D ] = \widehat{g}^{+} (0) \frac{\# \cI_n \cdot (N_n - N_{n-1})}{2^{n+1}} + \\
            +O\bigg( \frac{1}{\lambda(D)}  \mathop{\mathop{\sum\sum}_{  I_{n,j} \in \cI_n}}_{N\in \Delta_n} \frac{1}{2^{n+1}} \sum_{0 < |h| \le 2^{2n}} \frac{1}{|h|} \Big| \int_{D} e\big( h q_N x \big) dx \Big| \bigg) + O\Big( \# \cI_n N_n \frac{1}{2^{100n}} \Big) = \\
            =\widehat{g}^{+} (0) \frac{\# \cI_n \cdot (N_n - N_{n-1})}{2^{n+1}} +
            O\bigg( q_{N_{n-1}} 2^{2n} \# \cI_n \sum_{ N \in \Delta_n} \sum_{0 < |h| \le 2^{2n}} \frac{1}{|h|} \cdot \frac{1}{|h| q_N} \bigg) + \\
            + O\Big( \# \cI_n N_n \frac{1}{2^{100n}} \Big) = \widehat{g}^{+} (0) \frac{\# \cI_n \cdot (N_n - N_{n-1})}{2^{n+1}} + O\Big( 2^{4n} \frac{q_{N_{n-1}}}{q_{N_{n-1}+1}} \Big) + O\Big( \# \cI_n N_n \frac{1}{2^{100n}} \Big) = \\
            \widehat{g}^{+} (0) \frac{\# \cI_n \cdot (N_n - N_{n-1})}{2^{n+1}} + O\big(2^{-90 n}\big).
        \end{multline} 
        In particular, for any set $D$ of the form~\eqref{generic_set} such that $D \subseteq V$ one has
        \[
            \bE \big[   C_n | V \big] = \bE \big[   C_n | D \big] + O\big( 2^{-90n} \big).
        \] Recall that in light of~\eqref{fixed_collection}, our goal is to get a lower bound for $\lambda \big(B_n  | V \big).$  By the discussion in the beginning of this section, if the event $B_n$ does not hold, then $C_n \gs 3L \#\cI_n,$ and thus also $C_n \gs 2 \bE[C_n|V]$ by~\eqref{expectation_asymptotics} and the properties of $g^{+}(x)$ (recall that $2 \le \widehat{g}^{+}(0) \le \mfrac{5}{2}$), which in turn implies 
        \[
            \big|    C_n   - \bE[  C_n | V] \big| > \bE[  C_n | V].
        \]
        It is thus enough to provide an upper bound for $\lambda\big(\big|   C_n - \bE[ C_n | V] \big| > \bE[ C_n | V]  \big|  V \big).$ \par Applying Chebyshev's inequality, we find 
        \begin{multline*}
            \lambda\Big(\big|   C_n - \bE[  C_n | V] \big| > \bE[  C_n | V] \ \Big| \ V \Big) \le \\
            \ls \frac{1}{\bE[  C_n | V]^2 \cdot \lambda(V)} \int_{V} \bigg( \mathop{\mathop{\sum\sum}_{  \,I_{n,j} \in \cI_n}}_{N \in \Delta_n}  g^{+} \Big( 2^{n+1} \big( \{ q_N x\} - t_{n,j} \big) \Big) - \bE[ C_n | V] \bigg)^2 dx \\
            = \frac{1}{\bE[  C_n | V]^2 \cdot \lambda(V)} \int_{V} \bigg( \mathop{\mathop{\sum\sum}_{  \,I_{n,j} \in \cI_n}}_{N \in \Delta_n }   g^{+} \Big( 2^{n+1} \big( \{ q_N x\} - t_{n,j} \big) \Big) \bigg)^2 dx - 1 \qquad \\
            = \frac{1}{\bE[  C_n | V]^2 \cdot \lambda(V)} \int_{V} \mathop{\mathop{\sum\sum}_{I_{n,j_1}, I_{n,j_2} \in \cI_n}}_{  M,N \in \Delta_n} g^{+} \Big( 2^{n+1} \big( \{ q_M x\} - t_{n,j_1} \big) \Big) g^{+} \Big( 2^{n+1} \big( \{ q_N x\} - t_{n,j_2} \big) \Big) dx - 1.
        \end{multline*}
        
        We first evaluate the contribution of  $M,N$ that are close to each other. Namely, when $|M-N| \le 2^{(1-3\eps/4)n}$ we use the following trick: since a given point $\{q_N x\}$ can belong to  the support of $t\mapsto g^+(2^{n+1}(t-t_{n,j}))$ for at most $2$ values of $j,$ we have
        \[
            \sum_{  I_{n,j} \in \cI_n} g^{+} \Big( 2^{n+1} \big( \{ q_N x\} - t_{n,j} \big) \Big) \le 2.
        \] Applying this inequality to the sum over $j_2$, we find
\begin{multline*}
            \frac{1}{\bE[  C_n| V]^2 \lambda(V)} \int_{V}  \mathop{\mathop{\sum\sum}_{I_{n,j_1}, I_{n,j_2} \in \cI_n}}_{ \substack{ M,N \in \Delta_n : |M-N| \le 2^{(1-3\eps/4)n}}} \hspace{-5mm} g^{+} \Big( 2^{n+1} \big( \{ a_M x\} - t_{n,j_1} \big) \Big)    
            g^{+} \Big( 2^{n+1} \big( \{ q_N x\} - t_{n,j_2} \big) \Big) dx  \\
             \le \frac{2}{\bE[  C_n | V]^2 \lambda(V)} \int_{V} \mathop{\mathop{\sum\sum}_{  I_{n,j_1} \in \cI_n}}_{\substack{  M,N \in \Delta_n \\ |M-N| \le 2^{(1-3\eps/4)n}}} \hspace{-2mm} g^{+} \Big( 2^{n+1} \big( \{ a_M x\} - t_{n,j_2} \big) \Big) dx  \\
           \le 2\frac{2^{(1-3\eps/4)n}}{\bE[  C_n | V]^2} \bE[  C_n | V] \ll \frac{2^{(1-3\eps/4)n}}{\# \cI_n} \ll \frac{2^{(1-3\eps/4)n}}{2^{(1-\eps/2)n}} \ll 2^{-\frac{\eps n}{4}}.  \qquad
        \end{multline*}

        It remains to evaluate the contribution from $M, N$ such that $|M-N| > 2^{(1-3\eps/4)n}$. Here we apply Poisson summation again:
        \begin{multline*}
            \frac{1}{\bE[  C_n | V]^2 \cdot \lambda(V)} \int_{V} \hspace{-6mm}\mathop{\mathop{\sum\sum}_{I_{n,j_1}, I_{n,j_2} \in \cI_n}}_{\substack{  M,N \in \Delta_n : |M-N| > 2^{(1-3\eps/4)n}}} \hspace{-7mm}  g^{+} \Big( 2^{n+1} \big( \{ a_M x\} - t_{n,j_1} \big) \Big)  
            g^{+} \Big( 2^{n+1} \big( \{ q_N x\} - t_{n,j_2} \big) \Big) dx - 1 \\ = \frac{1}{\bE[  C_n | V]^2 \cdot \lambda(V)} \int_{V}  \mathop{\mathop{\sum\sum}_{I_{n,j_1}, I_{n,j_2} \in \cI_n}}_{\substack{  M,N \in \Delta_n : |M-N| > 2^{(1-3\eps/4)n}}} \hspace{-2mm} 
            \frac{1}{2^{2(n+1)}} \sum_{h_1, h_2 \in \bZ}  \widehat{g}^{+} \Big( \frac{h_1}{2^{n+1}} \Big) \widehat{g}^{+} \Big( \frac{h_2}{2^{n+1}} \Big) \cdot \\ \vspace{-2mm} \cdot e\Big( h_1 \big( a_M x - t_{n, j_1} \big) + h_2 \big( q_N x - t_{n, j_2} \big) \Big) dx - 1.
        \end{multline*} Note that the contribution of the zero frequencies $h_1 = h_2 = 0$ can be bounded from above by
        \[
            \le \frac{1}{\bE[ C_n | V]^2} \# \cI_n^2 (N_n - N_{n-1})^2 \frac{1}{2^{2(n+1)}} \widehat{g}^{+} (0)^2 \le 1 + O \big( 2^{-90n} \big)
        \] following~\eqref{expectation_asymptotics}. Next, if $\max\{|h_1|,|h_2|\} > 2^{2n}$ for either $i = 1$ or $i = 2$, we can bound the full expression by $O(2^{-90n})$ from~\eqref{tail_decay}. Thus, it remains to evaluate
        \begin{multline} \label{expression_to_evaluate}
            \frac{1}{\bE[  C_n | V]^2 \cdot \lambda(V)} \int_{V} \sum_{I_{n,j_1}, I_{n,j_2} \in \cI_n} \sum_{\substack{  M,N \in \Delta_n \\ |M-N|>2^{(1-3\eps/4)n}}}   
            \frac{1}{2^{2(n+1)}} \sum_{\substack{|h_1|, |h_2| \le 2^{2n} \\ |h_1| + |h_2| \neq 0}} \widehat{g}^{+} \Big( \frac{h_1}{2^{n+1}} \Big) \widehat{g}^{+} \Big( \frac{h_2}{2^{n+1}} \Big) \cdot \\ \cdot e\Big( h_1 \big( a_M x - t_{n, j_1} \big) + h_2 \big( q_N x - t_{n, j_2} \big) \Big) dx.
        \end{multline}

        Note that, by the condition $|M-N| > 2^{(1-3\eps/4)n}$, we have (assuming without loss of generality that $N>M$)
        \[
            \frac{q_N}{q_M} \gg \Big( 1 + \frac{1}{N^{1-\eps}} \Big)^{2^{(1-3\eps/4)n}}
            \gg \exp\big( 2^{\frac{\eps n}{5}} \big)
            \gg 2^{100n}.
        \] On the other hand, if $h_2 \neq 0$, then $\max \Big|\frac{h_1}{h_2}\Big| \ll 2^{2n}$. This clearly implies that, for any pair $(h_1,h_2)$, we have
        \[
            |h_1 q_M + h_2 q_N| \gg \min(q_M,q_N) \gg q_{N_{n-1}+1}.
        \] Then, similarly to~\eqref{integral_for_exp}, we find
        \[
            \frac{1}{\lambda(V)} \int_{V} e\big( h_1 q_M x + h_2 q_N x \big) dx \ll q_{N_{n-1}} 2^{2n} \cdot \frac{1}{q_{N_{n-1}+1}}.
        \] Consequently, the expression~\eqref{expression_to_evaluate} does not exceed
        \[
            \frac{q_{N_{n-1}} 2^{2n}}{\bE[  C_n | V]^2} \cdot \# \cI_n^2 N_n^2 \frac{2^{4n}}{2^{2n}} \frac{1}{q_{N_{n-1}+1}} \ll 2^{-90n}.
        \] Thus,
        \[
            \lambda\Big(\big|  C_n - \bE[  C_n | V] \big| > \bE[  C_n | V] \ \Big| \ V \Big) \le 2^{-\frac{\eps n}{4}}. 
        \] We finally conclude that
        \[
           \prod_{n \ge n_0}  \lambda\Big( \#\cI_n \gs 2^{(1-\frac{\varepsilon}{2})n} \Big) \ge \prod_{n \ge n_0} \Big( 1 - 2^{-\frac{\eps n}{4}} \Big) > 0,
        \] and since the above product can be made arbitrarily close to $1$ by choosing $n_0$ sufficiently large, Theorem~\ref{thm2} follows.

    %SECTION 5: Fractals
    
    %%%%%%%%%%%%%%%%%%%%%%%%%%%%%%%%%%%
    
    \section{Intersection with fractal sets} \label{sec_fractal}
    
     In this section, we prove Theorem~\ref{thm4}: if $G$ is as in the hypothesis of Theorem \ref{thm4} and $E:=E((q_n),x,\nu)$ is the random set defined in \eqref{def_E_set}, we show that under the assumptions in each of  the cases $(1)$ and $(2),$ 
    \[
    \dimH (G \cap E) = \frac{1}{\nu} + \dimH (G) - 1
    \]
    holds for almost all $x \in [0,1)$.  We first give the proof for the lower bound, which applies to all \textit{analytic} sets $G$ (which include compact sets and, in particular, Ahlfors regular sets as special cases) and will be the main part of our work. The   corresponding upper bound  follows directly from the work of Bugeaud and Durand~\cite{BD} and applies to \textit{any} real-valued sequence $(q_n)_{n \in \bN}$, but it requires the set $G$ to be Ahlfors regular (the definition will be given below).   We conclude the section with a proof of the remaining parts of Corollary~\ref{cor2}.  \\

    \subsection*{Lower bound. Baire category argument} 
        For the proof of the lower bound in Theorem \ref{thm4}, we make use of the following Baire category argument from~\cite{KPX}: 

    \newcommand{\KPXBaire}{\cite[Lemma~3.4]{KPX}}
    
    \begin{lemma} [Khoshnevisan, Peres, Xiao \KPXBaire] \label{34}
    Let $\mathbb{P}$ be a Borel probability measure and let $E = E(x) \subseteq [0,1]$ be a random Borel measurable set with the following property: for every compact set $F$ with $\dimH(F) > s$, one has that 
    that \begin{equation}\label{nonempty_inter}
        \mathbb{P}[E(x) \cap F \neq \emptyset] = 1.
    \end{equation} Then for every analytic set $G$ we get
    \[
        \mathbb{P}[\dimH(E(x) \cap G) \geq \dimH(G) - s] = 1.
    \]
    \end{lemma} 

    Li, Shieh, and Xiao~\cite{lixi} used this result to determine the hitting probability $\bP (E \cap G \neq \emptyset)$ and the packing dimension of $E \cap G$ for the Dvoretzky random set $E$ with i.i.d.\ centers and an analytic set $G$. This argument was also crucial in the work of Bugeaud and Durand: in~\cite[Lemma~3.4]{BD} they extended the argument of Lemma~\ref{34} to Hausdorff measures and general gauge functions when $G$ is compact, and the same result would play over to our setup. Precisely, the proof of Theorem~\ref{thm4} is a combination of Lemma~\ref{34} with the second moment method.\\

    Armed with Lemma~\ref{34}, in order to prove that the lower bound $ \dimH (G \cap E) \gs \frac{1}{\nu} + \dimH (G) - 1$ holds almost surely, it simply suffices to show that for any compact $F \subseteq [0,1]$ with $\dimH(F)>1 - \frac{1}{\nu},$ we have $E \cap F \neq \emptyset$ almost surely. So, let $\sig := 1 - \frac{1}{\nu}$ and  $\varepsilon > 0$, and fix a compact set $F \subseteq [0,1]$ with $\dimH (F) \geq \sig   + \eps$. \par 
   We now describe in more details how the colored tree framework comes into the proof.  We consider the colored tree corresponding to the sequence of points $\{q_Nx\}$ and the sequence $(N_n)_{n\in\bN}$ defined by \begin{equation} \label{above_range}
        N_n =\lfloor L 2^0 \rfloor + \lfloor L 2^{\frac{1}{\nu}} \rfloor + \ldots + \lfloor L 2^{\frac{n}{\nu}} \rfloor,
    \end{equation} where $L>0$ will be some appropriately chosen parameter. This means that at level $n$ of the tree we place the next $\lfloor L 2^{\frac{n}{\nu}} \rfloor$ new points of the sequence, and a vertex at level $n$ is colored if it contains one of these points $\{q_Nx\}$   with $N_{n-1} < N \ls N_n.$ Observe the difference in the definition of $\NN$ compared to the previous sections, as now $N_n$ depends on $\nu.$ The reason for this choice of $(N_n)_{n\in\bN}$  is the following: choosing $L \le \frac{1}{2026}\,\nu^{-\nu}$ to ensure that  $\frac{1}{2^n} \le \frac{1}{10 N^{\nu}}$ for all $N_{n-1} < N \ls N_n,$  every $\delta$ corresponding to a path with infinitely many colored vertices will necessarily belong to the set $E$ defined in \eqref{def_E_set}. \par  In what follows, an interval/vertex $I_{n,k}$ at level $n$ of the tree will be called: \newline (i) an \textit{$F$-interval} or \textit{$F$-vertex} if it contains some point in $F,$ \newline (ii)  an \textit{$E$-interval} or  \textit{$E$-vertex} if it contains one of the points $\{q_Nx\}$ with $N_{n-1} < N \ls N_n$ (equivalently if it is colored according to the definition in Section \ref{sec_binary}), and \newline (iii) an \textit{$EF$-interval} or \textit{$EF$-vertice} if it is both $E$- and $F$-interval/vertex. \par   Write $E_n(x)$ for the union of all $E$-intervals at level $n$. Then clearly the set $E$ from the statement of Theorem~\ref{thm4} contains the set
    \begin{equation} \label{smaller_subset}
     E_\infty := \limsup_{n \to \infty}  E_n (x).
    \end{equation}
   In order to show that almost surely\ $E \cap F \neq \emptyset$, it suffices to prove that almost surely $E_\infty \cap F \neq \emptyset$, and in turn, to prove this last statement it suffices to prove that almost surely there exists an infinite path of $F$-vertices that contains $E$-vertices infinitely often (in other words, there is a path containing infinitely many \textit{$EF$-vertices}). Lemma~\ref{34} would then imply that
    \[
    \dimH (E \cap G) \ge \dimH (  E_\infty \cap G) \ge \dimH (G) - \sig - \varepsilon
    \]
    almost surely. Letting $\eps \to 0$, we obtain the desired result.

    \subsection*{Lower bound. Moment computation}
   By Frostman's Lemma \cite[Theorem~8.8]{mattila} there exists a probability measure $\chi$ supported on $F$ and a constant $C>0$ such that   \[ \chi(B(x, r)) \le C r^{s} \qquad \text{ for all } x \in \bR \slash \bZ, \ \ r > 0   \] where $B(x,r)$ denotes the arc of length $2r$ centered at $x.$ Write $\tilde{F} \subseteq F$ for the support of $\chi.$ \par We aim to show that for any vertex $I_0$ such that $I_0 \cap \tilde{F}\neq \emptyset$ we have  \[ E\cap F \cap I_0 \neq \emptyset  \qquad \text{ almost surely.} \] This will imply that almost surely, all $F$-vertices $I_0$ with $I_0 \cap \tilde{F}\neq \emptyset$ have an $EF$-descendant, and thus almost surely there exists an infinite path with infinitely many $EF$-vertices, which gives $E_\infty \cap F \neq \emptyset.$  \par So, let $I_0$ be a vertex with $I_0\cap \tilde{F}\neq \emptyset,$ located at level $n_0$ of the tree.  For $n \gs n_0$ large enough,  set 
    \begin{equation*}
        K_n(F) = \{ 0 \le k < 2^n : I_{n,k} \subseteq I_0 \text{ and }  I_{n,k} \cap F \neq \emptyset \}.
    \end{equation*}
    Since $\chi(I_{n,k}) \le C 2^{-s n}$ for any binary interval $I_{n,k}$,
    \[
        1 \le \sum_{k \in K_n (F)} \chi(I_{n,k}) \le \# K_n (F) \cdot C 2^{-s n}, \qquad \text{which implies} \qquad \# K_n (F) \ge \frac{1}{C} 2^{s n}.
    \] Furthermore, we may assume that $\# K_n(F) \asymp 2^{s n}$, since restricting $F$ to a subset $F' \subset F$ with $\# K_n(F') \asymp 2^{s n}$ can only reduce the probability that $E \cap F \cap I_0 \neq \emptyset.$ 
  
    Next, at level $n$ we will estimate the probability that unusually many/few intervals $I_{n,k}$ with $k \in K_n (F)$ are colored by the points $\{ q_N x\}$. To this end, we again apply the Chebyshev inequality after bounding the second centered moment of the number of points  within 
    $\bigcup_{k \in K_n(F)} I_{n,k}$.   The proof will be split into two main cases, regarding whether the sequence $(q_N)_{N \in \bN}$ is a sequence of integers or an arbitrary real-valued sequence.\\

    \textit{Case 1: Integer-valued $(q_N)_{N \in \bN}$.} 
    We restrict to a subsequence $q_{n_k}$ indexed by the set $\bI$ satisfying the conditions of Theorem~\ref{thm4}. For convenience, we keep using the notation $q_N$ for this subsequence.
    % The estimate for the second moment then reduces to a gcd sum of the type appearing in the statement of Theorem~\ref{thm4}. 
    The number of points of the subsequence arriving at   level $n$ of the tree is equal to $\tilde N_n := \# (\bI \cap [N_n])$.     By the assumptions on the set $\bI$, there exist  infinitely many $n$ such that the number of points of the subsequence placed on the $n$-th level of the tree  satisfies \[ \tilde N_n - \tilde N_{n-1} \ge \kappa \frac{N_n}{f(N_n)}\] for some $\kappa > 0$. Indeed,  otherwise we would get
    \[
        \frac{N_{2M}}{f(N_{2M})} \ll \sum_{M < n \le 2M} \big( \tilde N_n - \tilde N_{n-1} \big) = o\Big( \frac{N_{2M}}{f(N_{2M})} \Big) \qquad \text{as } M \to \infty,  
    \] which is a contradiction. On the other hand, we may also assume that $\tilde N_n-\tilde N_{n-1}\le 2\tilde N_{n-1}$ by discarding extra points (the probability of finding an $EF$-interval can only decrease under such a restriction).\\
    
    Let $g^{-}(x)$ be the function from Lemma~\ref{smoothing}. Then the (weighted) number of points $\{q_N x\}$ in $K_n(F)$ is bounded from below by
    \[
    C_n := C_n (x) = \sum_{k \in K_n (F)} \sum_{\tilde N_{n-1} < N \le \tilde N_n} g^{-} \Big( 2^{n+1} \big( \{ q_N x\} - t_{k,n} \big) \Big) dx, 
    \] where $t_{k,n}$ is the center of $I_{n,k}$. By Poisson summation,
    \[
        C_n = \sum_{k \in K_n (F)} \frac{1}{2^{n+1}} \mathop{  \mathop{\sum\sum}_{ \tilde N_{n-1} < N \le \tilde N_n  } }_{h \in \bZ} \widehat{g}^{-} \Big( \frac{h}{2^{n+1}} \Big) e\big( h q_N x - h t_{k,n} \big).
    \] Then the expectation of $C_n$ satisfies
    \begin{equation} \label{expectation_C_n}
        \bE[C_n] = \# K_n (F) \frac{\widehat{g}^{-}(0) (\tilde N_n - \tilde N_{n-1})}{2^{n+1}}.
    \end{equation}

    The ``bad event'' now corresponds to the absence of $EF$-intervals at level $n$ after coloring. Using the second-moment method, we estimate the probability of the more likely event (note that we only need to show $C_n(x) > 0$)
    \[
    A_n := \Big\{ x \in [0, 1): \ |C_n - \bE[C_n]| > \frac{1}{100} \# K_n (F) \tilde N_n \cdot \frac{1}{2^n} \Big\}.
    \] 

    Next, by the Chebyshev inequality,
    \[
        \lambda (A_n) \ll
        \frac{2^{2n}}{(\# K_n (F) \tilde N_n)^2} \int_0^1 \bigg( \mathop{\mathop{\sum\sum}_{k \in K_n (F)}}_{\tilde N_{n-1} < N \le \tilde N_n} g^{-} \Big( 2^{n+1} \big( \{ q_N x \} - t_{k,n} \big) \Big) - \bE[C_n] \bigg)^2 dx.
    \] Opening the square we get
    \begin{multline} \label{opening_square}
    \lambda (A_n) \ll \frac{2^{2n}}{(\# K_n (F) \tilde N_n)^2}\hspace{-3mm} \mathop{\mathop{\sum\sum}_{k_1, k_2 \in K_n (F)}}_{\tilde N_{n-1} < M, N \le \tilde N_n}\hspace{-2mm} \int_0^1 g^{-} \Big( 2^{n+1} \big( \{ q_M x \} - t_{k_1, n} \big)\!\Big)  
    g^{-} \Big( 2^{n+1} \big(  \{ q_N x \} - t_{k_2,n} \big)\!\Big) dx \\  - \frac{2^{2n}}{(\# K_n (F) \tilde N_n)^2} \bE[C_n]^2 =: \sum_{\tilde N_{n-1} < M, N \le \tilde N_n} C(M, N),
    \end{multline} where
    \begin{multline*}
        C(M, N) = \frac{2^{2n}}{(\# K_n (F) \tilde N_n)^2} \sum_{k_1, k_2 \in K_n (F)} \int_0^1 g^{-} \Big( 2^{n+1} \big( \{ q_M x \} - t_{k_1, n} \big) \!\Big)    
    g^{-} \Big( 2^{n+1} \big(  \{ q_N x \} - t_{k_2,n} \big)\! \Big) dx \\ - \frac{2^{2n}}{(\# K_n (F) \tilde N_n)^2} \frac{\bE[C_n]^2}{(\tilde N_n - \tilde N_{n-1})^2}.
    \end{multline*}
  Next, we evaluate $C(M, N)$ in two different ways: \\

    \textit{First approach. } Here we apply the same trick used in Section~\ref{sec_lower}:
    \[
        \sum_{k_2 \in K_n (F)} g^{-} \Big( 2^{n+1} \big( \{ q_N x\} - t_{k_2, n} \big) \Big) \le 1.
    \] 
   
    This gives  
    \begin{equation} \label{first_bound}
    \begin{split} C(M, N) &\ll \frac{2^{2n}}{(\# K_n (F) \tilde N_n)^2} \sum_{k_1 \in K_n (F)} \int_0^1 g^{-} \Big( 2^{n+1} \big( \{ q_M x\} - t_{k_1, n} \big) \Big) dx \\
    &\ll \frac{2^n}{\# K_n (F) \tilde N_n^2} \ll \tilde N_n^{-\frac{\nu \eps}{2} - 1}, 
    \end{split}\end{equation} where we have used $\#K_n(F) \asymp 2^{s n} = 2^{n - n/\nu + \eps n}$, $2^n \asymp N_n^{\nu} \asymp \tilde N_n^{\nu} f(\tilde N_n)^{\nu} \ll \tilde N_n^{\nu + \nu \eps / 2}$ and simply omitted the contribution of the term with $\bE[C_n]^2$ since it is negative. \\

    \textit{Second approach. } Here we apply Poisson summation:
    \begin{multline} \label{before_split}
    C(M, N) = 
        \frac{2^{2n}}{(\# K_n (F) \tilde N_n)^2} \sum_{k_1, k_2 \in K_n (F)} \frac{1}{2^{2(n+1)}} \sum_{h_1, h_2 \in \bZ} \widehat{g}^{-} \Big( \frac{h_1}{2^{n+1}} \Big) \widehat{g}^{-} \Big( \frac{h_2}{2^{n+1}} \Big)\cdot \\
        \cdot\int_0^1 e\Big( h_1 (q_M x - t_{k_1, n}) - h_2 (q_N x - t_{k_2,n}) \Big) dx - \frac{2^{2n}}{(\# K_n (F) \tilde N_n)^2} \frac{\bE[C_n]^2}{(\tilde N_n - \tilde N_{n-1})^2}. 
    \end{multline} 
The contribution of term $h_1 = h_2 = 0$ is
    \[
        \frac{2^{2n}}{(\# K_n (F) \tilde N_n)^2} \sum_{k_1, k_2 \in K_n (F)} \frac{1}{2^{2(n+1)}} \big| \widehat{g}^{-}(0) \big|^2,
    \] which cancels out with the negative term containing $\bE[C_n]^2$. Using the decay rate of $\widehat{g}^{-}(x)$ we then get
    \begin{multline*}
        C(M, N) \ll \frac{2^{2n}}{(\# K_n (F) \tilde N_n)^2} \sum_{k_1, k_2 \in K_n (F)} \frac{1}{2^{2(n+1)}} \sum_{\substack{|h_1|, |h_2| \le 2^n n^{100} \\ |h_1|+|h_2| \neq 0}} \Big| \widehat{g}^{-} \Big( \frac{h_1}{2^{n+1}} \Big) \Big| \cdot \Big| \widehat{g}^{-} \Big( \frac{h_2}{2^{n+1}} \Big) \Big| \cdot \\
        \bigg| \int_0^1 e\big( h_1 q_M x - h_2 q_N x \big) dx \bigg| + O \Big( \frac{2^{2n}}{q_{\tilde N_{n-1}}} + \frac{1}{2^{100n}} \Big).
    \end{multline*}
    Next, by  orthogonality we find 
    \[
        C(M, N) \ll \frac{1}{\tilde N_n^2} \sum_{\substack{|h_1|, |h_2| \le 2^n n^{100} \\ |h_1|+|h_2| \neq 0}} c(h_1) c(h_2) \Id \Big( h_1 q_M = h_2 q_N \Big) + O \Big( \frac{2^{2n}}{q_{\tilde N_{n-1}}} + \frac{1}{2^{100n}} \Big),
    \] where
    \[
        c(h) := \min\Big( \frac{2^{n+1}}{|h|}, 1 \Big).
    \] We split the main term into three sums
    \[
        C(M, N) \ll T_1 + T_2 + T_3 + O \Big( \frac{2^{2n}}{q_{\tilde N_{n-1}}} + \frac{1}{2^{100n}} \Big),
    \] where $T_1$ is the sum over the range $\max(|h_1|, |h_2|) \le 2^{n+1}$, $T_2$ corresponds to the range $\min(|h_1|, |h_2|) \le 2^{n+1} < \max(|h_1|, |h_2|)$, and $T_3$ corresponds to the range $2^{n+1} < \min(|h_1|, |h_2|)$. Without loss of generality, assume that $M > N$; consequently, $|h_2| > |h_1|$. We have
    \[
        h_1 q_M = h_2 q_N \Longleftrightarrow \frac{h_1}{h_2} = \frac{q_N}{q_M} = \frac{q_N / \gcd(q_M, q_N)}{q_M / \gcd(q_M, q_N)}.
    \] Let $|h_1| =: \ell \frac{q_N}{\gcd(q_M, q_N)}$ and $|h_2| = \ell \frac{q_M}{\gcd(q_M, q_N)}$. In the range corresponding to $T_1$, we then get
    \[
        1 \le \ell \le \frac{2^{n+1} \gcd(q_M, q_N)}{q_M}.
    \] Since $c(h_1) = c(h_2) = 1$ in this range, we get
    \[
        T_1 \le \frac{1}{\tilde N_n^2} \sum_{\ell \le 2^{n+1} \gcd(q_M, q_N)/q_M} 1 \ll \frac{1}{\tilde N_n^2} \frac{2^n \gcd(q_M, q_N)}{\max(q_M, q_N)}.
    \] 
    
    Similarly, for $T_2$ we have $c(h_1) = 1$ and $c(h_2) = \frac{2^{n+1}}{|h_2|}$. Then, this implies (recall $|h_2| \leq 2^n n^{100}$),
    \[
        \ell \in J_{M,N} := \bigg[\frac{2^{n+1} \gcd(q_M, q_N)}{q_M}, \min \Big( \frac{2^{n+1} \gcd(q_M, q_N)}{q_N}, \frac{n^{100} 2^n \gcd(q_M, q_N)}{q_M} \Big) \bigg],  
    \] and thus
    \[
        T_2 \ll \frac{1}{\tilde N_n^2} \sum_{\ell \in J_{M, N}} \frac{2^{n+1} \gcd(q_M, q_N)}{\ell q_M} \ll
        \frac{2^n}{\tilde N_n^2} \frac{\gcd(q_M, q_N)}{\max(q_M, q_N)} \min \Big( \log \frac{\max(q_M, q_N)}{\min(q_M, q_N)}, \log n \Big), 
    \] which is of the desired shape since $\log n \asymp \log \log \tilde N_n$. 
    
    Next, for $T_3$ we get that $|h_1| = \ell \frac{q_N}{\gcd(q_M, q_N)}$ and $|h_2| = \ell \frac{q_M}{\gcd(q_M, q_N)}$ implies that 
       \[\ell \in 
        \tilde J_{M,N} := \bigg[ \frac{2^{n+1} \gcd(q_M, q_N)}{q_N}, \frac{2^n n^{100} \gcd(q_M, q_N)}{q_M} \bigg].
    \] 
    Consequently,
    \[
        T_3 \ll \frac{1}{\tilde N_n^2} \sum_{\ell \in \tilde J_{M,N}} \frac{2^{2(n+1)} \gcd(q_M, q_N)^2}{\ell^2 q_M q_N} \ll \frac{2^n}{\tilde N_n^2} \frac{\gcd(q_M, q_N)}{\max(q_M, q_N)}.
    \]

    \vspace{2ex}
    
    Finally, combining the bounds for $T_1$, $T_2$, and $T_3$, with~\eqref{first_bound}, we find that
    \[
        C(M, N) \ll \frac{2^n}{\tilde N_n^2} \min \bigg[ \frac{\gcd(q_M, q_N)}{\max(q_M, q_N)} \min \Big( \log \frac{\max(q_M, q_N)}{\min(q_M, q_N)}, \log \log \tilde N_n \Big), \ \frac{\tilde N_n^{1-\eps \nu / 2}}{2^n} \bigg].
    \] We can apply the bound $2^n \gg \tilde N_n^{\nu}$ in the denominator of the last fraction. Using~\eqref{opening_square} and the assumptions of Theorem~\ref{thm4} we then conclude that
    \[
        \lambda(A_n) \ll \frac{2^n \tilde N_n^{2-\nu}}{\tilde N_n^2 \psi(\tilde N_n) f^{\nu} (\tilde N_n)} \ll \frac{1}{\psi(\tilde N_n)}
    \] for an arbitrarily slowly growing function $\psi(n)$. Recall that the sequence of levels $n \in \{ n_i : i \ge 1 \}$ can be chosen arbitrarily sparse. In this way, we can always guarantee that
    \[
        \sum_{i=1}^{\infty} \lambda(A_{n_i}) \le \sum_{i=1}^{\infty} \frac{1}{\psi(N_{n_i})} < \infty.
    \] The desired result then follows by the first Borel-Cantelli lemma and Lemma~\ref{34}. \\

    \textit{Case 2: Real-valued almost-lacunary $(q_N)_{N \in \bN}$.} In this case we do not restrict to any subsequence, and the proof goes similarly until~\eqref{before_split}, with $N_n$ in place of $\tilde N_n$. Next, since we do not have orthogonality in this case, we evaluate the last integral in~\eqref{before_split} directly.   We split the sum over $M,N$ into two parts. For the first part, we impose the restriction $|M-N| \le N_n^{\nu \eps} \cdot (\psi(n))^{-1}$. We then apply the first bound~\eqref{first_bound}:
    \[
        \sum_{\substack{N_{n-1} < M,N \le N_n \\ |M-N| \le N_n^{\nu \eps} (\psi(n))^{-1}}} C(M, N) \ll \sum_{\substack{N_{n-1} < M,N \le N_n \\ |M-N| \le N_n^{\nu \eps} (\psi(n))^{-1}}} N_n^{-\nu \eps - 1} \ll \frac{1}{\psi(n)}.    
    \] For the remaining sum, with $|M-N| > N_n^{\nu \eps} (\psi(n))^{-1}$, we show that the quantity $h_1 q_M - h_2 q_N$ cannot be small in absolute value. Indeed, let
    \[
        \Delta_n := \floor{\frac{2^{\eps n}}{\psi(n)}}
    \] and assume again that $q_M > q_N$ and $|h_2| > |h_1|$. By the growth-rate assumption on $(q_N)_{N\in \bN}$, we have for sufficiently large $n$
    \begin{align*}
        \frac{q_M}{q_N} \ge \frac{q_{N+\Delta_n}}{q_N} & \ge \exp\Big(\Delta_n \log (1 + \tfrac{1}{\Phi(n)})\Big)  \\
       & \ge \exp \Big( \frac{1}{2} \cdot \frac{2^{\eps n}}{\psi(n) \Phi(n)} \Big) \ge \exp(n \log 2 + 101 \log n) \ge 2^{n+1} n^{100}, 
    \end{align*} since $\Phi(n) \ll 2^{o(1) \cdot n}$. Then we conclude that $|h_1 q_M - h_2 q_N| > q_N$. 
    
    Hence,
    \[
        \int_0^1 e \Big( h_1 q_M x - h_2 q_N x \Big) dx \ll \frac{1}{q_N}, 
    \] which gives
    \[
        \lambda(A_n) \ll \frac{1}{\psi(n)} + \frac{2^{2n} n^{200}}{q_{N_{n-1}}} + \frac{1}{2^{100 n}} \ll \frac{1}{\psi(n)}.
    \] Then we again obtain the desired result by choosing a sufficiently sparse sequence of levels $n $  and applying the first Borel-Cantelli lemma together with Lemma~\ref{34}. \\

    \subsection*{Upper bound}
    
    Here, we prove the upper bound in Theorem \ref{thm4}, that is, we show that for $G$ Ahlfors regular with dimension $s \in (0,1]$
    and any sequence of real numbers $(q_n)_{n \in \mathbb{N}}$ we have almost surely,
    \begin{equation}\label{intersec_dim_upper} \dimH\bigl(E((q_n),x,\nu)\cap G\bigr)
    \leq \frac{1}{\nu}+\dimH(G)-1,\end{equation}
    provided that the right-hand-side above is non-negative.
    Furthermore, we will show that if $\frac{1}{\nu}+\dimH(G)-1 < 0$, then $E((q_n),x,\nu)\cap G  = \emptyset$ almost surely.\\
    
    The proof  follows directly from the work of Bugeaud and Durand~\cite{BD}, but we provide some details for completeness. Let us recall the definition of Ahlfors regular sets, following~\cite[Definition~2.1]{BD}:
    \begin{defn}
    A compact set $G \subseteq \mathbb{T}$ is called Ahlfors regular with dimension $s \in (0,1]$ if there exists a real number $c > 0$ such that for all $x \in G$ and $r > 0$ one has
    \[
        \frac{r^{s}}{c} \leq \mathcal{H}^{s}(G \cap B(x,r)) \leq cr^{s},
    \] where $B(x,r)$ is the open arc centered at $x$ with length $2r$ and $\mathcal{H}^{s}$ is the Hausdorff measure. 
    \end{defn} We remark that the middle-third Cantor set is Ahlfors regular with dimension $\kappa = \frac{\log 2}{\log 3}$, and that $\mathbb{T}$ is Ahlfors regular with dimension $1$; we refer to \cite{ahlfors_ex} for further details and examples. We also note that, for Ahlfors regular sets $G$, the Hausdorff, box-counting, and packing dimensions measures coincide up to a constant (see~\cite{falconer}).\\
    
    We make use of the result of~\cite[Theorem~3.1]{BD} rephrased below: \\
    
    \begin{lemma}~\cite[Theorem~3.1]{BD}\label{packing_dichotomy}
    Given $\mathbf{r} = (r_n)_{n \in \mathbb{N}}$
    and a sequence of (not necessarily independent) random variables $X_n$, defined on a probability space $(\Omega, \mathcal{A},\mathbb{P}),$ that are all uniformly distributed in $[0,1)$, we define the random set
    \[E(\mathbf{r},\omega) = \{\gamma \in [0,1): \lVert X_n(\omega) - \gamma\rVert \leq r_n \text{ for inf. many } n\gs 1\}.\]
    
    \begin{enumerate}
    \item If $g$ and $h$ are doubling gauge functions, with $\mathcal{P}^g(G) < \infty$ (where $\mathcal{P}$ denotes the packing premeasure)
    and $\sum_{n \in \mathbb{N}} \frac{h(r_n)r_n}{g(r_n)} < \infty$, then
    \begin{equation}\mathbb{P}[\mathcal{H}^h(E(\mathbf{r},\omega) \cap G) = 0] = 1.\end{equation}
    \item If $g$ is a doubling gauge function with 
    $\sum_{n \in \mathbb{N}} \frac{r_n}{g(r_n)} < \infty$, then
    \begin{equation}\mathbb{P}[E(\mathbf{r},\omega) \cap G = \emptyset] = 1.\end{equation}
    \end{enumerate}
    \end{lemma}
    
    We will now prove \eqref{intersec_dim_upper}. Let $G$ be an Ahlfors regular set of dimension $s$, and let $\nu \geq 1$ be arbitrary.
    We fix $\varepsilon > 0$ arbitrary and apply Lemma \ref{packing_dichotomy} with $\Omega = [0,1), \mathcal{A} = \mathcal{B}([0,1)),\mathbb{P} = \lambda$ to the sequence $X_n(x) := q_nx \pmod 1$, with $h(x) := x^{1 / \nu + s + \varepsilon -1}, g(x) := x^{s + \varepsilon/2}$ (which are both trivially doubling) and $r_n = \frac{1}{n^{\nu}}$.
    Since $G$ has Hausdorff dimension $s$ and for Ahlfors regular sets, the packing and Hausdorff measures coincide (up to a constant), we have
    \[P^{g}(G) \ll \mathcal{H}^g(G) = \mathcal{H}^{s+\varepsilon/2}(G) < \infty.\]
    Further, we have
    \[ \sum_{n \in \mathbb{N}} \frac{h(r_n)r_n}{g(r_n)} = \sum_{n \in \mathbb{N}} \frac{1}{n^{1 + \nu \varepsilon/2}} < \infty,\]
    thus Lemma \ref{packing_dichotomy} (1) shows $\mathbb{P}[\mathcal{H}^h(E(\mathbf{r},x) \cap G) = 0] = 1$, and therefore,
    $\dimH(E(\mathbf{r},x) \cap G)) \leq \frac{1}{\nu} + s + \varepsilon -1$ almost surely. With $\varepsilon \to 0$, the result follows.
    In case  $\frac{1}{\nu}+\dimH(G)-1 < 0$, pick $\varepsilon > 0$ such that $\frac{1}{\nu}+\dimH(G)-1- \varepsilon < 0$.
    Choosing once more $g(x) := x^{s + \varepsilon/2}$, and now applying Lemma \ref{packing_dichotomy} (2), shows that almost surely, the intersection $E \cap G$ is empty, as required. \\

    \subsection*{Proof of Corollary~\ref{cor2}~(3)}
    
    Let $P \in \bZ[x]$ be a polynomial of degree $d \ge 1$. We will construct a set $\cM \subseteq P\big([\frac{N}{2},N]\big)$ with
    \begin{equation}\label{size_req}
    \#\cM \ge \frac{N}{(\log N)^{O(1)}}
    \end{equation}
    such that, apart from $O_P(N)$ exceptional pairs $(m,n)\in\cM^2$, one has
    \[
    \frac{\gcd(m,n)}{\max\{m,n\}} \le \frac{1}{N^{d-1}}.
    \]
    It is then immediate that Theorem~\ref{thm4} implies Corollary~\ref{cor2}(3). Indeed, the diagonal contribution $m=n$ (corresponding to $m=k$ in Theorem~\ref{thm4}), together with the contribution of the $O_P(N)$ exceptional pairs, is bounded using the estimate $N^{1-\nu-\eps\nu}$ for the $\min(\cdot,\cdot)$ term in Theorem~\ref{thm4}. For the remaining pairs, we have
    \[
    \sum_{\substack{m,n\in\cM\\ \gcd(m,n)/\max\{m,n\}\le N^{1-d}}}
    \frac{\gcd(m,n)}{\max\{m,n\}}\,
    \log\log N
    \ \ll\ 
    (\#\cM)^2\,N^{1-d}\,\log\log N,
    \]
    which yields the desired bound when $d > \nu+1$.
    
    Using (a variant of) the higher-dimensional sieve developed in~\cite{dh08}, we construct $\cM \subseteq P\big([\frac{N}{2},N]\big)$ satisfying~\eqref{size_req} and the following additional properties:
    \begin{enumerate}
    \item[(i)] For every $m\in\cM$ one has $\tau(m)\ll_P 1$, where $\tau$ denotes the divisor function.
    \item[(ii)] For each $e\in\bN$ let
    \[
    \rho_P(e):=\#\bigl\{a\bmod e:\ P(a)\equiv 0 \pmod e\bigr\}
    \]
    be the root-counting function (which is multiplicative). Then, if $e\mid m$ for some $m\in\cM$, then $\rho_P(e)\ll_P 1$. \\ 
    \end{enumerate}
    
    Next, for $e\in\bN$ define
    \[
    S_e:=\{m\in\cM:\ e\mid m\}.
    \]
    Then
    \begin{equation}\label{S_e_S_e_bound}
    \#\bigl\{(m,n)\in\cM^2:\ \gcd(m,n)>N\bigr\}
    \ \le\ \sum_{e>N} (\#S_e)^2
    \ \le\ \Bigl(\max_{e>N}\#S_e\Bigr)\sum_{e>N}\#S_e
    \end{equation}
    (compare~\cite[(4.8)]{richert}). We first bound $\max_{e>N}\#S_e$. For any $e>N$ we have
    \[
    \#S_e
    =\sum_{\substack{m\in\cM\\ e\mid m}}1
    \le \sum_{\substack{k\le N\\ e\mid P(k)}}1
    =\sum_{\substack{\ell=1\\ P(\ell)\equiv 0\ (\mathrm{mod}\ e)}}^{e}
      \ \sum_{\substack{k\le N\\ k\equiv \ell\ (\mathrm{mod}\ e)}}1
    \le \frac{\rho_P(e)}{e}\,N+\rho_P(e)
    \ll \rho_P(e)\ll_P 1,
    \]
    and hence $\max_{e>N}\#S_e\ll_P 1$. Next,
    \[
    \sum_{e>N}\#S_e
    \le \sum_{m\in\cM}\#\{e>N:\ e\mid m\}
    \le \sum_{m\in\cM}\tau(m)
    \ll_P \#\cM
    \ll \frac{N}{(\log N)^{O(1)}}
    \ll N,
    \]
    using~(i). Substituting these bounds into~\eqref{S_e_S_e_bound} yields
    \[
    \#\bigl\{(m,n)\in\cM^2:\ \gcd(m,n)>N\bigr\}\ \ll_P\ N,
    \]
    as required.\\
    
    It remains to construct $\cM$ with the stated properties. If $P$ has no fixed prime divisor, then $\cM$ can be obtained directly from the sieve framework of~\cite[Chapters~9--10]{dh08}. In the general case, $P$ might have finitely many fixed prime divisors; we make a minor modification of the construction (such as restricting to a suitable residue class modulo the product of these finitely many primes) in order to eliminate their influence.

    \begin{lemma}\label{jdim_sieve}
    Let $h = \prod_{i = 1}^j g_i$ be a polynomial with $g_i \in \mathbb{Z}[x]$ irreducible and $g_i \neq g_{i'}$ for all $i \neq i'$. 
    Let $C$ be such that $\rho_h(p) < p-1$ for all $p > C$.
    Then there exists $\varepsilon = \varepsilon_h > 0$ such that
    \[\#\left\{\frac{2N}{3} \leq n \leq N: \forall C < p \leq N^{\varepsilon}: \gcd(h(n),p) = 1\right\} \gg \frac{N}{(\log N)^{O(1)}},
    \]
    where the implied constant only depends on $h$.
    \end{lemma}
    
    \begin{proof}
    This follows from a standard application of a higher-dimensional sieve (by applying a fundamental-type Lemma and Selberg's sieve, combined with a Mertens-type formula, see Chapters 9 and 10 in~\cite{dh08}).
    Note that the result there was proven under the assumption of having no fixed prime divisor (i.e. $\rho_h(p) < p$ for all $p$). Note that $\rho_h(p) = p$ implies that $p$ divides the discriminant $D_h$. Since $h$ has no multiple roots, $D_h \neq 0$ and thus only finitely many primes divide $D_h$, showing that the condition $\rho_h(p) < p$ can only be violated for the first finitely many primes. Since we only sieve by primes where this condition is satisfied, the estimates done in~\cite{dh08} play over immediately. 
    \end{proof}
    
    Write $P = c \prod_{i = 1}^j f_i^{e_i}, e_i \geq 1$ with $f_i \in \mathbb{Z}[x]$ irreducible and let $H := \prod_{i = 1}^j f_i$. Since $\rho_f (n)$ is multiplicative, it suffices to understand it on prime-powers. Note that $\rho_f(p) \leq \max\{\deg f,p\}$; further note that  if $p \nmid D_f$ (where $D_f$ denotes the discriminant), then for all $r \in \mathbb{N}$,
    $\rho_f(p^r) \leq \rho_f(p)$ (see~\cite{nagel}).\\
    
    Next, we rewrite $H(n)$ as $H(n) = a_d n^d + \ldots + a_1 n + a_0$, $a_d \neq 0$. We claim that there exists $t = t_H \in \mathbb{N}$ so that for all primes $p$ there exists $b_p \pmod {p^t}$ with
    $H(b_p) := c_p \neq 0 \pmod {p^t}$. Indeed, let $b_p \in  (p^{t/2d},p^{t/d})$. Then 
    $H(b_p) = a_d b_p^d + O_{H}(b_p^{d-1})$, which for $t$ sufficiently large satisfies $0 < H(b_p) < p^t$, thus $H(b_p) \neq 0 \pmod {p^t}$.

    We now take a constant $C = C_H := \max\{D_H,\deg H\}$, $P_1 := \prod_{p \leq C} p^t$, and let $c_0 \pmod {P_1}$ be the solution of the congruence system $c_0 \equiv b_p \pmod {p^t}, \forall p \leq C$. Define 
    $h(n) := H(P_1 n + c_0) = \prod_{i=1}^jf_i(P_1 n + c_0)$
    and observe the following properties: 
    
    \begin{itemize}
    \item[(i)] $g_i(n) := f_i(P_1n + c_0)$ is irreducible;
    \item[(ii)] $\forall n \in \mathbb{N}$ and $\forall p \leq C$ one has $p^t \nmid h(n)$;
    \item[(iii)] $\forall p > C$ one has $\rho_h(p) < p$;
    \item[(iv)] $\forall p > C$ and $\forall r \geq 1$ one has $\rho_h(p^r) \ll 1$.
    \end{itemize} 
    
    Indeed, (i) follows from irreducibility of $f_i$, and (ii) follows from the construction of $t,c_0$ above. For (iii), note that $\gcd(p,P_1) = 1$ for $p > C$, thus we have $\rho_h(p) = \rho_H(p) \leq \deg H < p$. Finally, (iv) follows from $\rho_h(p^r) = \rho_H(p^r) \leq \rho_H(p)$ since $p > D_H$. 
    
    Using (i) and (iii), we apply Lemma~\ref{jdim_sieve} with $N_1 := \frac{N}{P_1 +1}$ in place of $N$. This provides us with a set $\mathcal{M}_1 \subset [2N_1/3,N_1]$ with $\#\mathcal{M}_1 \gg \frac{N}{(\log N)^{O(1)}}$ such that for all $m \in \mathcal{M}_1$ and all $1 \leq i \leq j$, we have
    \[p \mid g_i(m) \implies p < C \text{ or } p > N_1^{2\varepsilon} > N^{\varepsilon}.\] In particular,
    using (ii), we obtain that for all $m \in \mathcal{M}_1$, we have
    $\tau(h(m)) \ll_h 1$. We define (element-wise) $\mathcal{M}_2 := P_1\mathcal{M}_1 + c_0 \subseteq [N/2,N]$,
    and see that for $m \in \mathcal{M}_2$, $\tau(H(m)) \ll 1$, and further, setting $\mathcal{M} := P(\mathcal{M}_2)$, we get 
    $\tau(m) \ll 1$ for all $m \in \mathcal{M}$, since $P(n) \mid c H(n)^{\deg P}$. Finally, for all $e | n \in h(\bN)$ one has $\rho_P (e) \ll 1$ by (ii) and (iv).

    %APPENDIX. RANDOM MODEL
    
    %%%%%%%%%%%%%%%%%%%%%%%%%%%%%%%%%%%%%
    
\appendix

\makeatletter
\renewcommand{\@seccntformat}[1]{\csname the#1\endcsname:\enspace}
\makeatother

\section{Random model for Littlewood--Cassels}\label{app:random-model}

    In this appendix we prove Theorem~\ref{thm3}. We note that throughout the proof the implied constants in $\ll,\gg$ do not depend on the constants $C$ or $\varepsilon$ defined in~\eqref{fake_khintchine} and~\eqref{fake_convergence}, respectively. Consequently, at the end $C$ and $\varepsilon$ can be chosen sufficiently large and respectively small.

We will need the following lemma on the size of $\sum_n \psi(n)$ along narrower exponential ranges:

\begin{lemma} \label{about_exp_gap}
Assume \eqref{fake_khintchine}. Then for every $c > 1$, one has
\begin{equation}\label{fake_khin_c}
    \limsup_{N \to \infty} \sum_{2^N \leq n \leq 2^{(c^N)}} \psi(n) \gg_c C.
\end{equation}
\end{lemma}

\begin{proof}
Take $m \in \bN$ such that $c^m > 4$. Then by the pigeonhole principle, 
\[
    \limsup_{R \to \infty} \sum_{2^R \leq n \leq 2^{(c^R)}} \psi(n) > \frac{1}{m}
    \limsup_{R \to \infty} \sum_{2^R \leq n \leq 2^{(4^R)}} \psi(n).
\] Choosing $2^R \leq N \leq 2^{R+1}$ with $N$ as in~\eqref{fake_khintchine}, it follows immediately that
\[
    \limsup_{R \to \infty} \sum_{2^R \leq n \leq 2^{(4^R)}} \psi(n) > C,
\] which proves the claim.

\end{proof}

We will also need a few standard preliminaries on the distribution of integers in diophantine Bohr sets:

\begin{lemma}\cite[Lemma 6.2]{BHV20}\label{62}
    Let $\alpha$ be irrational with convergent denominators $(q_k)_{k \in \bN}, N \in \bN, \varepsilon_N > 0$ such that $N^{-1} < 2\varepsilon_N < \lVert q_2\alpha\rVert$. For $q_K \leq N < q_{K+1}$, we define 
    \[
        M := \max \Big( \varepsilon_N N, \ \min \Big( \varepsilon_N q_{K+1},\frac{N}{2q_K} \Big) \Big).
    \] Then 
    \[
        \lfloor M \rfloor \leq \#\{n \leq N: \lVert n\alpha \rVert < \varepsilon_N\} \leq 32M.
    \] 
\end{lemma}

\vspace{1ex}

\begin{lemma}\cite[Lemma 6.3]{BHV20}\label{63}
For any $\varepsilon_N> 0, N \in \mathbb{N}, \gamma \in \mathbb{R}$ we have
\[
    \#\{n \leq N: \lVert n\alpha - \gamma\rVert < \varepsilon_N\} \leq \#\{n \leq N: \lVert n\alpha \rVert < 2\varepsilon_N\} +1,
\] and if $\#\{n \leq \mfrac{N}{2}: \lVert n\alpha - \gamma\rVert < \mfrac{\varepsilon_N}{2} \} \geq 1$, then 
\[
    \# \big\{n \leq N: \lVert n\alpha - \gamma\rVert < \varepsilon_N \big\} \geq \# \big\{n \leq N: \lVert n\alpha \rVert < \frac{\varepsilon_N}{2} \big\} +1.
\]
\end{lemma} 

\vspace{1ex}

\begin{cor} \label{bohr_estim_cor}
There exist absolute $c_2 > c_1 > 0$ such that for
any $\alpha \in \Bad$ and $b \ge 300$ there exist $K = K(\alpha, b) > 0$ such that for any $\varepsilon_N \geq \mfrac{K}{N}$ the following inequalities hold:
\[
    c_1 N\varepsilon_N \le \# \Big\{n \leq N: \frac{\eps_N}{b} \le \lVert n\alpha - \gamma\rVert < \varepsilon_N \Big\} \le c_2 N\varepsilon_N.
\]
\end{cor}

\begin{proof}
First, note that for badly approximable $\alpha$ one has $\mfrac{q_{k+1}}{q_k} \asymp_{\alpha} 1$. Thus, applying the first part of Lemma~\ref{63} and then Lemma~\ref{62} (with $M=2\eps_N N$ and with $K(\alpha,b)$ taken sufficiently large) we find
\begin{multline*}
    \# \Big\{n \leq N: \frac{\eps_N}{b} \le \lVert n\alpha - \gamma\rVert < \varepsilon_N \Big\} \le \# \Big\{n \leq N: \lVert n\alpha - \gamma\rVert < \varepsilon_N \Big\} \le \\
    \# \Big\{n \leq N: \lVert n\alpha \rVert < 2\varepsilon_N \Big\} + 1 \le 64 \eps_N N + 1.
\end{multline*} 
For the lower bound, note that since the dispersion of $(n\alpha)_{n \leq N} \pmod 1$ is bounded by $\mfrac{K(\alpha,b)}{N}$ when $K$ is sufficiently large, we see that $\#\{n \leq \mfrac{N}{2}: \lVert n\alpha - \gamma\rVert < \mfrac{\varepsilon_N}{2} \} \geq 1$. Thus we may apply both parts of Lemma~\ref{63} and then Lemma~\ref{62} (with $M = \mfrac{\eps_N N}{2}$ and $M = \mfrac{2\eps_N N}{b}$ respectively) to find
\begin{multline*}
    \# \Big\{n \leq N: \frac{\eps_N}{b} \le \lVert n\alpha - \gamma\rVert < \varepsilon_N \Big\} = \# \Big\{n \leq N: \lVert n\alpha - \gamma\rVert < \varepsilon_N \Big\} - \\
    \# \Big\{n \leq N:  \lVert n\alpha - \gamma\rVert < \frac{\eps_N}{b} \Big\} \ge \# \Big\{ n \le N : \norm{n\alpha} < \frac{\eps_N}{2} \Big\} + 1 - \\
    \# \Big\{ n \le N: \norm{n\alpha} < \frac{2\eps_N}{b} \Big\} - 1 \ge \frac{\eps_N N}{2} - \frac{64 \eps_N N}{b} > \frac{\eps_N N}{4},
\end{multline*} which completes the proof.
\end{proof}

\vspace{1ex}

\subsection*{Covering case} 

Let $\psi:\N\to[0,\infty)$ be a monotonically decreasing function satisfying~\eqref{fake_khintchine}. We first observe that, without loss of generality, it suffices to prove the statement for functions $\psi$ satisfying the following additional properties:

\begin{enumerate}
\item[(i)] $\psi$ is constant on $b$-adic blocks $(b^k,b^{k+1}]$ and only takes values of the form $\mfrac{1}{2^{\ell}}$, i.e. for every $n\in[b^k,b^{k+1})$ one has $\psi(n)=\mfrac{1}{2^{\ell(k)}}$;
\item[(ii)] one has $\psi(n)\le \mfrac{C}{n}$ for all $n\in\bN$.
\end{enumerate}

The first part follows immediately from the Cauchy condensation test, at the cost of a factor of at most $2b$. 
The second part follows from the first one upon noticing that $\min \{ \psi(n), \mfrac{C}{n}\}$ satisfies~\eqref{fake_khintchine}.

Clearly, \eqref{inhomo_meas1} is equivalent to requiring that $[0,1)$ is covered infinitely often, almost surely, by the random arcs
\[
    r_n=\Bigl(X_n-\mfrac{\ell_n}{2},\,X_n+\mfrac{\ell_n}{2}\Bigr),
\] where
\begin{equation} \label{our_lengths}
    \l_n := \frac{\psi(n)}{\|n\alp-\gamma\|}.
\end{equation}
By Shepp's result~\cite{shepp}, it is sufficient that
\begin{equation}\label{shepp_cond}
    \sum_{n} \frac{1}{n^2} \exp \Big(\sum_{k \leq n} \l_k \Big) = \infty,
\end{equation}
provided that $(\l_n)_{n\in\mathbb{N}}$ is non-increasing. Since this is clearly not the case for the lengths $\ell_n$ defined in~\eqref{our_lengths}, we reorder them into a non-increasing sequence. Note that we may discard some of the lengths $\ell_n$ from the sequence: a covering by a subsequence clearly implies a covering by the original family.

For $\ell\in\bN$, define
\[
    S_{\ell} := \left\{ b^{2\ell} < n \le b^{b^{\ell}}: \frac{1}{b^{\ell}} \le \frac{\psi(n)}{\lVert n\alpha - \gamma\rVert} < \frac{1}{b^{\ell-1}} \right\}.
\] Note that $S_i\cap S_j=\emptyset$ for $i\neq j$. We only retain those lengths $\ell_k$ for which $k$ belongs to one of the sets $S_{\ell}$, $\ell\in\bN$. We will show that, for infinitely many $L$, one has
\begin{gather} \label{collection_size}
    T_1 := \sum_{\ell \le L} \# S_{\ell} \ll_{\alpha} b^{2L}, \\
\label{cond_on_sl}
    T_2 := \sum_{\ell \leq L} \frac{\#S_{\ell}}{b^\ell} \gg C L.
\end{gather}
This will imply Theorem~\ref{thm3} directly. Indeed, given such an $L$, pick the smallest $n$ with $n\ge T_1$. Then $\log n \le 3L\log b$ (for $n$ sufficiently large), and hence
\[
    \frac{1}{n^2} \exp \Big( \sum_{k \le n} \ell_k \Big)
    \ge \frac{1}{n^2} \exp \Big( \sum_{\ell \le L} \frac{\# S_{\ell}}{b^{\ell}} \Big)
    \ge \frac{1}{n^2} \exp \Big( C \frac{\log n}{3 \log b} \Big)
    \ge 1,
\]
since $C>0$ can be chosen sufficiently large in terms of $b$. The result then follows from Shepp's theorem.

Thus, we need to show~\eqref{collection_size} and~\eqref{cond_on_sl}. For $T_1$, we have (recall (i))
\begin{align*}
    T_1 & \le \sum_{\ell \le L} \# \Big\{ b^{2\ell} <  n \le b^{b^{\ell}}: \frac{1}{b^{\ell}} \le \frac{\psi(n)}{\norm{\alpha n - \gamma}} < \frac{1}{b^{\ell-1}} \Big\}  \\
   &  \quad\ls \sum_{\ell \le L} \sum_{2\ell \le j \le b^{\ell}} \# \Big\{ b^j < n \le b^{j+1}: \norm{\alpha n - \gamma} \le b^{\ell} \psi(b^j) \Big\}.
\end{align*} 

Note that by (ii), $b^{\ell}\psi(b^j)\le \mfrac{C b^{\ell}}{b^j}<1$ in the range $2\ell\le j$. Let $K(\alpha, b)$ be as in Corollary~\ref{bohr_estim_cor}. We split the indices $j$ into two classes: those for which $b^{\ell}\psi(b^j)\ge \mfrac{K(\alpha, b)}{b^j}$, where we will apply Corollary~\ref{bohr_estim_cor}; and those for which $b^{\ell}\psi(b^j)< \mfrac{K(\alpha, b)}{b^j}$, where we can bound the number of $n$ in the set
\[
    \Big\{ b^j < n \le b^{j+1}: \norm{\alpha n - \gamma} \le b^{\ell} \psi(b^j) \Big\}
\]
by $O_{\alpha}(1)$, since $\alpha\in\Bad$:
\begin{multline*}
    T_1 \ll \sum_{\ell \le L} \sum_{2\ell \le j \le b^{\ell}} b^j \cdot b^{\ell} \psi(b^j) \Id \Big( b^{\ell} \psi(b^j) \ge \frac{K(\alpha, b)}{b^j} \Big) + \sum_{\ell \le L} \sum_{2\ell \le j \le b^{\ell}} \Id \Big( b^{\ell} \psi(b^j) < \frac{K(\alpha, b)}{b^j} \Big)  \\
    \ll  \sum_{\ell \le L} \sum_{2\ell \le j \le b^{\ell}} b^j \cdot b^{\ell} \frac{K(\alpha, b)}{b^j} + b^L \ll_{\alpha} b^{2L}. 
\end{multline*}

Next, we evaluate $T_2$ from below using Corollary~\ref{bohr_estim_cor}:
\begin{multline*}
    T_2 \ge \sum_{L/2 < \ell \le L} \frac{1}{b^{\ell}} \sum_{2\ell \le j < b^{\ell/2}} \# \Big\{ b^j < n \le b^{j+1}: b^{\ell-1} \psi(b^j) < \norm{n\alpha-\gamma} \le b^{\ell} \psi(b^j) \Big\} \gg \\
    \sum_{L/2 < \ell \le L} \frac{1}{b^{\ell}} \sum_{2\ell \le j < b^{\ell/2}} b^{j+\ell} \psi(b^j) \Id \Big( b^{\ell} \psi(b^j) \ge \frac{K(\alpha, b)}{b^j} \Big) = \\
    \sum_{L/2 < \ell \le L} \frac{1}{b^{\ell}} \sum_{2\ell \le j < b^{\ell/2}} b^{j+\ell} \psi(b^j) - \sum_{L/2 < \ell \le L} \frac{1}{b^{\ell}} \sum_{2\ell \le j < b^{\ell/2}} b^{j+\ell} \psi(b^j) \Id \Big( b^{\ell} \psi(b^j) < \frac{K(\alpha, b)}{b^j} \Big) \ge \\
    \sum_{L/2 < \ell \le L} \sum_{b^{2\ell} < n < b^{b^{\ell/2}}} \psi(n) - \sum_{L/2 < \ell \le L} \frac{1}{b^{\ell}}\sum_{2\ell \le j < b^{\ell/2}} K(\alpha,\beta)\gg L \sum_{b^{2L} < n < b^{b^{L/4}}} \psi(n) - \sum_{L/2 < \ell \le L} \frac{K(\alpha, b)}{b^{\ell/2}} \gg \\
    L \sum_{b^{2L} < n < b^{b^{L/4}}} \psi(n) - O_{\alpha}(1) \gg CL
\end{multline*} for infinitely many $L$ since by Lemma~\ref{about_exp_gap}, 
\[
    \limsup_{L \to \infty} \sum_{b^{2L} < n < b^{b^{L/4}}} \psi(n) \gg C.
\]

\vspace{1ex}

\subsection*{Non-covering case}

Here we may again assume that $\psi(n)$ is constant on the intervals $(b^k,b^{k+1}]$, by the same reasoning as before. Moreover, we may assume that
\[
    \psi(n) \ge \mfrac{1}{n(\log n)^{100}},
\] since establishing non-covering for $\tilde{\psi}(n):=\max\!\left\{\psi(n),\,\mfrac{1}{n(\log n)^{100}}\right\}$ implies non-covering for $\psi(n)$, and since
\[
    \sum_{n\in\bN} \frac{1}{n(\log n)^{100}} < \infty,
\] this does not affect the condition~\eqref{fake_convergence}.

In this case, we partition the set of lengths $\ell_n = \mfrac{\psi(n)}{\norm{n\alpha}}$ into the subsets
\[
    S_{\ell} := \Big\{ n \in \bN: \frac{1}{b^{\ell}} \le \frac{\psi(n)}{\norm{n\alpha}} < \frac{1}{b^{\ell-1}} \Big\}.
\] Note that, since $\alpha \in \Bad$, we have $\norm{n\alpha} > \mfrac{c}{n}$ for some $c = c(\alpha, b) \in (0, 1)$ and all $n \in \bN$. Therefore,
\[
    \frac{c}{n} < \norm{n\alpha} \le \psi(n) b^{\ell} \le \frac{b^{\ell}}{n\log n},
\] implying that $c \log n < b^{\ell}$. Hence, we can write 
\[
    S_{\ell} = \Big\{ n \le b^{c_1 b^{\ell}}: \frac{1}{b^{\ell}} \le \frac{\psi(n)}{\norm{n\alpha}} < \frac{1}{b^{\ell-1}} \Big\},
\] where $c_1 = c_1 (\alpha) > 0$. In this case, our goal is to show that for all $L \in \bN$ one has
\begin{gather*}
    T_1 := \sum_{\ell \le L} \# S_{\ell} \gg \frac{b^L}{L^{100}}, \\
    T_2 := \sum_{\ell \le L} \frac{\# S_{\ell}}{b^{\ell-1}} \ll \eps L.
\end{gather*} Then, for a fixed $n$, choosing $L$ such that $T_1 (L-1) < n \le T_1 (L)$ (and, thus, $\log n \ge \mfrac{L}{2} \log b$) we will obtain
\[
    \sum_{n=1}^{\infty} \frac{1}{n^2} \exp \Big( \sum_{k \le n} \ell_k \Big) \ll \sum_{n=1}^{\infty} \frac{1}{n^2} \exp \Big( \eps \frac{2\log n}{\log b} \Big) < \infty,
\] which will give the desired result by Shepp's theorem.

First, we have
\[
    T_1 \gg \sum_{\ell \le L} \sum_{1 \le j < c_1 b^{\ell}} \# \Big\{ b^j < n \le b^{j+1} : b^{\ell-1} \psi(b^j) < \norm{n\alpha} \le b^{\ell} \psi(b^j) \Big\}.
\] Note that when $j$ is in the range
\[
    \frac{K(\alpha, b)}{b^j} \le \frac{b^{\ell}}{j^{100} b^j} \le b^{\ell} \psi(b^j) \le \frac{b^{\ell}}{j b^j} \le 1,
\] we may apply Corollary~\ref{bohr_estim_cor}. In particular, for the lower bound, we can narrow this range to $\ell - \log_b \ell < j \le b^{\ell/101}$. 

This gives
\[
    T_1 \gg \sum_{\ell \le L} \sum_{\ell - \log_b \ell < j \le b^{\ell/101}} b^{j+\ell} \psi(b^j) \gg \sum_{\ell \le L} b^{\ell} \sum_{\ell - \log_b \ell < j \le b^{\ell/101}} \frac{1}{j^{100}} \gg \frac{b^L}{L^{100}}. 
\]

Finally, combining Corollary~\ref{bohr_estim_cor} with the trivial bound we find
\begin{multline*}
    T_2 \le \sum_{\ell \le L} \frac{1}{b^{\ell-1}} \sum_{1 \le j \le c_1 b^{\ell}} \# \Big\{ b^j < n \le b^{j+1} : \norm{n\alpha} \le b^{\ell} \psi(b^j) \Big\} \ll \\
    \sum_{\ell \le L} \frac{1}{b^{\ell}} \sum_{1 \le j \le \ell - \log \ell} b^j + \sum_{\ell \le L} \frac{1}{b^{\ell}} \sum_{\ell - \log \ell < j \le c_1 b^{\ell}} b^{j+\ell} \psi(b^j) \ll O(1) + \sum_{\ell \le L} \sum_{b^{\ell/2} \le n \le b^{c_1 b^{\ell}}} \psi(n) \ll_b \eps L,
\end{multline*} where the last step follows from~\eqref{fake_khintchine} and the fact that changing the base from $2$ to $b$ in the summation range produces at most a constant factor depending on $b$ (by analogy with Lemma~\ref{about_exp_gap}). This completes the proof.

    \bibliographystyle{abbrv}
    \bibliography{Bib}

    \end{document}